\documentclass[12pt,reqno]{amsart}
\usepackage{amsfonts,amsmath,amssymb,amsthm}
\usepackage{mathrsfs,array,enumitem,a4wide}
\usepackage{tikz}
\usepackage[english]{babel}
\usepackage[colorlinks=true,linkcolor=blue,citecolor=magenta,
  urlcolor=cyan,backref=page]{hyperref}
\hypersetup{
  bookmarksdepth=3,
  pdftitle={Quantum super Schur--Weyl duality and character formulas for mirabolic Hecke algebras},
  pdfauthor={Naihuan Jing and Ning Liu},
  pdfkeywords={mirabolic Hecke algebras, quantum super Schur--Weyl duality, character formulas}
}
\allowdisplaybreaks[4]
\theoremstyle{plain}
\newtheorem{theorem}{Theorem}[section]
\newtheorem{mainthm}{Theorem}
\newtheorem{proposition}[theorem]{Proposition}
\newtheorem{corollary}[theorem]{Corollary}
\newtheorem{lemma}[theorem]{Lemma}
\newtheorem{remark}[theorem]{Remark}
\theoremstyle{definition}
\newtheorem{definition}[theorem]{Definition}
\newtheorem{example}[theorem]{Example}
\theoremstyle{plain}

\numberwithin{equation}{section}
\newcommand{\wt}{\mathrm{wt}}

\begin{document}

\title[Quantum super Schur--Weyl duality]{Quantum super Schur--Weyl duality and character formulas for mirabolic Hecke algebras}
\author{Naihuan Jing}
\address{Department of Mathematics, North Carolina State University, Raleigh, NC 27695, USA}
\email{jing@ncsu.edu}
\author{Ning Liu}
\address{School of Mathematics, Hefei University of Technology, Hefei, Anhui 230009, China}
\email{mathliu123@outlook.com}

\subjclass[2020]{Primary 20C08; Secondary 05E05, 05E10, 17B37, 20G43}

\keywords{mirabolic Hecke algebras, quantum super Schur--Weyl duality,
quantum Schur superalgebras, Frobenius formula, Roichman rule,
\(q\)-rook monoid algebra, tensor-space annihilators}

\begin{abstract}
We establish a quantum super Schur--Weyl duality for the mirabolic Hecke algebra \(H_n^{\mathrm{mir}}(q)\) with \(q\) indeterminate and determine the corresponding bimodule decomposition.
The centralizer of the mirabolic action is shown as a direct sum of homogeneous
quantum Schur superalgebras.  We describe the tensor-space annihilator and
the resulting faithful quotient.  When nonzero, the annihilator is generated
by an explicit spectral idempotent attached to the smallest rectangle
excluded by the hook condition.  From the quantum super Schur--Weyl duality,
we derive a super Frobenius formula for mirabolic Hecke algebras.  This formula
then yields a Murnaghan--Nakayama rule and Regev-type formulas for the
irreducible characters.  We also construct a super mirabolic RSK bijection from words in a super alphabet with an additional even letter to hook semistandard insertion tableaux together with recording pairs that distinguish the positions of the additional letter. This correspondence yields a Roichman formula for irreducible characters of \(H_n^{\mathrm{mir}}(q)\).  Finally, we establish a parameter-inversion isomorphism between \(H_n^{\mathrm{mir}}(q^{-1})\) and the \(q\)-rook monoid algebra \(R_n(q)\) and use it to transport the duality, annihilators, and character formulas in both directions.
The transported tensor decomposition provides a representation-theoretic interpretation of
hook and two-row character sums for \(R_n(q)\).
\end{abstract}

\maketitle
\tableofcontents

\section{Introduction}
\label{sec:introduction}

\subsection{Background and Motivation}
\label{subsec:introduction-background}

Hecke algebra characters admit symmetric-function and tableau descriptions,
as in Ram's Frobenius formula \cite{Ram91} and Roichman's
character rule \cite{Roi97}.  For the \(q\)-rook monoid algebra,
Halverson \cite{Hal04} established tensor-space duality, and
Dieng--Halverson--Poladian \cite{DHP03} obtained Frobenius,
Murnaghan--Nakayama, and Roichman formulas.  In the super setting,
Berele--Regev's hook tableaux and hook Schur functions describe covariant
tensor representations \cite{BR87}; Moon and Mitsuhashi established
quantum super Schur--Weyl duality \cite{M03,Mit06}.
Mitsuhashi proved the type \(A\) super Frobenius formula \cite{Mit10},
whereas Zhao obtained quantum Regev identities \cite{Z19} and later
gave a superinsertion proof of Mitsuhashi's formula \cite[Remark~25]{Z24}.

For the mirabolic Hecke algebra, Wan \cite{Wan26} established ordinary
Schur--Weyl duality and Frobenius and Murnaghan--Nakayama formulas.
She also introduced the associated mirabolic \(q\)-Schur algebra
\cite[Remark~3.3]{Wan26}.  These results provide the starting point
for a quantum super version of mirabolic tensor and character theory.

We develop this theory on the tensor powers of a natural quantum
superalgebra module with an additional even trivial summand.
Our results describe the commuting images, their bimodule decomposition,
and the tensor-space annihilator.  Weighted traces and signed insertion
then yield character formulas, while parameter inversion relates these
results to the \(q\)-rook algebra.

\subsection{Main Results and Methods}
\label{subsec:introduction-main-results}

Let \(v\) be an indeterminate, put \(q=v^2\), and set
\(\mathbb F=\mathbb C(q)\subset\mathbb K=\mathbb C(v)\).
Fix \(n\ge1\) and \(m,s\ge0\) with \(m+s\ge1\).
Hecke algebras and their irreducible characters are defined over
\(\mathbb F\); a subscript \(\mathbb K\) denotes scalar extension.
All traces are ordinary traces, and centralizers are taken in ordinary
associative endomorphism algebras.

\subsubsection{Schur--Weyl duality and Schur superalgebras}
\label{subsubsec:intro-schur-weyl}

Let \(V\) be the natural \(U_{v^{-1}}(\mathfrak{gl}(m|s))\)-module
with even and odd dimensions \(m,s\), and put
\(\widetilde V=\mathbb Kv_\star\oplus V\), with
\(\mathbb Kv_\star\) even and trivial.
Let \(H_n^{\mathrm{mir}}(q)\) be the mirabolic Hecke algebra.
On \(\widetilde V^{\otimes n}\), use its action from
Proposition~\ref{prop:super-mirabolic-representation} and the quantum
action of \eqref{eq:conjugated-quantum-actions}.
Put \(H(m,s)=\{\lambda:\lambda_{m+1}\le s\}\), padding partitions
by zeros, and write \(\lambda'\) for the conjugate partition.
For \(\lambda\in H(m,s)\), let \(L_{m|s}(\lambda)\) be the
covariant simple quantum module labelled by \(\lambda\).  For \(\lambda\vdash a\), \(0\le a\le n\),
let \(N^{(n)}_{(\lambda,a)}\) be the rank-\(a\) mirabolic simple module
of Subsection~\ref{subsec:irreducible-characters}; for \(a\ge1\), its
associated Specht module has shape \(\lambda'\)
(Corollary~\ref{cor:mirabolic-specht-label}).

\begin{mainthm}[Schur--Weyl duality]
\label{thm:intro-super-mirabolic-schur-weyl}
The images of \(H^{\mathrm{mir}}_{n,\mathbb K}(q)\) and
\(U_{v^{-1}}(\mathfrak{gl}(m|s))\) on \(\widetilde V^{\otimes n}\)
are mutual full centralizers, with bimodule decomposition
\[
 \widetilde V^{\otimes n}\cong
 \bigoplus_{a=0}^n
 \bigoplus_{\substack{\lambda\vdash a\\\lambda\in H(m,s)}}
 L_{m|s}(\lambda)\otimes_{\mathbb K}
 N^{(n)}_{(\lambda,a),\mathbb K}.
\]
\end{mainthm}

Theorem~\ref{thm:intro-super-mirabolic-schur-weyl} extends Wan's
ordinary mirabolic duality \cite[Theorem~3.2]{Wan26} to the super setting.
Its proof groups tensors by the number of nonstar factors and applies
quantum super Schur--Weyl duality \cite{M03,Mit06} within each rank
(Theorem~\ref{thm:super-mirabolic-schur-weyl}).

Let \(\mathscr S_q(m|s,a)\) and \(\mathscr S_q(m|s;\le n)\)
be the quantum image algebras on \(V^{\otimes a}\) and
\(\widetilde V^{\otimes n}\), respectively, with
\(\mathscr S_q(m|s,0)=\mathbb K\).
Proposition~\ref{prop:bounded-degree-schur-decomposition} identifies the
bounded-degree quantum Schur superalgebra as
\[
 \mathscr S_q(m|s;\le n)
 \cong\bigoplus_{a=0}^n\mathscr S_q(m|s,a).
\]
The summands are the homogeneous quantum Schur superalgebras of
Du--Rui \cite[Definition~5.6 and Proposition~8.3]{DR11}, as explained in
Subsection~\ref{subsec:bounded-degree-quantum-schur-superalgebra}.
This extends Wan's mirabolic \(q\)-Schur algebra
\cite[Remark~3.3]{Wan26} and parallels the extended Schur algebra
for the classical rook monoid \cite[Section~3]{AM24}.

\subsubsection{Tensor-space annihilators and faithful quotients}
\label{subsubsec:intro-annihilators}

Let \(W_{\mathbb F,n}\) be the \(\mathbb F\)-span of the ordered
tensor basis of \(\widetilde V^{\otimes n}\), with mirabolic action
\(\Phi_{m|s,n}:H_n^{\mathrm{mir}}(q)\to
\operatorname{End}_{\mathbb F}(W_{\mathbb F,n})\).
Set \(\rho=((s+1)^{m+1})\) and \(r_0=|\rho|=(m+1)(s+1)\);
thus \(\rho\) is the smallest rectangle outside \(H(m,s)\).
For a standard tableau \(T\) of shape \(\rho\), let
\(g_T\in H_{r_0}^{\mathrm{mir}}(q)\) be the spectral idempotent
in \eqref{eq:mirabolic-rectangle-idempotent} selecting its seminormal
basis vector.  For \(n\ge r_0\), let
\(\jmath_{r_0,n}:H_{r_0}^{\mathrm{mir}}(q)\hookrightarrow H_n^{\mathrm{mir}}(q)\)
be the embedding preserving generator indices.

\begin{mainthm}[Mirabolic tensor-space annihilator]
\label{thm:intro-mirabolic-annihilator}
The representation \(\Phi_{m|s,n}\) is faithful for \(n<r_0\).
For \(n\ge r_0\) and every standard tableau \(T\) of shape \(\rho\),
\[
 \ker\Phi_{m|s,n}
 =H_n^{\mathrm{mir}}(q)\,\jmath_{r_0,n}(g_T)\,H_n^{\mathrm{mir}}(q).
\]
\end{mainthm}

Theorem~\ref{thm:intro-super-mirabolic-schur-weyl} identifies the missing
simple modules by the condition \(\rho\subseteq\lambda\), and
for \(n\ge r_0\), \(\jmath_{r_0,n}(g_T)\) acts as a nonzero operator on exactly these
modules (Theorem~\ref{thm:super-mirabolic-annihilator}).
The faithful quotient and \(\mathscr S_q(m|s;\le n)\) are therefore
mutual full centralizers after extension to \(\mathbb K\)
(Corollary~\ref{cor:mirabolic-quotient-double-centralizer}).

\subsubsection{Character formulas and RSK}
\label{subsubsec:intro-characters}

Let \(x=(x_1,\ldots,x_m)\) and \(y=(y_1,\ldots,y_s)\) record
weights of the even and odd basis vectors, and let \(s_\lambda(x|y)\)
be the hook Schur function, which is the character of
\(L_{m|s}(\lambda)\) for \(\lambda\in H(m,s)\).
Write \(\chi^{(n)}_\lambda\) for the character of
\(N^{(n)}_{(\lambda,|\lambda|)}\).
For a composition \(\mu=(\mu_1,\ldots,\mu_\ell)\) of size
\(|\mu|=k\le n\), let \(\widehat T^{(n)}_\mu\) be the standard
mirabolic element defined in \eqref{eq:mirabolic-standard-element}.
For \(d\ge1\), let \(\mathsf q_d(x|y;q)\) be the polynomial defined by
\eqref{eq:super-mirabolic-one-cycle-series}, and put
\(\mathsf q_\mu=\prod_i\mathsf q_{\mu_i}\), with
\(\mathsf q_\varnothing=1\) for the empty composition of size zero.

\begin{mainthm}[Super mirabolic Frobenius formula]
\label{thm:intro-super-mirabolic-frobenius}
For every composition \(\mu\) of \(k\le n\), the hook Schur expansion is
\[
 \mathsf q_\mu(x|y;q)=
 \sum_{a=0}^k
 \sum_{\substack{\lambda\vdash a\\\lambda\in H(m,s)}}
 \chi^{(n)}_\lambda(\widehat T^{(n)}_\mu)s_\lambda(x|y).
\]
\end{mainthm}

The proof compares weighted traces on tensor words with traces from
Theorem~\ref{thm:intro-super-mirabolic-schur-weyl}
(Theorem~\ref{thm:super-mirabolic-frobenius}); \(s=0\) recovers
Wan's Frobenius formula \cite[Theorem~4.4]{Wan26}.
Symmetric-function identities give a self-contained derivation of the mirabolic
Murnaghan--Nakayama rule (Theorem~\ref{thm:mirabolic-murnaghan-nakayama-rule});
Wan derived this rule from the \(q\)-rook rule \cite[Theorem~5.4]{Wan26}.
Setting \(x_i=y_j=1\) gives the Regev-type character sum of
Theorem~\ref{thm:super-mirabolic-theorem-B}, whose full-rank restriction
recovers Zhao's quantum Regev identity \cite[Theorem~A]{Z19}
after matching the Hecke normalizations.

For \(N\ge0\), use the ordered alphabet
\(\mathcal A_{m|s}=\{\star\prec x_1\prec\cdots\prec x_m
\prec y_1\prec\cdots\prec y_s\}\), with even letters \(\star,x_i\)
and odd letters \(y_j\).
Let \(\mathrm{HT}_{m|s}(\lambda)\) consist of tableaux of shape
\(\lambda\) on the nonstar letters, weakly increasing in rows and
columns, with even letters strict down columns and odd letters strict
along rows.  Let \(\mathrm{RT}_N(\lambda)\) consist of pairs
\(Q=(Q^{(0)},Q^{(1)})\), where \(Q^{(1)}\) is a standard tableau
of shape \(\lambda\) with entries in \(\{1,\ldots,N\}\), and
\(Q^{(0)}\) lists the remaining entries in increasing order.
Word and insertion-tableau weights are products of their nonstar letters,
with star weight \(1\).
For a composition \(\mu\) of \(N\), let \(\wt_\mu(Q)\) be the
recording-tableau weight defined in \eqref{eq:def-super-roichman-weight}.

\begin{mainthm}[Super mirabolic RSK correspondence and Roichman formula]
\label{thm:intro-super-mirabolic-rsk}
Signed column insertion gives a weight-preserving bijection
\[
 \mathrm{RSK}:\mathcal A_{m|s}^{N}\xrightarrow{\sim}
 \coprod_{\substack{\lambda\text{ a partition}\\|\lambda|\le N}}
 \mathrm{HT}_{m|s}(\lambda)\times\mathrm{RT}_N(\lambda).
\]
For every composition \(\mu\) of \(N\le n\) and every
\(\lambda\vdash a\) with \(0\le a\le N\),
\[
 \chi^{(n)}_\lambda(\widehat T^{(n)}_\mu)
 =\sum_{Q\in\mathrm{RT}_N(\lambda)}\wt_\mu(Q).
\]
\end{mainthm}

The construction uses signed column insertion in the formulation of
La Scala--Nardozza--Senato \cite[Definition~2.6]{LNS06} and records
star positions separately (Definition~\ref{def:super-mirabolic-rsk}
and Theorem~\ref{thm:super-mirabolic-rsk-bijection}).
The resulting hook Schur expansion has coefficients
\(\sum_{Q\in\mathrm{RT}_N(\lambda)}\wt_\mu(Q)\)
(Proposition~\ref{prop:super-relative-position} and
Theorem~\ref{thm:super-mirabolic-schur-expansion}).
Comparison with Theorem~\ref{thm:intro-super-mirabolic-frobenius},
using sufficiently many even variables, gives the Roichman formula
(Theorem~\ref{thm:super-mirabolic-roichman}).
Travkin's geometric correspondence \cite[Sections~3.3--3.4]{Tra09}
uses decorated permutations and outputs two standard tableaux and an
additional partition, whereas our correspondence uses super words and
records the star positions in \(Q^{(0)}\).

\subsubsection{Parameter inversion}
\label{subsubsec:intro-parameter-inversion}

For the \(q\)-rook monoid algebra \(R_n(q)\) in the normalization of
\cite{DHP03}, Proposition~\ref{prop:parameter-inversion-isomorphism} gives
\[
 \Theta_q:H_n^{\mathrm{mir}}(q^{-1})\xrightarrow{\sim}R_n(q).
\]
This transports the duality, annihilator, and character formulas to
\(R_n(q)\) (Theorems~\ref{thm:super-q-rook-schur-weyl},
\ref{thm:super-q-rook-frobenius}, \ref{thm:super-q-rook-regev}
and Corollaries~\ref{cor:super-q-rook-annihilator},
\ref{cor:q-rook-roichman}).
The pure-even specializations recover Xiao's generic annihilator theorem
\cite[Theorem~1.2]{X17} and the Frobenius formula of
Dieng--Halverson--Poladian \cite[Theorem~2.3]{DHP03}; at full rank,
the super Frobenius formula recovers Mitsuhashi's formula
\cite[Theorem~5.5]{Mit10}.

The tensor decomposition gives a representation-theoretic interpretation
of the hook and two-row character sums in
\cite[Equations~(3.35)--(3.36)]{JWL26}, answering the question posed there
(Propositions~\ref{cor:jwl-hook-representation}
and~\ref{prop:jwl-two-row-representation}).

In the opposite direction, the character-table structure of
\cite[Corollary~2.5]{DHP03} and the generating function of
\cite[Theorem~6.12]{JL26} yield the mirabolic formulas in
Corollaries~\ref{thm:mirabolic-block-upper-triangular}
and~\ref{thm:mirabolic-character-generating-function}.

\subsection{Organization}
\label{subsec:introduction-organization}

Section~\ref{sec:preliminaries} fixes notation and relates the two algebras.
Sections~\ref{sec:super-mirabolic-SW}--\ref{sec:tensor-space-kernels}
prove the duality and determine the annihilator and faithful quotient.
Sections~\ref{sec:super-frobenius}--\ref{sec:mirabolic-rsk} derive character
formulas and construct the super mirabolic RSK correspondence.
Section~\ref{sec:q-rook-normalization} obtains the \(q\)-rook results and
their representation-theoretic interpretations, then the mirabolic
character tables and generating functions.

\section{Preliminaries}
\label{sec:preliminaries}

We first introduce the tableau notation and the algebras used below.
We then relate the mirabolic and rook algebras by parameter inversion
and describe their irreducible modules and algebra decompositions.
Finally, we recall the supersymmetric functions needed for the character formulas.

Let \(v\) be an indeterminate, put \(q=v^2\), and set
\(\mathbb F=\mathbb C(q)\subset\mathbb K=\mathbb C(v)\).
The Hecke, mirabolic Hecke, and \(q\)-rook algebras and their characters are defined over \(\mathbb F\).  A subscript \(\mathbb K\)
on an algebra or a module denotes scalar extension to \(\mathbb K\).
Quantum superalgebras and their natural modules are over \(\mathbb K\).
We use ordinary associative centralizers.
Throughout, \(n\ge1\), \(m,s\ge0\), and \(m+s\ge1\).

Let \(V_{\mathbb F}\) have ordered homogeneous basis
\(v_1,\ldots,v_{m+s}\), with the first \(m\) vectors even and the
last \(s\) odd.  Adjoin an even vector \(v_\star\), ordered first,
and set
\[
 \widetilde V_{\mathbb F}=\mathbb Fv_\star\oplus V_{\mathbb F},
 \qquad W_{\mathbb F,n}=\widetilde V_{\mathbb F}^{\otimes_{\mathbb F}n},
 \qquad V=\mathbb K\otimes_{\mathbb F}V_{\mathbb F},
 \qquad\widetilde V=\mathbb K\otimes_{\mathbb F}\widetilde V_{\mathbb F}.
\]
Tensor powers of \(V\) and \(\widetilde V\) are taken over
\(\mathbb K\).

\subsection{Partitions, compositions, and tableaux}

We write \([r]=\{1,\ldots,r\}\) for an integer \(r\ge0\), with
\([0]=\varnothing\).  A composition of \(r\), denoted \(\alpha\vDash r\), is a finite sequence
\(\alpha=(\alpha_1,\ldots,\alpha_\ell)\) of positive integers with
\(
        |\alpha|=\alpha_1+\cdots+\alpha_\ell=r.
\)
The empty composition is denoted by \(\varnothing\) and has size \(0\).  A
partition is a weakly decreasing composition.  We write
\[
        \mathscr P_r=\{\lambda\mid \lambda\vdash r\},
        \qquad
        \mathscr P_{\le n}=\coprod_{r=0}^{n}\mathscr P_r,
\]
and \(\mathscr P=\coprod_{r\ge0}\mathscr P_r\).  The length of a partition or
composition \(\alpha\) is denoted by \(\ell(\alpha)\), and the conjugate
partition of \(\lambda\) is denoted by \(\lambda'\).  We pad partitions
with zero parts.  The notation \(\mu\subseteq\lambda\) means
\(\mu_i\le\lambda_i\) for every \(i\).

If \(\mu\subseteq\lambda\), the skew diagram is denoted by \(\lambda/\mu\), and
\(
        |\lambda/\mu|=|\lambda|-|\mu|.
\)
We use the English convention for Young diagrams.  Thus row indices increase
downward and column indices increase from left to right.  A skew diagram is
called a generalized ribbon if it contains no \(2\times2\) block.
A standard Young tableau of shape \(\lambda\vdash a\) fills the diagram
by \(1,\ldots,a\), each once, increasing along rows and columns.
Write \(\operatorname{SYT}(\lambda)\) for this set and
\(f^\lambda=|\operatorname{SYT}(\lambda)|\).

For a composition
\(
        \mu=(\mu_1,\ldots,\mu_\ell)
\)
of \(r\), set
\[
        B(\mu)
        =
        \{\mu_1,\ \mu_1+\mu_2,\ \ldots,\ \mu_1+\cdots+\mu_\ell=r\}.
\]
This is the set of right endpoints of the consecutive \(\mu\)-blocks;
set \(B(\varnothing)=\varnothing\).  If
\(r\le n\), we also write
\[
        \mu\uparrow n=(1^{\,n-r},\mu_1,\ldots,\mu_\ell).
\]

A bicomposition of \(r\) is a pair
\(
        (\alpha;\beta)
\)
of compositions, either of which may be empty, such that
\(
        |\alpha|+|\beta|=r.
\)
We write \((\alpha;\beta)\vDash r\).

\subsection{Hecke algebras and standard elements}

Let \(H_n(q)\) be the Iwahori--Hecke algebra of type \(A_{n-1}\) over
\(\mathbb F\), generated by
\(
        T_1,\ldots,T_{n-1},
\)
with braid relations and quadratic relation
\begin{equation}\label{e:Hecke-reln}
        T_i^2=(q-1)T_i+q,
        \qquad\text{equivalently}\qquad
        (T_i-q)(T_i+1)=0.
\end{equation}
Here \(\mathfrak S_n\) is the symmetric group on \([n]\), with adjacent
transpositions \(s_i=(i,i+1)\); its length \(\ell(w)\) is the number of
inversions of \(w\).  For \(w\in\mathfrak S_n\), set
\(T_w=T_{i_1}\cdots T_{i_d}\), where
\(w=s_{i_1}\cdots s_{i_d}\) is any reduced expression.  The braid
relations make this independent of the chosen reduced expression.  We set
\(H_0(q)=H_1(q)=\mathbb F\).

For a composition \(\mu=(\mu_1,\ldots,\mu_\ell)\) of \(r\), let
\(
        \gamma_\mu\in\mathfrak S_r
\)
be the product of the disjoint cycles on consecutive blocks,
\[
        (1,\ldots,\mu_1)
        (\mu_1+1,\ldots,\mu_1+\mu_2)
        \cdots
        (\mu_1+\cdots+\mu_{\ell-1}+1,\ldots,r).
\]
If a part is equal to \(1\), the corresponding cycle is the identity.  We use
\(
        T_{\gamma_\mu}
\)
for the corresponding Hecke element.  Products act on tensors from right
to left.  When \(\mu\uparrow n\) is used,
the nontrivial cycles are shifted to the last \(|\mu|\) letters.

\subsection{The mirabolic Hecke algebra}
\label{subsec:mirabolic-presentation}

The mirabolic Hecke algebra $H^{\mathrm{mir}}_n(q)$ is generated by
\[
        T_1,\ldots,T_{n-1},
        \qquad
        P_1,\ldots,P_n,
\]
where the \(T_i\)'s satisfy the Hecke relations of type \(A\) above, and the
\(P_i\)'s satisfy
\[
        P_i^2=P_i,
        \qquad
        P_iP_j=P_jP_i=P_i
        \quad (j<i).
\]
The mixed relations are
\[
        P_iT_j=T_jP_i
        \quad (i<j),
\]
\[
        P_iT_j=T_jP_i=-P_i
        \quad (j<i),
\]
and
\[
        P_{i+1}=-P_iT_i^{-1}P_i
        \qquad (1\le i<n),
\]
where
\(
        T_i^{-1}=q^{-1}(T_i-(q-1)).
\)
We also use the convention \(P_0=1\).

For a composition \(\mu\) of \(r\le n\), define
\begin{equation}
\label{eq:mirabolic-standard-element}
        \widehat T^{(n)}_\mu
        =
        P_{n-r}T_{\gamma_{\mu\uparrow n}}
        \in H^{\mathrm{mir}}_n(q).
\end{equation}
For \(\mu=\varnothing\), this gives \(\widehat T^{(n)}_\varnothing=P_n\).
For an associative algebra \(A\), write \([A,A]\) for the
\(\mathbb F\)-linear span of all commutators \(ab-ba\), with \(a,b\in A\).
For the special case of partitions $\mu$, it is known
\cite[Theorem~2.13]{Wan26} that the images of the elements
\(
        \widehat T^{(n)}_\mu,
        \mu\in\mathscr P_{\le n},
\)
form a basis of the cocenter
\(
        H^{\mathrm{mir}}_n(q)/
        [H^{\mathrm{mir}}_n(q),H^{\mathrm{mir}}_n(q)].
\)
Consequently character values on these elements determine the irreducible
characters.

At \(q=1\), changing the sign of each Hecke generator identifies
the specialization with the rook monoid algebra.  At generic parameter,
the following isomorphism relates the two presentations.

\subsection{The \texorpdfstring{\(q\)}{q}-rook monoid algebra and parameter inversion}

Let \(R_n(q)\) be the \(q\)-rook monoid algebra, written in the normalization
used in \cite{DHP03}, with Hecke generators
\(\mathsf T_i\) and idempotents \(\mathsf P_j\).  Thus
\[
(\mathsf T_i-q)(\mathsf T_i+1)=0,\qquad
\mathsf T_i\mathsf P_j=\mathsf P_j\mathsf T_i=q\mathsf P_j
\quad(i<j),
\]
\[
\mathsf T_i\mathsf P_j=\mathsf P_j\mathsf T_i\quad(j<i),\qquad
\mathsf P_{i+1}=q\mathsf P_i\mathsf T_i^{-1}\mathsf P_i.
\]
For a permutation \(w=s_{i_1}\cdots s_{i_d}\) in reduced form, write
\(\mathsf T_w=\mathsf T_{i_1}\cdots\mathsf T_{i_d}\).
We use \(\mathsf P_0=1\); as usual,
\(\mathsf P_i\mathsf P_j=\mathsf P_j\mathsf P_i=\mathsf P_i\) for \(j<i\).

Let \(\Theta_q\) denote the assignment from the mirabolic generators at
parameter \(q^{-1}\) to \(R_n(q)\) given by
\[
 T_i\longmapsto-q^{-1}\mathsf T_i,\qquad P_j\longmapsto\mathsf P_j.
\]

\begin{proposition}
\label{prop:parameter-inversion-isomorphism}
The assignment \(\Theta_q\) extends to an algebra isomorphism
\[
 \Theta_q:H^{\mathrm{mir}}_n(q^{-1})\xrightarrow{\sim}R_n(q).
\]
\end{proposition}

\begin{proof}
The braid relations are unchanged by the common rescaling, and
\[
q^{-2}\mathsf T_i^2
=(q^{-1}-1)(-q^{-1}\mathsf T_i)+q^{-1},
\]
so \(-q^{-1}\mathsf T_i\) satisfies the mirabolic Hecke relation at
parameter \(q^{-1}\).  The two mixed relations follow respectively from
\[
(-q^{-1}\mathsf T_j)\mathsf P_i
=\mathsf P_i(-q^{-1}\mathsf T_j)=-\mathsf P_i\quad(j<i)
\]
and the corresponding commutation relation in \(R_n(q)\).  Finally,
\[
-\mathsf P_i(-q^{-1}\mathsf T_i)^{-1}\mathsf P_i
=q\mathsf P_i\mathsf T_i^{-1}\mathsf P_i
=\mathsf P_{i+1}.
\]
Thus the assignment is a homomorphism.  The reverse assignment
\(\mathsf T_i\mapsto-qT_i\), \(\mathsf P_j\mapsto P_j\) satisfies the same
relations and is its inverse.
\end{proof}

\begin{remark}
\label{rem:parameter-inversion-comparison}
The isomorphism in Proposition~\ref{prop:parameter-inversion-isomorphism}
relates \(H_n^{\mathrm{mir}}(q^{-1})\) and \(R_n(q)\), with different
parameters.  Wan \cite[Remark~2.9]{Wan26} discusses an isomorphism
between \(H_n^{\mathrm{mir}}(q)\) and \(R_n(q)\) at the same parameter \(q\).
\end{remark}

For a composition \(\mu\) of \(k\le n\), put
\[
        d(\mu)=|\mu|-\ell(\mu),\qquad
        T^R_\mu=\mathsf P_{n-k}\mathsf T_{\gamma_{\mu\uparrow n}},
\]
with \(T^R_\varnothing=\mathsf P_n\).

\begin{lemma}
\label{lem:standard-elements-under-theta}
For every such \(\mu\), the standard elements transform according to
\[
        \Theta_q\left(
        \widehat T^{(n)}_\mu\big|_{q\mapsto q^{-1}}\right)
        =(-q^{-1})^{d(\mu)}T^R_\mu.
\]
\end{lemma}

\begin{proof}
The cycle product \(T_{\gamma_{\mu\uparrow n}}\) contains
\(\sum_i(\mu_i-1)=d(\mu)\) Hecke generators, each of which contributes
\(-q^{-1}\), while \(P_{n-k}\) maps to \(\mathsf P_{n-k}\).
\end{proof}

\subsection{Irreducible modules and characters}
\label{subsec:irreducible-characters}

We first record how ordinary Specht labels change under parameter
inversion and the sign automorphism.  We then describe the simple
\(q\)-rook modules over \(\mathbb F\) and use
\(\Theta_{q^{-1}}\) to fix the mirabolic module labels.

For \(a\ge0\) and \(\nu\vdash a\), let
\(S^\nu_{a,\mathrm{std}}\) be the ordinary Specht module of
\(H_a(q)\) over \(\mathbb F\).  For \(a\ge1\), we use the
convention that each generator \(T_i\) acts by \(q\) on the
one-dimensional module \(S^{(a)}_{a,\mathrm{std}}\).  The empty
Specht module is \(\mathbb F\).  These modules are absolutely irreducible
and exhaust the simple modules of the split semisimple algebra
\(H_a(q)\) \cite[Theorems~3.22--3.23]{Ram97}.  A subscript
\(\mathbb K\) on a module denotes its scalar extension from
\(\mathbb F\) to \(\mathbb K\).
For a tableau \(L\) with distinct entries and an entry \(i\) of
\(L\), write \(\operatorname{ct}_L(i)\) for its column index
minus its row index.

Fix \(a\ge1\), and write \(u_i\), for \(1\le i<a\), for the
generators of \(H_a(q^{-1})\) over the same field \(\mathbb F\).
For \(\nu\vdash a\), let \(\widetilde S^\nu_{a,\mathrm{std}}\)
be the ordinary Specht module of \(H_a(q^{-1})\) over \(\mathbb F\).
On the one-dimensional module \(\widetilde S^{(a)}_{a,\mathrm{std}}\),
each \(u_i\) acts by \(q^{-1}\).
Denote the parameter-inversion isomorphism by \(\phi\) and the
sign automorphism by \(\psi\).  These \(\mathbb F\)-algebra maps
are given by
\[
 \begin{aligned}
 \phi:H_a(q)&\longrightarrow H_a(q^{-1}),& T_i&\longmapsto u_i^{-1},\\
 \psi:H_a(q)&\longrightarrow H_a(q),& T_i&\longmapsto-qT_i^{-1}=(q-1)-T_i.
 \end{aligned}
\]
Write \(\phi^*\) and \(\psi^*\) for pullback of modules along
these maps.

\begin{lemma}
\label{lem:hecke-specht-labels}
For every \(\lambda\vdash a\), there are isomorphisms of
\(H_a(q)\)-modules
\[
 \phi^*\widetilde S^\lambda_{a,\mathrm{std}}
 \cong S^\lambda_{a,\mathrm{std}},
 \qquad
 \psi^*S^\lambda_{a,\mathrm{std}}
 \cong S^{\lambda'}_{a,\mathrm{std}}.
\]
\end{lemma}

\begin{proof}
The elements \(u_i^{-1}\), for \(1\le i<a\), satisfy the defining
relations of \(H_a(q)\), as do the elements \(-qT_i^{-1}\).
The maps \(\phi\) and \(\psi\) have inverses given by \(u_i\mapsto T_i^{-1}\)
and \(T_i\mapsto-qT_i^{-1}\), respectively.
We identify the Specht labels of the pullbacks by comparing the
eigenvalues of the multiplicative Jucys--Murphy elements.  In
\(H_a(q)\), define these elements by
\begin{equation}
\label{eq:type-a-multiplicative-jm}
 J_1=1,\qquad
 J_k=q^{1-k}T_{k-1}\cdots T_1^2\cdots T_{k-1}
 \quad(2\le k\le a).
\end{equation}
In \(H_a(q^{-1})\), define the multiplicative Jucys--Murphy
elements \(J_k'\) by
\[
 J_1'=1,\qquad
 J_k'=q^{k-1}u_{k-1}\cdots u_1^2\cdots u_{k-1}
 \quad(2\le k\le a).
\]
For \(\lambda\vdash a\), the Specht module
\(S^\lambda_{a,\mathrm{std}}\) has a seminormal basis indexed by
the standard tableaux \(L\) of shape \(\lambda\).  The basis
vector indexed by \(L\) is an eigenvector of \(J_k\) with
eigenvalue \(q^{\operatorname{ct}_L(k)}\), for \(1\le k\le a\); see
\cite[(3.20), Theorems~3.22--3.23]{Ram97}.
Our \(T_i\) is \(v\) times Ram's corresponding generator, with
his parameter \(q\) identified with our \(v\).
The resulting matrix coefficients are
rational functions of \(q\), so these models are defined over
\(\mathbb F\).  The complete content sequence determines a standard
tableau, since the addable boxes of every partition have distinct
contents.

Pullback along an algebra isomorphism preserves irreducibility.
Since different standard tableaux have different complete content
sequences, the eigenvalues below determine the Specht label of
each pullback.

The map \(\phi\) fixes the scalar \(q\) and sends \(T_i\) to
\(u_i^{-1}\), so \(\phi(J_k)=(J_k')^{-1}\).
On \(\widetilde S^\lambda_{a,\mathrm{std}}\), the eigenvalue of
\(J_k'\) on the vector indexed by \(L\) is
\(q^{-\operatorname{ct}_L(k)}\).  Thus \(J_k\) acts on
\(\phi^*\widetilde S^\lambda_{a,\mathrm{std}}\) with eigenvalue
\(q^{\operatorname{ct}_L(k)}\), proving the first isomorphism.
For \(2\le k\le a\), substitution in
\eqref{eq:type-a-multiplicative-jm} gives
\[
 \psi(J_k)
 =q^{1-k}(-q)^{2k-2}
   T_{k-1}^{-1}\cdots T_1^{-2}\cdots T_{k-1}^{-1}
 =J_k^{-1};
\]
the same identity holds for \(k=1\).  The content sequence
on \(\psi^*S^\lambda_{a,\mathrm{std}}\) is therefore the negative
of the original sequence.  Transposing a standard tableau negates
its contents, which proves the second isomorphism.
\end{proof}

Consequently, for \(\lambda\vdash a\), pullback along
\(\phi\circ\psi\), given by \(T_i\mapsto-q u_i\), sends
\(\widetilde S^\lambda_{a,\mathrm{std}}\)
to a module isomorphic to \(S^{\lambda'}_{a,\mathrm{std}}\).  For \(a=0\), the corresponding maps are the
identity of \(\mathbb F\) and preserve the empty label.

For \(\lambda\vdash a\), with \(0\le a\le n\), an
\(n\)-standard rook tableau of shape \(\lambda\) is a row- and
column-strict filling by an \(a\)-element subset of \([n]\).
For such a tableau \(L\) and \(1\le i<n\), the filling \(s_iL\)
is obtained by interchanging the labels \(i,i+1\), replacing one
by the other if only one occurs.

Halverson constructs the simple \(R_{n,\mathbb K}(q)\)-modules
indexed by \(\lambda\in\mathscr P_{\le n}\), with seminormal basis
vectors \(\zeta_L\) indexed by the \(n\)-standard rook tableaux
of shape \(\lambda\)
\cite[Theorems~1.3 and~3.2]{Hal04}.  They have forms over \(\mathbb F\):
write \(\operatorname{entries}(L)\) for the set of entries of \(L\),
and replace each basis vector by
\[
 w_L=v^{\sum_{i\in\operatorname{entries}(L)}i}\zeta_L.
\]
For \(1\le i<n\), suppose that \(i\) does not occur in \(L\),
whereas \(i+1\) does.  Since replacing \(i+1\) by \(i\) decreases
the sum of the entries by one, the change of basis transforms Halverson's
formula \cite[Equation~(3.7)]{Hal04} into
\[
 \mathsf T_iw_L=(q-1)w_L+q\,w_{s_iL},
 \qquad
 \mathsf T_iw_{s_iL}=w_L.
\]
Thus the coefficients in these two cases belong to \(\mathbb F\).
When both entries occur, their exchange preserves the sum of the
entries, so the coefficients are unchanged and already belong to
\(\mathbb F\).  When neither occurs, \(\mathsf T_i\) acts by \(q\).
The projection operators are unchanged by the rescaling.  By
\cite[Equation~(3.8)]{Hal04}, for every basis tableau \(L\) and
\(0\le j\le n\), their action is
\begin{equation}
\label{eq:q-rook-projection-action}
 \mathsf P_jw_L=
 \begin{cases}
  w_L,&\operatorname{entries}(L)\cap[j]=\varnothing,\\
  0,&\text{otherwise}.
 \end{cases}
\end{equation}
Thus \(\mathsf P_j\) fixes the basis vectors whose tableaux avoid
\(1,\ldots,j\) and kills all the others.  All generator matrices
therefore have entries in \(\mathbb F\), so the
\(\mathbb F\)-span of the \(w_L\) is stable under all generators.
Denote this module by \(M^{R,n}_\lambda\),
and its scalar extension by \(M^{R,n}_{\lambda,\mathbb K}\).
Since the \(w_L\) also form a \(\mathbb K\)-basis of Halverson's
module, \(M^{R,n}_{\lambda,\mathbb K}\) is isomorphic to that module.
Thus the modules \(M^{R,n}_{\lambda,\mathbb K}\), for
\(\lambda\in\mathscr P_{\le n}\), form a complete set of pairwise
nonisomorphic simple \(R_{n,\mathbb K}(q)\)-modules
\cite[Theorem~3.2]{Hal04}.

For \(\lambda\vdash a\), with \(0\le a\le n\), choosing the
\(a\) entries and replacing them by their relative order gives
\[
 \dim_{\mathbb F}M^{R,n}_\lambda=\binom na f^\lambda.
\]

Consider the natural action map from \(R_n(q)\) to the direct sum
of the endomorphism algebras of \(M^{R,n}_\lambda\), for
\(\lambda\in\mathscr P_{\le n}\).

\begin{proposition}
\label{prop:q-rook-split-decomposition}
The natural action map is an isomorphism of \(\mathbb F\)-algebras
\[
 R_n(q)\xrightarrow{\sim}
 \bigoplus_{\lambda\in\mathscr P_{\le n}}
       \operatorname{End}_{\mathbb F}(M^{R,n}_\lambda).
\]
\end{proposition}

\begin{proof}
We first prove injectivity of the action map.
Let \(x\in R_n(q)\) act by zero on every \(M^{R,n}_\lambda\).
After scalar extension it acts by zero on Halverson's complete family
of simple modules.  By \cite[Theorem~1.3]{Hal04} and scalar extension
in characteristic zero, \(R_{n,\mathbb K}(q)\) is semisimple.
Its regular module is therefore a direct sum of copies of these simples.
Thus the extended element acts by zero on the
identity and is zero.  Injectivity of scalar extension gives \(x=0\).

To prove surjectivity, we compare the dimensions of the source and target:
\[
 \begin{aligned}
 \sum_{\lambda\in\mathscr P_{\le n}}
       \bigl(\dim_{\mathbb F}M^{R,n}_\lambda\bigr)^2
 =\sum_{a=0}^{n}\binom na^2
       \sum_{\lambda\vdash a}(f^\lambda)^2=\sum_{a=0}^{n}\binom na^2a!
 =\dim_{\mathbb F}R_n(q).
 \end{aligned}
\]
For each \(a\ge0\), the identity
\(\sum_{\lambda\vdash a}(f^\lambda)^2=a!\) comes from the
irreducible representations of \(\mathfrak S_a\).  The last equality
follows from \cite[(1.1), Theorem~1.3]{Hal04}.
The source and target have the same finite dimension, so the
injective action map is an isomorphism.
\end{proof}

Proposition~\ref{prop:q-rook-split-decomposition} shows that \(R_n(q)\) is split semisimple over
\(\mathbb F\), and that the \(R_n(q)\)-modules
\(M^{R,n}_\lambda\), for \(\lambda\in\mathscr P_{\le n}\), form
a complete set of pairwise nonisomorphic absolutely irreducible
modules over \(\mathbb F\).

For \(\lambda\vdash a\), with \(0\le a\le n\), define the
mirabolic module \(N^{(n)}_{(\lambda,a)}\) by pulling
\(M^{R,n}_\lambda\) at parameter \(q^{-1}\) back through
\(\Theta_{q^{-1}}:H^{\mathrm{mir}}_n(q)\to R_n(q^{-1})\).
The parameter change replaces every rational coefficient in \(q\)
by the corresponding coefficient in \(q^{-1}\).
Proposition~\ref{prop:q-rook-split-decomposition} and \(\Theta_{q^{-1}}\)
show that the modules \(N^{(n)}_{(\lambda,a)}\) form a complete set of pairwise
nonisomorphic absolutely irreducible mirabolic modules over
\(\mathbb F\).

Fix \(1\le a\le n\) and \(\lambda\vdash a\).  Since
\(P_{n-a}\) is idempotent, \(P_{n-a}N^{(n)}_{(\lambda,a)}\)
is the image of this projection on \(N^{(n)}_{(\lambda,a)}\).

For \(1\le i<a\), the operator \(T_{n-a+i}\) commutes with
\(P_{n-a}\).  Thus, for \(z\in N^{(n)}_{(\lambda,a)}\),
\[
 T_{n-a+i}(P_{n-a}z)=P_{n-a}(T_{n-a+i}z)
 \in P_{n-a}N^{(n)}_{(\lambda,a)},
\]
so these operators preserve the subspace.

For the following corollary and its proof, write
\(t_1,\ldots,t_{a-1}\) for the generators of \(H_a(q)\).
We define their action on the subspace by
\[
 t_i\cdot u=T_{n-a+i}u
 \qquad
 \bigl(u\in P_{n-a}N^{(n)}_{(\lambda,a)},\ 1\le i<a\bigr).
\]
The operators on the right satisfy the defining Hecke relations,
so this makes \(P_{n-a}N^{(n)}_{(\lambda,a)}\) an
\(H_a(q)\)-module.

\begin{corollary}
\label{cor:mirabolic-specht-label}
There is an isomorphism of \(H_a(q)\)-modules
\[
 P_{n-a}N^{(n)}_{(\lambda,a)}\cong S^{\lambda'}_{a,\mathrm{std}}.
\]
\end{corollary}

\begin{proof}
Let \(E=(\phi\circ\psi)^*\widetilde S^\lambda_{a,\mathrm{std}}\).
Its underlying vector space is that of
\(\widetilde S^\lambda_{a,\mathrm{std}}\), with \(t_i\) acting
as \(-q u_i\), for \(1\le i<a\).  Choose its seminormal basis
\(\{b_Q:Q\in\operatorname{SYT}(\lambda)\}\) in the normalization
of \cite[(3.7), Remark~3.1]{Hal04}, at parameter \(q^{-1}\).

Write \(w_L\) also for the basis of \(N^{(n)}_{(\lambda,a)}\)
obtained from the rook basis by parameter inversion and pullback.
Formula~\eqref{eq:q-rook-projection-action} also describes the action
of \(P_j\) on this basis: parameter inversion leaves its
coefficients \(0,1\) unchanged, and \(\Theta_{q^{-1}}\) sends
\(P_j\) to \(\mathsf P_j\).  Hence \(P_{n-a}\) retains exactly
those basis vectors whose tableaux have no entries in \([n-a]\).
Since each tableau has \(a\) entries, its entry set must then be
\(\{n-a+1,\ldots,n\}\).  For each such tableau \(L\), let
\(\bar L\) be obtained by subtracting \(n-a\) from every entry.
Define an \(\mathbb F\)-linear map by
\[
 f:P_{n-a}N^{(n)}_{(\lambda,a)}\longrightarrow E,
 \qquad w_L\longmapsto b_{\bar L}.
\]
Subtracting \(n-a\) is a bijection on the indexing tableaux, with
inverse given by adding \(n-a\) to every entry.  Thus \(f\) sends
a basis bijectively to a basis and is a linear isomorphism.

We check compatibility with the generators.  For \(1\le i<a\),
both \(n-a+i\) and \(n-a+i+1\) occur in \(L\), and relabelling
preserves box contents and standardness.  The rescaling from \(\zeta_L\) to
\(w_L\) is constant on this entry set and does not change the
generator matrices.  Hence \cite[(3.7), Remark~3.1]{Hal04} identifies
the action of \(\mathsf T_{n-a+i}\) at parameter \(q^{-1}\)
with that of \(u_i\) on the basis \(b_{\bar L}\).
Since the pullback sends \(T_{n-a+i}\) to
\(-q\mathsf T_{n-a+i}\), we obtain
\[
 f(t_i\cdot w_L)=f(T_{n-a+i}w_L)
 =-q u_i f(w_L)=t_i\cdot f(w_L).
\]
Thus \(f\) is an \(H_a(q)\)-module isomorphism.
Finally, Lemma~\ref{lem:hecke-specht-labels} gives
\(E\cong S^{\lambda'}_{a,\mathrm{std}}\), completing the proof.
\end{proof}

For \(1\le a\le n\), we also have
\(P_{n-a+1}N^{(n)}_{(\lambda,a)}=0\): no tableau with \(a\)
entries can avoid all of \(1,\ldots,n-a+1\).
For \(a=0\), the module is one-dimensional, all \(P_j\) act by
\(1\), and all \(T_i\) act by \(-1\).

Write \(\chi^{(n)}_\lambda\) for the character of
\(N^{(n)}_{(\lambda,|\lambda|)}\), and \(\chi^{A,\lambda}\)
for the character of \(S^\lambda_{a,\mathrm{std}}\),
where \(a=|\lambda|\).  The mirabolic character table has rows
indexed by \(\lambda\in\mathscr P_{\le n}\) and columns indexed
by \(\widehat T^{(n)}_\mu\), for \(\mu\in\mathscr P_{\le n}\).

\subsection{Supersymmetric functions}

Let
\[
        x=(x_1,\ldots,x_m),
        \qquad
        y=(y_1,\ldots,y_s).
\]
All these variables commute: \(x_i\) records the weight of the even
basis vector \(v_i\), and \(y_j\) that of the odd basis vector
\(v_{m+j}\).  With a formal variable \(t\), define the complete
supersymmetric functions by
\[
        \sum_{r\ge0}h_r(x|y)t^r
        =
        \frac{\prod_{j=1}^{s}(1+y_jt)}
             {\prod_{i=1}^{m}(1-x_it)}.
\]
Set \(h_r(x|y)=0\) for \(r<0\); the series gives \(h_0(x|y)=1\).
For a partition \(\lambda\) and an integer \(d\ge\ell(\lambda)\), define
\[
 s_\lambda(x|y)=
 \det\bigl(h_{\lambda_i-i+j}(x|y)\bigr)_{1\le i,j\le d},
\]
with \(s_\varnothing(x|y)=1\).  Increasing \(d\) by one adds a last
row \((0,\ldots,0,1)\), so the determinant is independent of \(d\).
By \cite[Corollary~6.5]{BR87}, the polynomial \(s_\lambda(x|y)\)
is nonzero precisely for the \((m,s)\)-hook partitions
\[
        H(m,s)=\{\lambda\mid \lambda_{m+1}\le s\}.
\]
For \(\lambda\in H(m,s)\), write \(L_{m|s}(\lambda)\) for the
covariant irreducible \(U_{v^{-1}}(\mathfrak{gl}(m|s))\)-module over
\(\mathbb K\) with label \(\lambda\).

The mirabolic tensor space contains one extra even trivial summand.  We record
it by the extended alphabet
\[
        z=(z_0,z_1,\ldots,z_m)=(1,x_1,\ldots,x_m).
\]
For a composition \(\alpha=(\alpha_1,\ldots,\alpha_p)\), define
\[
        M_\alpha(z)
        =
        \sum_{0\le i_1<\cdots<i_p\le m}
        z_{i_1}^{\alpha_1}\cdots z_{i_p}^{\alpha_p},
        \qquad
        M_\varnothing(z)=1.
\]
Similarly, for \(\beta=(\beta_1,\ldots,\beta_t)\), set
\[
        M_\beta(y)
        =
        \sum_{1\le j_1<\cdots<j_t\le s}
        y_{j_1}^{\beta_1}\cdots y_{j_t}^{\beta_t},
        \qquad
        M_\varnothing(y)=1.
\]

For ordinary symmetric functions, let \(p_r,h_r,e_r\) denote the
power-sum, complete, and elementary symmetric functions of degree \(r\),
respectively.  For a sequence \(\alpha=(\alpha_1,\ldots,\alpha_d)\) of
nonnegative integers, put \(h_\alpha=\prod_i h_{\alpha_i}\) and
\(e_\alpha=\prod_i e_{\alpha_i}\).  Their coefficient field is \(\mathbb F\), and
\(h_0=e_0=1\), while \(h_r=e_r=0\) for \(r<0\).  A virtual alphabet
\(A\) specifies a specialization \(f\mapsto f[A]\) by its power sums
\(p_r[A]\).  We use
\(p_r[A+B]=p_r[A]+p_r[B]\),
\(p_r[A-B]=p_r[A]-p_r[B]\), and
\(p_r[AB]=p_r[A]p_r[B]\).  A single letter \(u\) has
\(p_r[u]=u^r\); in particular \(p_r[q-1]=q^r-1\).
For \(\mu\subseteq\lambda\) and \(d\ge\ell(\lambda)\), the skew
Schur function is
\[
 s_{\lambda/\mu}[A]
 =\det\bigl(h_{\lambda_i-\mu_j-i+j}[A]\bigr)_{1\le i,j\le d}.
\]
It is zero if \(\mu\not\subseteq\lambda\), and the empty skew diagram
has value \(1\).  The super specialization is given by
\(p_r(x|y)=\sum_i x_i^r+(-1)^{r-1}\sum_j y_j^r\).

For a formal variable \(t\), define
\[
 \sigma_t[A]=\sum_{r\ge0}h_r[A]t^r
 =\exp\left(\sum_{r\ge1}\frac{p_r[A]t^r}{r}\right).
\]
For an alphabet \(Z=(z_1,z_2,\ldots)\), put
\(\sigma_Z[A]=\prod_{i\ge1}\sigma_{z_i}[A]\).  Products and sums in
infinitely many letters are interpreted in the degree completion, or
coefficientwise in finitely many specified variables.  Thus
\[
 \sigma_t[q-1]=\frac{1-t}{1-qt},
 \qquad \sigma_{-t}[1-q]=\frac{1+qt}{1+t}.
\]
Here \(-t\), and later \(-Z\), mean ordinary negation of variables;
they are distinguished from the additive inverse of a virtual alphabet.
The notation \([t^r]\) denotes the coefficient of \(t^r\); for a
finite-support sequence \(\alpha\) of nonnegative integers,
\([Z^\alpha]\) extracts \(\prod_i z_i^{\alpha_i}\).

\section{Quantum super Schur--Weyl duality}
\label{sec:super-mirabolic-SW}

We first construct the commuting tensor actions and recall quantum super
Schur--Weyl duality in type \(A\).  We then prove the mirabolic
Schur--Weyl duality and determine the bimodule decomposition.
Finally, we decompose the quantum image by tensor degree.

\subsection{Tensor actions and type \texorpdfstring{\(A\)}{A} duality}

We use the fields \(\mathbb F=\mathbb C(q)\subset\mathbb K=\mathbb C(v)\),
where \(q=v^2\), and the spaces \(V\) and \(\widetilde V\) fixed in
Section~\ref{sec:preliminaries}.  Tensor powers of \(V\) and
\(\widetilde V\) in this section are over \(\mathbb K\).
The Hecke matrices are defined over \(\mathbb F\); their scalar extensions
will act on the same ordered tensor bases over \(\mathbb K\).

Write \(U=U_{v^{-1}}(\mathfrak{gl}(m|s))\) for the quantum general linear
superalgebra with the Hopf superalgebra convention of \cite[Section~3]{Mit06}.
For \(1\le j\le m+s\), let \(|v_j|\in\{0,1\}\) be the parity of \(v_j\),
so \(|v_j|=0\) for \(j\le m\) and \(|v_j|=1\) for \(j>m\).
Its Cartan generators are \(K_j^{\pm1}\), and its root generators are
\(E_i,F_i\), for \(1\le i<m+s\), of parity
\(|v_i|+|v_{i+1}|\pmod2\).  Put
\[
 A_i=K_i^{(-1)^{|v_i|}}K_{i+1}^{-(-1)^{|v_{i+1}|}}
 \qquad(1\le i<m+s).
\]
The coproduct is
\begin{equation}
\label{eq:quantum-coproduct-convention}
\begin{aligned}
 \Delta(K_j)&=K_j\otimes K_j,\\
 \Delta(E_i)&=E_i\otimes A_i^{-1}+1\otimes E_i,\\
 \Delta(F_i)&=F_i\otimes1+A_i\otimes F_i.
\end{aligned}
\end{equation}
Products of homogeneous tensor operators use the rule
\((a\otimes b)(u\otimes w)=(-1)^{|b||u|}au\otimes bw\), where
\(|b|\) and \(|u|\) are their parities.  With \(\delta\) denoting the Kronecker delta, the natural action is
\[
 K_jv_b=v^{-\delta_{jb}}v_b,\qquad
 E_iv_b=\delta_{b,i+1}v_i,\qquad
 F_iv_b=\delta_{b,i}v_{i+1}
 \quad(1\le b\le m+s).
\]
On the even vector \(v_\star\),
all \(K_j\) act by \(1\), and all \(E_i,F_i\) act by \(0\).
Thus \(\widetilde V=\mathbb Kv_\star\oplus V\) is a \(U\)-module.
The extended basis has index set \(\{\star,1,\ldots,m+s\}\), ordered by
\(\star<1<\cdots<m+s\), and \(|v_\star|=0\).
These conventions include \(m=0\) and \(s=0\);
a root is odd precisely when \(|v_i|\ne |v_{i+1}|\).

For \(a\ge0\), let \(\rho_a\) denote the standard tensor action on
\(V^{\otimes a}\), obtained by iterating
\eqref{eq:quantum-coproduct-convention}; see
\cite[Sections~2.1 and~3.2]{BKK00},
with quantum parameter \(v^{-1}\).
Applying the same construction to
\(\widetilde V=\mathbb Kv_\star\oplus V\) defines
\(\widetilde\rho_n\) on \(\widetilde V^{\otimes n}\).
In degree zero the action is the counit on \(\mathbb K\).
For a word \(\mathbf b=(b_1,\ldots,b_d)\) in the natural or extended
index alphabet,
write \(v_{\mathbf b}=v_{b_1}\otimes\cdots\otimes v_{b_d}\), and define
its inversion number by
\[
 \operatorname{inv}(\mathbf b)
 =\#\{(r,t):1\le r<t\le d,\ b_r>b_t\}.
\]
To compare the Hecke action below with the standard tensor operators,
define invertible diagonal maps \(D_a\) on \(V^{\otimes a}\) and
\(\widetilde D_n\) on \(\widetilde V^{\otimes n}\) by
\[
 D_av_{\mathbf b}=v^{-\operatorname{inv}(\mathbf b)}v_{\mathbf b},
 \qquad
 \widetilde D_nv_{\mathbf b}
 =v^{-\operatorname{inv}(\mathbf b)}v_{\mathbf b},
 \qquad D_0=1,
\]
using the natural and extended alphabets, respectively.  The quantum
actions used below are
\begin{equation}
\label{eq:conjugated-quantum-actions}
 \Psi_a(u)=D_a\rho_a(u)D_a^{-1},\qquad
 \widetilde\Psi_n(u)
 =\widetilde D_n\widetilde\rho_n(u)\widetilde D_n^{-1}
 \quad(u\in U).
\end{equation}
All centralizers below are full centralizers in ordinary endomorphism
algebras.

Define \(R\in\operatorname{End}_{\mathbb F}
(\widetilde V_{\mathbb F}^{\otimes2})\) on the extended tensor basis,
for \(i,j\in\{\star,1,\ldots,m+s\}\), by
\begin{equation}
\label{eq:mirabolic-tensor-operator}
R(v_i\otimes v_j)=
\begin{cases}
(1-q)v_i\otimes v_j+(-1)^{|v_i||v_j|}v_j\otimes v_i,&i<j,\\
v_i\otimes v_i,&i=j,\ |v_i|=0,\\
-qv_i\otimes v_i,&i=j,\ |v_i|=1,\\
(-1)^{|v_i||v_j|}qv_j\otimes v_i,&i>j.
\end{cases}
\end{equation}
On \(W_{\mathbb F,n}=\widetilde V_{\mathbb F}^{\otimes n}\), let
\(R_i\) act as \(R\) in positions \(i,i+1\) and as the identity in
the other positions, and put \(\mathcal T_i=-R_i\) for \(1\le i<n\).
Use the same symbols for their scalar extensions to
\(\widetilde V^{\otimes n}\).
For \(0\le j\le n\), define the projection \(\mathcal P_j\) by
\begin{equation}
\label{eq:tensor-projection-action}
 \mathcal P_jv_{\mathbf b}=
 \begin{cases}
 v_{\mathbf b},&b_1=\cdots=b_j=\star,\\
 0,&\text{otherwise,}
 \end{cases}
 \qquad \mathcal P_0=\operatorname{id}.
\end{equation}
For a reduced expression \(w=s_{i_1}\cdots s_{i_d}\) in \(\mathfrak S_n\),
write \(\mathcal T_w=\mathcal T_{i_1}\cdots\mathcal T_{i_d}\).

Let \(S_i\) be the standard tensor operators of
\cite[Equation~(2)]{Mit06}, with its parameter set to \(v^{-1}\),
on the ordered basis \(\{v_\star,v_1,\ldots,v_{m+s}\}\).
By \cite[Theorem~2.1]{Mit06}, the operators \(S_i\) satisfy the braid relations
and \(S_i^2=(v^{-1}-v)S_i+1\).  To compare their coefficients with
those of \(R_i\), let \(s_i\mathbf b\) be obtained from \(\mathbf b\)
by exchanging its letters in positions \(i,i+1\).  When
\(b_i\ne b_{i+1}\), conjugation by \(\widetilde D_n\) multiplies the coefficient of
\(v_{s_i\mathbf b}\) in \(vS_iv_{\mathbf b}\) by
\[
 v^{\operatorname{inv}(\mathbf b)-\operatorname{inv}(s_i\mathbf b)}
 =\begin{cases}
 v^{-1},&b_i<b_{i+1},\\
 v,&b_i>b_{i+1}.
 \end{cases}
\]
Thus the exchange coefficients become
\((-1)^{|v_{b_i}||v_{b_{i+1}}|}\) and
\((-1)^{|v_{b_i}||v_{b_{i+1}}|}q\), respectively.
The diagonal coefficients are unchanged.  This gives
\begin{equation}
\label{eq:R-conjugation-standard}
 R_i=\widetilde D_n(vS_i)\widetilde D_n^{-1},
 \qquad R_i^2=(1-q)R_i+q.
\end{equation}
This also proves the braid relations for \(R_i\) and \(\mathcal T_i\),
and makes \(\mathcal T_w\) independent of its reduced expression.
Since the entries lie in \(\mathbb F\), these matrix identities hold
over \(\mathbb F\) as well.  The same calculation on \(V^{\otimes a}\) gives the identity with
\(D_a\) in place of \(\widetilde D_n\).

Consider the assignments
\(T_i\mapsto\mathcal T_i\) and \(P_j\mapsto\mathcal P_j\).

\begin{proposition}
\label{prop:super-mirabolic-representation}
These assignments define a representation of
\(H^{\mathrm{mir}}_n(q)\) on \(W_{\mathbb F,n}\) whose scalar
extension commutes with \(\widetilde\Psi_n\).
\end{proposition}

\begin{proof}
We first verify the defining relations in
Subsection~\ref{subsec:mirabolic-presentation}.
Equation~\eqref{eq:R-conjugation-standard} gives the braid relations and
\(\mathcal T_i^2=(q-1)\mathcal T_i+q\).
By \eqref{eq:tensor-projection-action}, the projections satisfy
\(\mathcal P_i^2=\mathcal P_i\) and
\(\mathcal P_i\mathcal P_j=\mathcal P_j\mathcal P_i=\mathcal P_i\)
for \(j<i\).  For \(i<j\), the projection \(\mathcal P_i\) and the
operator \(\mathcal T_j\) act on disjoint sets of positions, so they
commute.  For \(j<i\), the two positions affected by \(\mathcal T_j\)
lie among the first \(i\).  On two \(\star\) letters it acts by \(-1\).
On every word having a nonstar letter in these first \(i\) positions,
it preserves the presence of such a letter.  Thus
\[
 \mathcal P_i\mathcal T_j=\mathcal T_j\mathcal P_i=-\mathcal P_i
 \qquad(j<i).
\]

For the recursion, fix a tensor basis vector \(v_{\mathbf b}\).
The quadratic relation gives
\(\mathcal T_i^{-1}=q^{-1}(\mathcal T_i-(q-1))\).
If the first \(i\) letters of \(\mathbf b\) are not all \(\star\),
then both \(\mathcal P_{i+1}\) and
\(-\mathcal P_i\mathcal T_i^{-1}\mathcal P_i\) send
\(v_{\mathbf b}\) to zero.
If they are all \(\star\), we use the two adjacent formulas
\[
 \mathcal T_i^{-1}(v_\star\otimes v_\star)=-v_\star\otimes v_\star,
 \qquad
 \mathcal T_i^{-1}(v_\star\otimes v_b)=-q^{-1}v_b\otimes v_\star
 \quad(1\le b\le m+s).
\]
If \(b_{i+1}=\star\), the first formula shows that
\(-\mathcal P_i\mathcal T_i^{-1}\mathcal P_i\) fixes
\(v_{\mathbf b}\).  Otherwise, the second formula moves the natural
letter into position \(i\), and the leftmost \(\mathcal P_i\)
sends the resulting vector to zero.  These are exactly the two cases
in \eqref{eq:tensor-projection-action} for \(\mathcal P_{i+1}\).
Therefore
\(\mathcal P_{i+1}=-\mathcal P_i\mathcal T_i^{-1}\mathcal P_i\).
All the defining relations hold over \(\mathbb F\).

We next verify commutation with \(\widetilde\Psi_n(U)\).
The extended space \(\widetilde V=\mathbb Kv_\star\oplus V\) is the natural module for
\(U_{v^{-1}}(\mathfrak{gl}(m+1|s))\), restricted to the Hopf subalgebra
on the last \(m+s\) basis vectors.  Its first even vector is trivial
for this subalgebra, so the induced tensor action restricts to
\(\widetilde\rho_n\).  The standard operators \(S_i\) commute with
\(\widetilde\rho_n(U)\) by \cite[Propositions~4.1--4.2]{Mit06}.  Equations
\eqref{eq:conjugated-quantum-actions} and
\eqref{eq:R-conjugation-standard} therefore imply
\[
 [R_i,\widetilde\Psi_n(u)]
 =\widetilde D_n[vS_i,\widetilde\rho_n(u)]\widetilde D_n^{-1}=0
 \quad(u\in U).
\]
Let \(\operatorname{pr}_\star:\widetilde V\to\widetilde V\)
be the projection onto \(\mathbb Kv_\star\) along \(V\), so that
\(\operatorname{pr}_\star(v_\star)=v_\star\) and
\(\operatorname{pr}_\star(V)=0\).
Both summands are \(U\)-submodules, so \(\operatorname{pr}_\star\)
is an even \(U\)-module endomorphism.
By the defining action in \eqref{eq:tensor-projection-action},
\[
 \mathcal P_j
 =\operatorname{pr}_\star^{\otimes j}
  \otimes\operatorname{id}_{\widetilde V}^{\otimes(n-j)}
 \qquad(0\le j\le n).
\]
Each factor is an even module map, so \(\mathcal P_j\) commutes
with the action induced by the iterated coproduct:
\([\mathcal P_j,\widetilde\rho_n(u)]=0\) for every \(u\in U\).
Both \(\mathcal P_j\) and \(\widetilde D_n\) are diagonal in the word
basis.  They therefore commute, and
\[
 [\mathcal P_j,\widetilde\Psi_n(u)]
 =\widetilde D_n[\mathcal P_j,\widetilde\rho_n(u)]
      \widetilde D_n^{-1}=0.
\]
The operators \(\mathcal T_i=-R_i\) and \(\mathcal P_j\) generate
the mirabolic image, completing the proof of commutation.
\end{proof}

Denote this representation and its scalar extension by
\[
\begin{aligned}
 \Phi_{m|s,n}&:H^{\mathrm{mir}}_n(q)
 \longrightarrow\operatorname{End}_{\mathbb F}(W_{\mathbb F,n}),\\
 \Phi_{m|s,n,\mathbb K}&:H^{\mathrm{mir}}_{n,\mathbb K}(q)
 \longrightarrow\operatorname{End}_{\mathbb K}(\widetilde V^{\otimes n}).
\end{aligned}
\]
Here \(H^{\mathrm{mir}}_{n,\mathbb K}(q)
=\mathbb K\otimes_{\mathbb F}H^{\mathrm{mir}}_n(q)\), as in the Introduction.

For \(a\ge0\), write
\(H_{a,\mathbb K}(q)=\mathbb K\otimes_{\mathbb F}H_a(q)\), and let
\[
 \Phi^A_{m|s,a}:H_{a,\mathbb K}(q)
 \longrightarrow\operatorname{End}_{\mathbb K}(V^{\otimes a}),
 \qquad T_i\longmapsto-R_i
\]
be the restricted action in type \(A\).  In degree zero this is the scalar
action on \(V^{\otimes0}=\mathbb K\).

For a weight \((\alpha;\beta)\), the generator \(K_j\) acts by
\(v^{-\alpha_j}\) for \(j\le m\) and by \(v^{-\beta_{j-m}}\)
for \(j>m\).  For \(\lambda\in H(m,s)\), the module
\(L_{m|s}(\lambda)\) introduced above is the irreducible covariant
\(U\)-module of highest weight
\[
 \bigl(\lambda_1,\ldots,\lambda_m;
 \max(\lambda'_1-m,0),\ldots,\max(\lambda'_s-m,0)\bigr).
\]
The grading is the one inherited from tensor space;
the highest-weight vector has parity
\(\sum_{i>m}\lambda_i\pmod2\).
We record the standard quantum super Schur--Weyl duality in these conventions.

\begin{proposition}
\label{prop:type-a-schur-weyl}
For every \(a\ge0\), the images \(\Psi_a(U)\) and
\(\Phi^A_{m|s,a}(H_{a,\mathbb K}(q))\) are mutual full centralizers
on \(V^{\otimes a}\), with bimodule decomposition
\begin{equation}
\label{eq:type-a-bimodule-decomposition}
 V^{\otimes a}\cong
 \bigoplus_{\substack{\lambda\vdash a\\\lambda\in H(m,s)}}
 L_{m|s}(\lambda)\otimes_{\mathbb K}
 S^{\lambda'}_{a,\mathrm{std},\mathbb K}.
\end{equation}
\end{proposition}

\begin{proof}
For \(a=0\), both assertions are scalar identities.  For \(a>0\)
and \(m,s>0\), use Moon's quantum parameter \(v^{-1}\).
The coproduct and natural action in \cite[(1.7), (1.12)]{M03}
agree with \(\rho_a\), and the Hecke generators in
\cite[(2.4)]{M03} act by \(v^{-1}S_i\), with parameter \(q^{-1}\).
The image \(\rho_a(U)\) and the algebra generated by the \(S_i\)
are mutual full centralizers by
\cite[Theorem~3.13 and Corollary~3.14]{M03}, and their bimodule
decomposition with highest-weight labels is
\cite[Theorem~5.16]{M03}; see also
\cite[Theorems~4.4 and~5.1]{Mit06}.

We now convert the Hecke action \(u_i\mapsto v^{-1}S_i\) to the
action \(T_i\mapsto-R_i\) used here.
Recall from Lemma~\ref{lem:hecke-specht-labels} that the generators
\(u_i\) of \(H_a(q^{-1})\) determine the isomorphism
\(\phi\circ\psi:H_a(q)\to H_a(q^{-1})\), given by
\(T_i\mapsto-q u_i\).  Pullback along \(\phi\circ\psi\) changes the
Specht label \(\lambda\) to \(\lambda'\), and
\eqref{eq:R-conjugation-standard} gives
\[
 D_a\bigl(-q(v^{-1}S_i)\bigr)D_a^{-1}=-R_i.
\]
Conjugating the quantum action by the same \(D_a\) gives \(\Psi_a\).
The identity \(\Psi_a(u)D_a=D_a\rho_a(u)\) makes \(D_a\) an
isomorphism between these two \(U\)-modules.  It preserves weights
and parity, so the quantum highest-weight label is unchanged.
This proves the assertions when \(m,s>0\).

For \(s=0<m\), use ordinary quantum Schur--Weyl duality
\cite[Remark~2.2 and Theorem~4.4]{Mit06} with the same generator
comparison.  For \(m=0<s\), all root generators are even, and
\[
 K_j^{-1}v_b=v^{\delta_{jb}}v_b,
 \qquad A_i=K_i^{-1}K_{i+1}.
\]
Thus, with \(K_j^{-1}\) as Cartan generators, the coproduct and
natural action are those of the ordinary quantum group at parameter
\(v\).  The displayed highest weight is \(\lambda'\) in this case.
Since every natural basis vector is odd,
\cite[Equation~(2)]{Mit06} identifies \(-S_i\) with the ordinary
natural tensor operator at parameter \(v\).  Multiplying by \(v\)
gives the standard generator for the Hecke parameter \(q=v^2\),
as in \cite[Remark~2.2]{Mit06}.  Thus these generators act by
\(-vS_i\), and \eqref{eq:R-conjugation-standard} gives
\[
 D_a(-vS_i)D_a^{-1}=-R_i.
\]
Ordinary quantum Schur--Weyl duality now gives the stated Specht
label and both centralizer equalities.
\end{proof}

For \(\lambda\in H(m,s)\), the character of
\(L_{m|s}(\lambda)\) is the hook Schur function \(s_\lambda(x|y)\);
see \cite[Theorem~5.1]{BKK00} and \cite[Theorem~6.10(b)]{BR87}.

\subsection{Schur--Weyl duality and the bimodule decomposition}

For a subset \(A\subseteq\{1,\ldots,n\}\), let \(M_A\) be the span
of the tensor words whose non-\(v_\star\) positions are exactly \(A\).
For \(0\le a\le n\), let \(M_a\) be the span of the tensor words with
exactly \(a\) factors different from \(v_\star\).  Thus
\[
 M_a=\bigoplus_{|A|=a}M_A,
 \qquad\widetilde V^{\otimes n}=\bigoplus_{a=0}^nM_a.
\]
For \(A=\{a_1<\cdots<a_a\}\), let
\(\tau_A:V^{\otimes a}\to M_A\) insert the \(r\)-th factor in
position \(a_r\), putting \(v_\star\) in all remaining positions.
It is a linear isomorphism: its inverse deletes the factors in
positions outside \(A\).
The map \(\tau_A\) intertwines the \(U\)-actions \(\rho_a\) and
\(\widetilde\rho_n\), namely
\begin{equation}
\label{eq:tau-standard-intertwining}
 \widetilde\rho_n(u)\tau_A=\tau_A\rho_a(u)\qquad(u\in U).
\end{equation}
Writing \(\varepsilon\) for the counit, the action on each inserted factor is
\(uv_\star=\varepsilon(u)v_\star\), and
\[
 (\varepsilon\otimes\operatorname{id})\Delta(u)
 =(\operatorname{id}\otimes\varepsilon)\Delta(u)=u.
\]
These identities remove the inserted factors from the iterated
coproduct; their even parity introduces no additional signs.
Define
\[
 \ell(A)=\#\{(r,t):r\in A,\ t\in\{1,\ldots,n\}\setminus A,\ r<t\}.
\]
Inserting the stars preserves the inversions between natural letters.
Since \(\star\) is the smallest letter, the additional inversions
are exactly the pairs counted by \(\ell(A)\).  Hence, for every
natural tensor word \(\mathbf b\),
\begin{equation}
\label{eq:tau-diagonal-relation}
 \operatorname{inv}(\tau_A\mathbf b)
 =\operatorname{inv}(\mathbf b)+\ell(A),
 \qquad
 \widetilde D_n\tau_A=v^{-\ell(A)}\tau_AD_a.
\end{equation}
Combining \eqref{eq:tau-standard-intertwining} and
\eqref{eq:tau-diagonal-relation} with
\eqref{eq:conjugated-quantum-actions} gives, for every \(u\in U\),
\[
\begin{aligned}
 \widetilde\Psi_n(u)\tau_A
 &=v^{\ell(A)}\widetilde D_n\widetilde\rho_n(u)\tau_AD_a^{-1}\\
 &=v^{\ell(A)}\widetilde D_n\tau_A\rho_a(u)D_a^{-1}
 =\tau_A\Psi_a(u).
\end{aligned}
\]
The map
\[
 \bigoplus_{|A|=a}V^{\otimes a}\longrightarrow M_a,
 \qquad (x_A)_A\longmapsto\sum_{|A|=a}\tau_A(x_A)
\]
is therefore a \(U\)-module isomorphism.  There are \(\binom na\)
summands, each carrying the action \(\Psi_a\), so this gives
\begin{equation}
\label{eq:rank-u-module-decomposition}
 M_a\cong\mathbb K^{\binom na}\otimes_{\mathbb K}V^{\otimes a},
\end{equation}
with \(U\) acting trivially on the first factor.

We next construct projections onto the individual spaces \(M_a\)
that belong to the quantum action image.  Set
\[
 K_{\mathrm{tot}}=\prod_{j=1}^{m+s}K_j^{-2}\in U.
\]
The defining relations give
\[
 K_jE_iK_j^{-1}=v^{-\delta_{ji}+\delta_{j,i+1}}E_i,
 \qquad
 K_jF_iK_j^{-1}=v^{\delta_{ji}-\delta_{j,i+1}}F_i.
\]
For each \(i\), the exponents sum to zero as \(j\) runs from
\(1\) to \(m+s\).  Thus \(K_{\mathrm{tot}}\) commutes with
\(E_i,F_i\), as well as with all \(K_j\), and is central.
Equation~\eqref{eq:quantum-coproduct-convention} also gives
\[
 \Delta(K_{\mathrm{tot}})
 =K_{\mathrm{tot}}\otimes K_{\mathrm{tot}}.
\]
Its action on a natural basis vector and on the trivial factor is
\[
 K_{\mathrm{tot}}v_b
 =\left(\prod_{j=1}^{m+s}v^{2\delta_{jb}}\right)v_b=qv_b
 \quad(1\le b\le m+s),
 \qquad K_{\mathrm{tot}}v_\star=v_\star.
\]
The coproduct makes these scalars multiply on tensor words.
Every word in \(M_a\) has \(a\) natural factors and \(n-a\) star
factors, so \(\widetilde\rho_n(K_{\mathrm{tot}})\) acts there by
\(q^a\).  Since \(\widetilde D_n\) preserves \(M_a\), conjugation
does not change this scalar:
\[
 \widetilde\Psi_n(K_{\mathrm{tot}})|_{M_a}
 =q^a\operatorname{id}_{M_a}.
\]
As \(q\) is indeterminate, the scalars \(1,q,\ldots,q^n\) are
distinct.  For \(0\le a\le n\), define central elements of \(U\) by
\[
 e_a=\prod_{\substack{0\le b\le n\\b\ne a}}
       \frac{K_{\mathrm{tot}}-q^b}{q^a-q^b}.
\]
On \(M_a\), every factor in this product acts by \(1\); on
\(M_b\) with \(b\ne a\), the factor \(K_{\mathrm{tot}}-q^b\)
acts by zero.  Thus the action images
\(\pi_a=\widetilde\Psi_n(e_a)\) are the projections onto \(M_a\):
\begin{equation}
\label{eq:rank-projections}
 \pi_a|_{M_b}=\delta_{ab}\operatorname{id}_{M_b}
 \qquad(0\le a,b\le n).
\end{equation}
In particular, for any choice of \(u_0,\ldots,u_n\in U\),
\[
 \left.\widetilde\Psi_n\!\left(\sum_{a=0}^ne_au_a\right)\right|_{M_b}
 =\widetilde\Psi_n(u_b)|_{M_b}.
\]
This will allow us to choose the quantum action independently on
each summand \(M_b\).

\begin{theorem}[Super mirabolic Schur--Weyl duality]
\label{thm:super-mirabolic-schur-weyl}
The two tensor actions are mutual full centralizers:
\[
\begin{aligned}
 \operatorname{End}_U(\widetilde V^{\otimes n})
 &=\Phi_{m|s,n,\mathbb K}(H^{\mathrm{mir}}_{n,\mathbb K}(q)),\\
 \operatorname{End}_{H^{\mathrm{mir}}_{n,\mathbb K}(q)}
       (\widetilde V^{\otimes n})&=\widetilde\Psi_n(U).
\end{aligned}
\]
Their bimodule decomposition is
\[
 \widetilde V^{\otimes n}\cong
 \bigoplus_{a=0}^n
 \bigoplus_{\substack{\lambda\vdash a\\\lambda\in H(m,s)}}
 L_{m|s}(\lambda)\otimes_{\mathbb K}
 N^{(n)}_{(\lambda,a),\mathbb K}.
\]
\end{theorem}

\begin{proof}
We first determine the quantum action on each \(M_a\).
For \(\lambda\vdash a\) in \(H(m,s)\), define the vector space
\[
 E_{a,\lambda}=\mathbb K^{\binom na}\otimes_{\mathbb K}
                    S^{\lambda'}_{a,\mathrm{std},\mathbb K}.
\]
Combining the \(U\)-module isomorphism
\eqref{eq:rank-u-module-decomposition} with the type \(A\) decomposition
\eqref{eq:type-a-bimodule-decomposition} gives
\[
 M_a\cong\bigoplus_{\substack{\lambda\vdash a\\\lambda\in H(m,s)}}
             L_{m|s}(\lambda)\otimes_{\mathbb K}E_{a,\lambda}.
\]
Thus \(E_{a,\lambda}\) is the multiplicity space of
\(L_{m|s}(\lambda)\) in this \(U\)-module decomposition, and
\(U\) acts trivially on it.
As recalled in Section~\ref{subsec:irreducible-characters}, the Specht
modules are absolutely irreducible and pairwise nonisomorphic.
In the type \(A\) decomposition, an endomorphism
commuting with the Hecke action therefore acts arbitrarily on each
\(L_{m|s}(\lambda)\) and as the identity on its Specht factor.
Proposition~\ref{prop:type-a-schur-weyl} identifies this centralizer
with the quantum image:
\[
 \Psi_a(U)=
 \bigoplus_{\substack{\lambda\vdash a\\\lambda\in H(m,s)}}
 \operatorname{End}_{\mathbb K}(L_{m|s}(\lambda))
       \otimes 1_{S^{\lambda'}_{a,\mathrm{std},\mathbb K}}.
\]
On \(M_a\), the same operator acts on every copy indexed by \(A\).
Using \eqref{eq:rank-projections} to separate the ranks gives
\begin{equation}
\label{eq:quantum-image-on-multiplicities}
 \widetilde\Psi_n(U)=
 \bigoplus_{a=0}^n\bigoplus_{\substack{\lambda\vdash a\\\lambda\in H(m,s)}}
 \operatorname{End}_{\mathbb K}(L_{m|s}(\lambda))
          \otimes 1_{E_{a,\lambda}}.
\end{equation}
An operator commuting with this quantum image preserves each
summand \(L_{m|s}(\lambda)\otimes E_{a,\lambda}\) and has the form
\(1_{L_{m|s}(\lambda)}\otimes B\) there, with
\(B\in\operatorname{End}_{\mathbb K}(E_{a,\lambda})\).
By Proposition~\ref{prop:super-mirabolic-representation}, every
mirabolic operator commutes with the quantum action.  It therefore
acts on the \(E_{a,\lambda}\)-factor of each summand, making
\(E_{a,\lambda}\) a mirabolic module.

We next prove that \(E_{a,\lambda}\) is simple.
Fix \(1\le a\le n\) and \(\lambda\vdash a\) in \(H(m,s)\),
and put \(A_0=\{n-a+1,\ldots,n\}\).
We begin by identifying the projection image \(P_{n-a}E_{a,\lambda}\).
Among the sets \(A\) of size \(a\), only \(A_0\) has no position
among the first \(n-a\).  The projection formula
\eqref{eq:tensor-projection-action} therefore gives
\[
 \mathcal P_{n-a}\tau_A=\delta_{A,A_0}\tau_A
 \qquad(|A|=a).
\]
On the retained copy, the local action
\eqref{eq:mirabolic-tensor-operator} gives
\[
 \mathcal T_{n-a+i}\tau_{A_0}
 =\tau_{A_0}\Phi^A_{m|s,a}(T_i)
 \qquad(1\le i<a).
\]
The type \(A\) decomposition \eqref{eq:type-a-bimodule-decomposition}
therefore identifies
\[
 P_{n-a}E_{a,\lambda}\cong S^{\lambda'}_{a,\mathrm{std},\mathbb K}
\]
as an \(H_{a,\mathbb K}(q)\)-module, with its \(i\)-th generator
acting by \(T_{n-a+i}\).  These generators commute with \(P_{n-a}\).
We now show that this projection image generates \(E_{a,\lambda}\)
as a mirabolic module.  For \(x\in V\),
\eqref{eq:mirabolic-tensor-operator} gives
\(\mathcal T_i(x\otimes v_\star)=-qv_\star\otimes x\).
Equivalently,
\[
 \mathcal T_i^{-1}(v_\star\otimes x)
       =-q^{-1}x\otimes v_\star,
\]
so the inverse generator moves a natural factor one position to the left.
For \(A=\{a_1<\cdots<a_a\}\), start with the natural factors in
\(A_0\) and move the first one to \(a_1\), then the second to
\(a_2\), and so on.  Each move crosses only a star, so the order
of the natural factors is preserved.  If the successive positions
of these exchanges are \(i_1,\ldots,i_d\), the local formula gives
\[
 \mathcal T_{i_d}^{-1}\cdots\mathcal T_{i_1}^{-1}\tau_{A_0}(x)
 =(-q^{-1})^d\tau_A(x)
 \qquad(x\in V^{\otimes a}).
\]
By \eqref{eq:quantum-image-on-multiplicities}, this mirabolic
operator acts as the identity on \(L_{m|s}(\lambda)\).
The displayed equality therefore says that its action on
\(E_{a,\lambda}\) maps the \(A_0\)-copy of
\(S^{\lambda'}_{a,\mathrm{std},\mathbb K}\) onto its \(A\)-copy,
by the nonzero scalar \((-q^{-1})^d\).  Since \(A\) was arbitrary,
the projection image generates every copy in \(E_{a,\lambda}\).

We combine this generation property with split semisimplicity.
The latter follows from
Propositions~\ref{prop:parameter-inversion-isomorphism}
and~\ref{prop:q-rook-split-decomposition}.
Write \(E_{a,\lambda}=Q_1\oplus\cdots\oplus Q_k\) as a sum of
simple mirabolic modules.  Then
\[
 P_{n-a}E_{a,\lambda}=\bigoplus_{j=1}^kP_{n-a}Q_j
\]
as \(H_{a,\mathbb K}(q)\)-modules.  Since the left side is a simple
Specht module, exactly one summand, say \(P_{n-a}Q_1\), is nonzero.
But \(Q_1\) is stable under the mirabolic algebra and contains
the projection image that generates \(E_{a,\lambda}\).
Consequently \(E_{a,\lambda}=Q_1\), proving simplicity.

We now determine the label of this simple module.  The classification
obtained from Propositions~\ref{prop:q-rook-split-decomposition}
and~\ref{prop:parameter-inversion-isomorphism} gives
\(E_{a,\lambda}\cong N^{(n)}_{(\mu,r),\mathbb K}\)
for some \(\mu\vdash r\).  We already know
\(P_{n-a}E_{a,\lambda}\ne0\), while
\eqref{eq:tensor-projection-action} gives
\(P_{n-a+1}E_{a,\lambda}=0\): there are only \(n-a\) stars in
each tensor word of \(M_a\).
For \(N^{(n)}_{(\mu,r),\mathbb K}\), the projection
formula \eqref{eq:q-rook-projection-action}, unchanged by pullback,
retains tableaux whose \(r\) entries lie in the final available
positions.  There are \(a\) such positions for \(P_{n-a}\) and
\(a-1\) for \(P_{n-a+1}\).  If at least \(r\) positions are
available, choose \(r\) of their labels and replace the entries of
any standard tableau of shape \(\mu\) increasingly by these labels.
This gives a retained basis vector.  If fewer than \(r\) positions
are available, no such tableau exists.  Therefore
\[
 P_{n-a}N^{(n)}_{(\mu,r),\mathbb K}\ne0\iff r\le a,
 \qquad
 P_{n-a+1}N^{(n)}_{(\mu,r),\mathbb K}=0\iff r\ge a.
\]
These two conditions force \(r=a\).  To determine \(\mu\), apply
Corollary~\ref{cor:mirabolic-specht-label} to this simple module:
\[
 S^{\mu'}_{a,\mathrm{std},\mathbb K}
 \cong P_{n-a}N^{(n)}_{(\mu,a),\mathbb K}
 \cong P_{n-a}E_{a,\lambda}
 \cong S^{\lambda'}_{a,\mathrm{std},\mathbb K}.
\]
Distinct partitions label nonisomorphic Specht modules, so
\(\mu=\lambda\).
Hence \(E_{a,\lambda}\cong N^{(n)}_{(\lambda,a),\mathbb K}\).
For \(a=0\), the multiplicity space is one-dimensional.
By \eqref{eq:tensor-projection-action} and
\eqref{eq:mirabolic-tensor-operator}, all \(P_j\) act by \(1\)
and all \(T_i\) act by \(-1\), giving the
empty-partition simple.  This proves the bimodule decomposition.

Finally, we deduce both centralizer equalities from this bimodule
decomposition.  The mirabolic modules just obtained are pairwise
nonisomorphic.  Proposition~\ref{prop:q-rook-split-decomposition},
transported through \(\Theta_{q^{-1}}\) from
Proposition~\ref{prop:parameter-inversion-isomorphism}, identifies
the mirabolic action image as
\[
 \bigoplus_{a=0}^n\bigoplus_{\substack{\lambda\vdash a\\\lambda\in H(m,s)}}
 1_{L_{m|s}(\lambda)}\otimes
       \operatorname{End}_{\mathbb K}(N^{(n)}_{(\lambda,a),\mathbb K}).
\]
Equation~\eqref{eq:quantum-image-on-multiplicities} shows that an
operator commutes with \(U\) exactly when it preserves these
summands and acts only on their \(N\)-factors.  Hence this image equals
\(\operatorname{End}_U(\widetilde V^{\otimes n})\).
Conversely, commuting with all endomorphisms of the \(N\)-factors
leaves precisely the operators on the \(L\)-factors in
\eqref{eq:quantum-image-on-multiplicities}.  This proves the other
centralizer equality.
\end{proof}

Wan established the ordinary mirabolic Schur--Weyl duality and
bimodule decomposition in \cite[Theorem~3.2]{Wan26}.
Theorem~\ref{thm:super-mirabolic-schur-weyl} gives their quantum super
analogues on \(\widetilde V^{\otimes n}\), with \(V\) of dimension
\(m|s\).

\subsection{The bounded-degree quantum Schur superalgebra}
\label{subsec:bounded-degree-quantum-schur-superalgebra}

For \(a\ge0\), define the image algebra of \(U\) on \(V^{\otimes a}\) by
\[
 \mathscr S_q(m|s,a)=\Psi_a(U)
 \subseteq\operatorname{End}_{\mathbb K}(V^{\otimes a}),
 \qquad\mathscr S_q(m|s,0)=\mathbb K.
\]
On the full tensor space set
\[
 \mathscr S_q(m|s;\le n)=\widetilde\Psi_n(U)
 \subseteq\operatorname{End}_{\mathbb K}(\widetilde V^{\otimes n}).
\]
The subscript \(q=v^2\) records the Hecke parameter.  These are ordinary
associative image algebras.  We call \(\mathscr S_q(m|s;\le n)\) the
bounded-degree quantum Schur superalgebra.

\begin{proposition}
\label{prop:bounded-degree-schur-decomposition}
Restriction by rank gives an algebra isomorphism
\[
 \mathscr S_q(m|s;\le n)\cong
 \bigoplus_{a=0}^n\mathscr S_q(m|s,a).
\]
\end{proposition}

\begin{proof}
By \eqref{eq:rank-u-module-decomposition}, on
\(M_a\cong\mathbb K^{\binom na}\otimes V^{\otimes a}\) the quantum
action is \(1\otimes\Psi_a\).  Hence restriction to all
the \(M_a\) corresponds to the map
\[
 \widetilde\Psi_n(u)\longmapsto(\Psi_a(u))_{a=0}^n.
\]
If \(\widetilde\Psi_n(u)=0\), each \(\Psi_a(u)=0\), so the map
is well defined.  Conversely, if every \(\Psi_a(u)=0\), then
\(\widetilde\Psi_n(u)\) vanishes on every \(M_a\) and hence on
their direct sum.  This proves injectivity.  Restriction preserves
addition and multiplication, so it is an algebra homomorphism.

For a tuple \((x_a)_{a=0}^n\), choose \(u_a\in U\) with
\(\Psi_a(u_a)=x_a\).  By \eqref{eq:rank-projections}, the element
\(\sum_a e_au_a\) acts as \(1\otimes x_b\) on each \(M_b\).
Its image therefore maps to the prescribed tuple, proving surjectivity.
\end{proof}

By Theorem~\ref{thm:super-mirabolic-schur-weyl},
\(\mathscr S_q(m|s;\le n)\) and
\(\Phi_{m|s,n,\mathbb K}(H^{\mathrm{mir}}_{n,\mathbb K}(q))\)
are mutual full centralizers.

We now express \(\mathscr S_q(m|s;\le n)\) as the linear dual
of a finite-dimensional supercoalgebra.  The tensor operator in
\cite[Equation~(8.1.1)]{DR11}, at parameter \(v\), is \(S_i^{-1}\).
Equation~\eqref{eq:R-conjugation-standard} gives
\[
 D_a(vS_i^{-1})D_a^{-1}
 =qR_i^{-1}=R_i+(q-1)\operatorname{id}.
\]
Adding a scalar multiple of the identity does not change a centralizer.
Thus Proposition~\ref{prop:type-a-schur-weyl} and
\cite[Corollary~8.4]{DR11} identify \(\mathscr S_q(m|s,a)\)
with the degree-\(a\) quantum Schur superalgebra at parameter \(v\)
by conjugation with \(D_a\).

Let \(\mathcal A_v(m|s,a)\) be the degree-\(a\) subcoalgebra of
Manin's quantum matrix superalgebra \cite[Section~1.2]{Man89}
over \(\mathbb K\), with matrix
generators of degree one; see \cite[Definition~9.1 and Corollary~9.3]{DR11}.
Its linear dual is
\(\mathcal A_v(m|s,a)^*
=\operatorname{Hom}_{\mathbb K}(\mathcal A_v(m|s,a),\mathbb K)\).
For homogeneous functionals \(f,g\) and
\(\Delta(c)=\sum c_{(1)}\otimes c_{(2)}\) with homogeneous terms,
\cite[Equation~(9.4.2)]{DR11} defines its multiplication by
\[
 (fg)(c)=\sum(-1)^{|g||c_{(1)}|}f(c_{(1)})g(c_{(2)}).
\]
By \cite[Theorem~9.8]{DR11}, this dual is isomorphic to
\(\mathscr S_q(m|s,a)\) as an underlying associative algebra.
Combining this with
Proposition~\ref{prop:bounded-degree-schur-decomposition} gives
\begin{equation}
\label{eq:bounded-degree-coalgebra-duality}
 \mathscr S_q(m|s;\le n)
 \cong\bigoplus_{a=0}^n\mathcal A_v(m|s,a)^*
 \cong\left(\bigoplus_{a=0}^n\mathcal A_v(m|s,a)\right)^*.
\end{equation}
The last isomorphism uses the finite direct sum of subcoalgebras,
whose dual has componentwise multiplication.  Degree zero contributes
\(\mathbb K\) on both sides.

\section{Tensor-space annihilators and faithful quotients}
\label{sec:tensor-space-kernels}

We first construct idempotents associated with the smallest rectangle
outside the hook.  We then use
Theorem~\ref{thm:super-mirabolic-schur-weyl} to determine the tensor-space
annihilator and its generators.  Finally, we describe the faithful
quotient action and its centralizer.

\subsection{Jucys--Murphy idempotents and the forbidden rectangle}
\label{subsec:rectangular-jm-idempotents}

We work over \(\mathbb F=\mathbb C(q)\), using the quantum decomposition
over \(\mathbb K=\mathbb C(v)\), where \(q=v^2\) is indeterminate.
Put
\[
        \rho=((s+1)^{m+1}),\qquad r_0=|\rho|=(m+1)(s+1).
\]
For every partition \(\lambda\),
\begin{equation}
\label{eq:nonhook-contains-rectangle}
        \lambda_{m+1}>s
        \quad\Longleftrightarrow\quad
        \rho\subseteq\lambda.
\end{equation}
Indeed, either condition says that the first \(m+1\) rows of
\(\lambda\) all have at least \(s+1\) boxes.  Thus \(\rho\) is the
smallest partition outside \(H(m,s)\).

For each \(r\ge1\), define the Jucys--Murphy elements in \(R_r(q)\) by
\begin{equation}
\label{eq:q-rook-jm-elements}
 \mathsf J_i
 =q^{-(i-1)}
 (\mathsf T_{i-1}\cdots\mathsf T_1)(1-\mathsf P_1)
 (\mathsf T_1\cdots\mathsf T_{i-1}),
 \qquad 1\le i\le r.
\end{equation}
Thus \(\mathsf J_1=1-\mathsf P_1\) and
\(\mathsf J_i=q^{-1}\mathsf T_{i-1}\mathsf J_{i-1}\mathsf T_{i-1}\)
for \(i\ge2\).  For \(\lambda\in\mathscr P_{\le r}\), the action
on the basis \(\{w_L\}\) of \(M^{R,r}_\lambda\) is
\begin{equation}
\label{eq:q-rook-jm-eigenvalues}
 \mathsf J_iw_L=
 \begin{cases}
 q^{\operatorname{ct}_L(i)}w_L,&i\text{ occurs in }L,\\
 0,&i\text{ does not occur in }L,
 \end{cases}
\end{equation}
by \cite[Proposition~3.5]{Hal04}: each \(w_L\) is a nonzero scalar
multiple of Halverson's \(\zeta_L\), so the eigenvalues are unchanged.

By \eqref{eq:q-rook-jm-eigenvalues}, the operators \(\mathsf J_i\)
and \(\mathsf J_j\) are diagonal in the same basis of every
\(M^{R,r}_\lambda\), so their commutator acts by zero on all these
modules.  The action map in
Proposition~\ref{prop:q-rook-split-decomposition} is injective.  Hence
\begin{equation}
\label{eq:q-rook-jm-commuting}
 \mathsf J_i\mathsf J_j=\mathsf J_j\mathsf J_i
 \qquad(1\le i,j\le r).
\end{equation}

For \(1\le i\le r_0\), let
\[
 \mathcal E_i=\{0\}\cup
 \{q^c:c\in\mathbb Z,\ 1-i\le c\le i-1\}.
\]
These are sets of distinct elements of \(\mathbb F\).  For
\(T\in\operatorname{SYT}(\rho)\), define
\begin{equation}
\label{eq:rectangle-jm-primitive-idempotent}
 f_T=
 \prod_{i=1}^{r_0}
 \prod_{\substack{\alpha\in\mathcal E_i\\
                  \alpha\ne q^{\operatorname{ct}_T(i)}}}
 \frac{\mathsf J_i-\alpha}
      {q^{\operatorname{ct}_T(i)}-\alpha}
 \in R_{r_0}(q),
\end{equation}
and put
\begin{equation}
\label{eq:rectangle-jm-central-idempotent}
 f_\rho=\sum_{T\in\operatorname{SYT}(\rho)}f_T.
\end{equation}
The factors commute by \eqref{eq:q-rook-jm-commuting}, and every
denominator is a nonzero element of \(\mathbb F\).

Recall that an idempotent is an element \(e\) satisfying \(e^2=e\).
Two idempotents \(e,e'\) are orthogonal if \(ee'=e'e=0\).
A nonzero idempotent is primitive if it cannot be written as a sum
of two nonzero orthogonal idempotents.
By Proposition~\ref{prop:q-rook-split-decomposition},
\[
 R_{r_0}(q)\cong
 \bigoplus_{\lambda\in\mathscr P_{\le r_0}}
 \operatorname{End}_{\mathbb F}(M^{R,r_0}_\lambda).
\]
The \(\rho\)-block is the two-sided ideal corresponding to the
\(\lambda=\rho\) summand under this isomorphism.  Its identity
acts as the identity on \(M^{R,r_0}_\rho\) and as zero on every
\(M^{R,r_0}_\lambda\) with \(\lambda\ne\rho\).

\begin{lemma}
\label{lem:rectangle-jm-idempotents}
The elements \(f_T\), for \(T\in\operatorname{SYT}(\rho)\), are
pairwise orthogonal primitive idempotents in the \(\rho\)-block of
\(R_{r_0}(q)\), and \(f_\rho\) is the identity of that block.
\end{lemma}

\begin{proof}
Fix \(T\in\operatorname{SYT}(\rho)\).  We first compute the action
of \(f_T\) on every simple module \(M^{R,r_0}_\lambda\).
Let \(\lambda\in\mathscr P_{\le r_0}\), and let \(L\) be any
\(r_0\)-standard rook tableau of shape \(\lambda\), indexing a
basis vector \(w_L\) of this module.  Unlike \(T\), the tableau
\(L\) need not contain all the entries \(1,\ldots,r_0\), and its
shape need not be \(\rho\).

By \eqref{eq:q-rook-jm-eigenvalues}, \(w_L\) is an eigenvector of
each \(\mathsf J_i\).  If \(i\) occurs in \(L\), the boxes with
entries at most \(i\) form a Young diagram with at most \(i\)
boxes.  The box containing \(i\) therefore has content between
\(1-i\) and \(i-1\).  If \(i\) is absent, the eigenvalue is zero.
Thus \(\mathcal E_i\) contains every possible eigenvalue on \(w_L\),
as well as \(q^{\operatorname{ct}_T(i)}\).
For fixed \(i\), the product over \(\alpha\) in
\eqref{eq:rectangle-jm-primitive-idempotent} acts as \(1\) when
the eigenvalue is \(q^{\operatorname{ct}_T(i)}\): each numerator
then equals its denominator.  At any other eigenvalue in
\(\mathcal E_i\), one numerator is zero, so this product acts as
zero.  Hence \(f_Tw_L=w_L\) precisely when the eigenvalues agree
with those on \(w_T\) for every \(i\); otherwise \(f_Tw_L=0\).

We now check that this agreement occurs only for \(L=T\).
It holds when \(L=T\).  If some entry \(i\) is absent from \(L\),
then \eqref{eq:q-rook-jm-eigenvalues} gives \(\mathsf J_iw_L=0\),
whereas its eigenvalue on \(w_T\) is
\(q^{\operatorname{ct}_T(i)}\ne0\).  The corresponding product
over \(\alpha\) therefore annihilates \(w_L\).
It remains to consider \(L\ne T\) containing all entries
\(1,\ldots,r_0\).  Let \(k\) be the first entry occupying different
boxes in \(L\) and \(T\).  The smaller entries occupy the same
boxes and form a partition \(\nu\).  The two boxes containing
\(k\) are distinct addable boxes of \(\nu\).  Such boxes have
distinct contents: for boxes in rows \(i<j\), the difference is
\[
        (\nu_i+1-i)-(\nu_j+1-j)=\nu_i-\nu_j+j-i>0.
\]
Since \(q\) is indeterminate, this implies
\(q^{\operatorname{ct}_L(k)}\ne q^{\operatorname{ct}_T(k)}\).
The product over \(\alpha\) for \(i=k\) therefore annihilates
\(w_L\).  We have proved, on the basis of every simple module, that
\[
 f_Tw_L=
 \begin{cases}
 w_L,&L=T,\\
 0,&L\ne T.
 \end{cases}
\]
Under the \(\mathbb F\)-algebra isomorphism in
Proposition~\ref{prop:q-rook-split-decomposition}, \(f_T\) is
therefore the diagonal matrix unit indexed by \(T\) in the
\(\rho\)-block, with zero component in every other block.
In particular, \(f_T^2=f_T\) and \(f_Tf_{T'}=0\) for \(T\ne T'\).
Its one-dimensional image cannot split into the images of two
nonzero orthogonal idempotents, so \(f_T\) is primitive.
Summing over \(T\in\operatorname{SYT}(\rho)\) gives the identity
on \(M^{R,r_0}_\rho\) and zero on every other simple module.
Thus \(f_\rho\) is the identity of the \(\rho\)-block.
\end{proof}

In rank zero, set \(R_0(q)=H_0^{\mathrm{mir}}(q)=\mathbb F\)
and take \(\Theta_q\) to be the identity.  For \(0\le r\le n\), let
\[
        \iota_{r,n}:R_r(q)\hookrightarrow R_n(q)
\]
be the embedding preserving the indices of all generators; for \(r=0\)
it is the scalar inclusion.  Halverson proves the embedding after
extension to \(\mathbb K\) \cite[Corollary~3.3]{Hal04}.
The defining relations give the homomorphism over \(\mathbb F\),
and injectivity over \(\mathbb K\) implies injectivity over
\(\mathbb F\).

\begin{lemma}
\label{lem:rectangle-tower-support}
For \(n\ge r_0\), \(\lambda\in\mathscr P_{\le n}\), and every
\(T\in\operatorname{SYT}(\rho)\), the element
\(\iota_{r_0,n}(f_T)\) acts nontrivially on
\(M^{R,n}_\lambda\) if and only if \(\rho\subseteq\lambda\).
\end{lemma}

\begin{proof}
Fix \(T\in\operatorname{SYT}(\rho)\), and let \(w_L\) be an
arbitrary tableau basis vector of \(M^{R,n}_\lambda\).
We first determine when \(\iota_{r_0,n}(f_T)\) retains \(w_L\).
The embedding takes each \(\mathsf J_i\), for \(i\le r_0\),
to the element given by the same formula in \(R_n(q)\).
By \eqref{eq:q-rook-jm-eigenvalues} and
\eqref{eq:rectangle-jm-primitive-idempotent}, the factor belonging
to \(i\) acts as \(1\) precisely when \(i\) occurs in \(L\)
with content \(\operatorname{ct}_T(i)\), and acts as \(0\) otherwise.
If all these factors act as \(1\), the entries \(1,\ldots,r_0\)
form a Young subdiagram: every box to the left of or above one of
them has a smaller entry.  As in the proof of
Lemma~\ref{lem:rectangle-jm-idempotents}, their successive contents
identify this subtableau with \(T\).  Consequently,
\[
 \iota_{r_0,n}(f_T)w_L=
 \begin{cases}
 w_L,&\text{the entries \(1,\ldots,r_0\) of \(L\) form \(T\)},\\
 0,&\text{otherwise}.
 \end{cases}
\]
In particular, nonzero action requires \(\rho\subseteq\lambda\).

Conversely, assume \(\rho\subseteq\lambda\).  We construct a
tableau \(L\) retained by this projection, starting with the
prescribed tableau \(T\).  Fill the remaining boxes of \(\lambda\)
from the top row down, from left to right within each row, with
\(r_0+1,\ldots,|\lambda|\).  Each new box is addable: completed
rows above have their lengths in \(\lambda\), and rows below retain
their lengths in \(\rho\).  Denote the resulting row- and
column-strict tableau by \(L\).  It is an \(n\)-standard rook
tableau, since \(|\lambda|\le n\); the labels
\(|\lambda|+1,\ldots,n\) are absent.  The entries
\(1,\ldots,r_0\) of \(L\) form the prescribed tableau \(T\),
so \(\iota_{r_0,n}(f_T)w_L=w_L\ne0\).
\end{proof}

The conclusion of Lemma~\ref{lem:rectangle-tower-support} also holds
for \(\iota_{r_0,n}(f_\rho)\).
Indeed, its summands are orthogonal projections onto the possible
initial tableaux of shape \(\rho\); at least one of these occurs
exactly when \(\rho\subseteq\lambda\).

An argument \(q^{-1}\) on an element, ideal, or representation means
that both the algebra parameter and its rational coefficients are
replaced by their values at \(q^{-1}\).  In particular,
\(f_T(q^{-1}),f_\rho(q^{-1})\in R_{r_0}(q^{-1})\).
We use \(\iota_{r,n}\) for the generator-preserving inclusions at
either parameter.  Define the corresponding mirabolic inclusion by
\[
 \jmath_{r,n}
 =\Theta_{q^{-1}}^{-1}\circ\iota_{r,n}\circ\Theta_{q^{-1}}
 :H_r^{\mathrm{mir}}(q)\hookrightarrow H_n^{\mathrm{mir}}(q),
\]
where each isomorphism is taken at the appropriate rank.  This map is
injective and preserves the generator indices, since both parameter
changes use the same scalar on each Hecke generator.
We also write \(\jmath_{r,n}\) for the generator-preserving inclusion
at parameter \(q^{-1}\).  Define
\begin{equation}
\label{eq:mirabolic-rectangle-idempotent}
 g_T=\Theta_{q^{-1}}^{-1}(f_T(q^{-1})),\qquad
 g_\rho=\Theta_{q^{-1}}^{-1}(f_\rho(q^{-1}))
                    =\sum_{T\in\operatorname{SYT}(\rho)}g_T.
\end{equation}
Thus \(g_T\) is a primitive spectral idempotent in the
\(\rho\)-block of \(H_{r_0}^{\mathrm{mir}}(q)\), and
\(g_\rho\) is the identity of that block.

The modules \(N^{(n)}_{(\lambda,|\lambda|)}\) in
Subsection~\ref{subsec:irreducible-characters} are defined by pullback
along \(\Theta_{q^{-1}}\).  Under their identification with the
parameter-inverted rook modules, \eqref{eq:mirabolic-rectangle-idempotent}
gives, on every corresponding tableau basis vector,
\[
 \jmath_{r_0,n}(g_T)w_L
 =\iota_{r_0,n}\bigl(f_T(q^{-1})\bigr)w_L.
\]
Thus Lemma~\ref{lem:rectangle-tower-support}, with \(q\) replaced by
\(q^{-1}\), gives
\begin{equation}
\label{eq:mirabolic-rectangle-support}
 \jmath_{r_0,n}(g_T)N^{(n)}_{(\lambda,|\lambda|)}\ne0
 \quad\Longleftrightarrow\quad \rho\subseteq\lambda
 \qquad(n\ge r_0,\ |\lambda|\le n).
\end{equation}
There is no change of the partition label in this assertion.
The same equivalence holds for \(g_\rho\), by the
support statement for \(f_\rho\) above.

\subsection{The mirabolic annihilator and its quotient action}
\label{subsec:mirabolic-annihilator}

Let
\[
 I^{\mathrm{mir}}_{m|s,n}(q)=\ker\Phi_{m|s,n}
                  \subseteq H_n^{\mathrm{mir}}(q)
\]
be the annihilator of \(W_{\mathbb F,n}\) for the mirabolic tensor action.

\begin{theorem}[Mirabolic tensor-space annihilator]
\label{thm:super-mirabolic-annihilator}
If \(n<r_0\), then \(I^{\mathrm{mir}}_{m|s,n}(q)=0\).
If \(n\ge r_0\), then, for every \(T\in\operatorname{SYT}(\rho)\),
\begin{equation}
\label{eq:mirabolic-annihilator-principal}
 I^{\mathrm{mir}}_{m|s,n}(q)
 =H_n^{\mathrm{mir}}(q)\,\jmath_{r_0,n}(g_T)\,H_n^{\mathrm{mir}}(q).
\end{equation}
\end{theorem}

\begin{proof}
We first compute the kernel after extending scalars to \(\mathbb K\).
Put \(A=H_n^{\mathrm{mir}}(q)\) and
\(A_{\mathbb K}=\mathbb K\otimes_{\mathbb F}A\).
By Propositions~\ref{prop:q-rook-split-decomposition}
and~\ref{prop:parameter-inversion-isomorphism}, \(A_{\mathbb K}\) is
split semisimple, with simple modules
\(N^{(n)}_{(\lambda,|\lambda|),\mathbb K}\) indexed by
\(\lambda\in\mathscr P_{\le n}\).
Theorem~\ref{thm:super-mirabolic-schur-weyl} shows that such a
module occurs in tensor space precisely when \(\lambda\in H(m,s)\):
its multiplicity space is then the nonzero module \(L_{m|s}(\lambda)\).
A matrix in the block indexed by \(\lambda\) acts only on copies
of the corresponding simple module.  If a copy occurs, that block
acts faithfully on it; if no copy occurs, the whole block acts by
zero.  Thus the isomorphisms in
Propositions~\ref{prop:parameter-inversion-isomorphism}
and~\ref{prop:q-rook-split-decomposition} identify the kernel with
the sum of the factors indexed by \(\lambda\notin H(m,s)\):
\begin{equation}
\label{eq:mirabolic-kernel-wedderburn}
 \ker\Phi_{m|s,n,\mathbb K}
 =\bigoplus_{\substack{\lambda\in\mathscr P_{\le n}\\
                       \lambda_{m+1}>s}}
   \operatorname{End}_{\mathbb K}
      (N^{(n)}_{(\lambda,|\lambda|),\mathbb K}).
\end{equation}
If \(n<r_0\), this sum is empty, so the kernel over \(\mathbb K\)
is zero.

Suppose \(n\ge r_0\).  We now identify this kernel with the
two-sided ideal generated by \(e=\jmath_{r_0,n}(g_T)\).
A nonzero idempotent in a full matrix algebra generates that algebra
as a two-sided ideal: after diagonalizing the idempotent, choose a
diagonal position \(h\) with entry \(1\); the matrix units then
satisfy \(E_{ih}eE_{hj}=E_{ij}\).
Thus an idempotent in a split semisimple algebra generates precisely
the blocks on which it acts nontrivially.
Equations~\eqref{eq:mirabolic-rectangle-support} and
\eqref{eq:nonhook-contains-rectangle} identify the support of \(e\)
with the blocks in \eqref{eq:mirabolic-kernel-wedderburn}.  Hence
\[
 \ker\Phi_{m|s,n,\mathbb K}=A_{\mathbb K}eA_{\mathbb K}.
\]

We descend this equality to \(\mathbb F\).  Scalar extension is
exact, and the natural identification
\[
 \mathbb K\otimes_{\mathbb F}\operatorname{End}_{\mathbb F}(W_{\mathbb F,n})
 \cong\operatorname{End}_{\mathbb K}(\widetilde V^{\otimes n})
\]
identifies the extended representation with \(\Phi_{m|s,n,\mathbb K}\).
Consequently,
\[
 \mathbb K\otimes_{\mathbb F}I^{\mathrm{mir}}_{m|s,n}(q)
 =\ker\Phi_{m|s,n,\mathbb K}.
\]
For \(J=AeA\), expansion of products gives
\[
 \mathbb K\otimes_{\mathbb F}J=A_{\mathbb K}eA_{\mathbb K},
\]
because \((\alpha\otimes a)e(\beta\otimes b)
=\alpha\beta\otimes aeb\).
Write \(I=I^{\mathrm{mir}}_{m|s,n}(q)\).  The subspaces \(I\)
and \(J\) have the same scalar extension.  If \(x\in I\), its
class \(x+J\in A/J\) therefore satisfies
\[
 1\otimes(x+J)=0
 \quad\text{in }\mathbb K\otimes_{\mathbb F}(A/J).
\]
The natural map from an \(\mathbb F\)-vector space to its scalar
extension is injective, so \(x+J=0\), or \(x\in J\).
This proves \(I\subseteq J\).  Applying the same argument with
\(I\) and \(J\) interchanged proves \(I=J\).  The zero-kernel conclusion when \(n<r_0\)
descends by the same injectivity.
\end{proof}

The same proof, using the support of \(g_\rho\), gives
\[
 I^{\mathrm{mir}}_{m|s,n}(q)
 =H_n^{\mathrm{mir}}(q)\,\jmath_{r_0,n}(g_\rho)
                         \,H_n^{\mathrm{mir}}(q)
 \qquad(n\ge r_0).
\]
For \(n\ge r_0\), the \(\rho\)-block is a nonzero block in the
kernel.  The representation is therefore faithful exactly when
\(n<(m+1)(s+1)\).
By the tableau bases in Section~\ref{subsec:irreducible-characters},
\(\dim_{\mathbb F}N^{(n)}_{(\lambda,|\lambda|)}
=\binom n{|\lambda|}f^\lambda\).  Dimensions are preserved by scalar
extension, so \eqref{eq:mirabolic-kernel-wedderburn} gives
\begin{equation}
\label{eq:mirabolic-kernel-dimension}
 \dim_{\mathbb F}I^{\mathrm{mir}}_{m|s,n}(q)
 =\sum_{\substack{\lambda\in\mathscr P_{\le n}\\\lambda_{m+1}>s}}
       \left(\binom n{|\lambda|}f^\lambda\right)^2.
\end{equation}
The empty sum is zero.

Define the quotient over \(\mathbb F\) by
\[
 \overline H^{\mathrm{mir}}_{m|s,n}(q)
 =H_n^{\mathrm{mir}}(q)/I^{\mathrm{mir}}_{m|s,n}(q),
\]
so the first isomorphism theorem gives
\[
 \overline H^{\mathrm{mir}}_{m|s,n}(q)
 \cong\Phi_{m|s,n}\bigl(H_n^{\mathrm{mir}}(q)\bigr).
\]
The induced action of this quotient on \(W_{\mathbb F,n}\) is
therefore faithful.  Write
\(\overline H^{\mathrm{mir}}_{m|s,n,\mathbb K}(q)\)
for its scalar extension to \(\mathbb K\), identified with its
faithful image on \(\widetilde V^{\otimes n}\).

\begin{corollary}
\label{cor:mirabolic-quotient-double-centralizer}
The algebras \(\overline H^{\mathrm{mir}}_{m|s,n,\mathbb K}(q)\)
and \(\mathscr S_q(m|s;\le n)\) are mutual full centralizers on
\(\widetilde V^{\otimes n}\).
\end{corollary}

\begin{proof}
Exactness of scalar extension identifies the mirabolic image with
\(\overline H^{\mathrm{mir}}_{m|s,n,\mathbb K}(q)\) by the first
isomorphism theorem.  An acting algebra and its image have the same
commutant.  Theorem~\ref{thm:super-mirabolic-schur-weyl}
therefore gives both centralizer equalities.
\end{proof}

Proposition~\ref{prop:q-rook-split-decomposition} and
Theorem~\ref{thm:super-mirabolic-schur-weyl} also give
\begin{align}
\label{eq:mirabolic-quotient-dimension}
 \dim_{\mathbb F}\overline H^{\mathrm{mir}}_{m|s,n}(q)
 &=\sum_{a=0}^{n}\sum_{\substack{\lambda\vdash a\\\lambda\in H(m,s)}}
       \left(\binom na f^\lambda\right)^2,\\
\label{eq:mirabolic-schur-dimension}
 \dim_{\mathbb K}\mathscr S_q(m|s;\le n)
 &=\sum_{a=0}^{n}\sum_{\substack{\lambda\vdash a\\\lambda\in H(m,s)}}
       s_\lambda(1^m|1^s)^2.
\end{align}
For the first identity, the surviving mirabolic blocks have sizes
\(\binom na f^\lambda\), and scalar extension preserves dimension.
For the second, \eqref{eq:quantum-image-on-multiplicities} gives one
matrix block \(\operatorname{End}_{\mathbb K}(L_{m|s}(\lambda))\)
for each \((a,\lambda)\); the same matrix acts on every copy of
\(L_{m|s}(\lambda)\) in tensor space.  Its contribution is therefore
\((\dim_{\mathbb K}L_{m|s}(\lambda))^2=s_\lambda(1^m|1^s)^2\),
since the character of \(L_{m|s}(\lambda)\) is \(s_\lambda(x|y)\),
as recalled after
Proposition~\ref{prop:type-a-schur-weyl}.

\section{Super Frobenius, Murnaghan--Nakayama, and Regev-type formulas}
\label{sec:super-frobenius}
\label{sec:character-formulas-from-super-frobenius}

We first derive the super Frobenius formula by taking weighted traces in
Theorem~\ref{thm:super-mirabolic-schur-weyl}.  From this formula, we
then obtain a Murnaghan--Nakayama rule and Regev-type character formulas
for the mirabolic Hecke algebra.

\subsection{A super mirabolic Frobenius formula}
\label{subsec:super-mirabolic-frobenius}

In this subsection we derive a super analogue of Wan's mirabolic Frobenius
formula, using weighted traces.

Recall the commuting weight alphabets \(x=(x_1,\ldots,x_m)\),
\(y=(y_1,\ldots,y_s)\), and the extended alphabet
\(z=(1,x_1,\ldots,x_m)\).  Its first letter has weight one and
records the even trivial summand \(\mathbb Kv_\star\).

For a formal variable \(t\), define the one-cycle functions
\(\mathsf q_d(x|y;q)\), \(d\ge1\), by
\begin{equation}
\label{eq:super-mirabolic-one-cycle-series}
 1+(q-1)\sum_{d\ge1}\mathsf q_d(x|y;q)t^d
 =\frac{1+qt}{1+t}
  \prod_{i=1}^m\frac{1+qx_it}{1+x_it}
  \prod_{j=1}^s\frac{1-y_jt}{1-qy_jt}.
\end{equation}
These functions have the bicomposition expression
\begin{equation}
\label{eq:super-mirabolic-bicomposition-expansion}
\begin{aligned}
        \mathsf q_d(x|y;q)
        &=
        (-1)^{d-1}
        \sum_{(\alpha;\beta)\vDash d}
        M_\alpha(z)M_\beta(y)
        (-q)^{|\beta|-\ell(\beta)}
        (1-q)^{\ell(\alpha)+\ell(\beta)-1}.
\end{aligned}
\end{equation}
To verify this expression, expand each single-letter factor using
\[
 \frac{1+qu}{1+u}=1+(q-1)\sum_{r\ge1}(-1)^{r-1}u^r,
 \qquad
 \frac{1-u}{1-qu}=1+(q-1)\sum_{r\ge1}q^{r-1}u^r,
\]
where \(u\) is a formal variable.  Choosing the nonconstant factors
records the distinct even and odd letters, while their exponents give
\((\alpha;\beta)\).  After division by \(q-1\), their coefficient is
\[
 (-1)^{|\alpha|-\ell(\alpha)}
 q^{|\beta|-\ell(\beta)}
 (q-1)^{\ell(\alpha)+\ell(\beta)-1},
\]
which is the coefficient in
\eqref{eq:super-mirabolic-bicomposition-expansion} because
\(|\alpha|+|\beta|=d\).

For a composition \(\mu=(\mu_1,\ldots,\mu_\ell)\), set
\[
        \mathsf q_\mu(x|y;q)
        =
        \prod_{i=1}^{\ell}
        \mathsf q_{\mu_i}(x|y;q).
\]
For the empty composition, we set
\[
        \mathsf q_\varnothing(x|y;q)=1.
\]

For \(d\ge1\), recall the cycle
\(\gamma_{(d)}=s_1\cdots s_{d-1}\in\mathfrak S_d\), with
\(\gamma_{(1)}=1\).  Let \(\mathcal T_{\gamma_{(d)}}\) be the action of
the Hecke element \(T_{\gamma_{(d)}}\) on \(\widetilde V^{\otimes d}\)
under \(T_i\mapsto\mathcal T_i=-R_i\).
Let \(W_{x|y}\) be the diagonal operator
defined by
\[
        W_{x|y}v_\star=v_\star,\qquad
        W_{x|y}v_i=x_i v_i\ (1\le i\le m),\qquad
        W_{x|y}v_{m+j}=y_jv_{m+j}\ (1\le j\le s).
\]
The operator is understood after extending scalars to the polynomial ring
in the weight variables.

\begin{lemma}
\label{lem:one-cycle-weighted-trace}
For every \(d\ge1\), the weighted trace is
\begin{equation}
\label{eq:one-cycle-weighted-trace}
        \operatorname{Tr}_{\widetilde V^{\otimes d}}
        \left(
        W_{x|y}^{\otimes d}\mathcal T_{\gamma_{(d)}}
        \right)
        =
        \mathsf q_d(x|y;q).
\end{equation}
\end{lemma}

\begin{proof}
Since \(\mathcal T_i=-R_i\), we have
\[
 \mathcal T_{\gamma_{(d)}}=(-1)^{d-1}R_1\cdots R_{d-1}.
\]
We compute the diagonal coefficients of this product on tensor words.
For \(d=1\), the product is the identity.  For \(d\ge2\), the rightmost
operator acts first.  After \(R_{d-1}\) has acted, position \(d\) will
never be changed again.  Thus a diagonal contribution cannot use the swap
term when the last two letters differ.  Repeating this argument from right
to left shows that every unequal adjacent pair must use the non-swap term.
Such a term exists precisely when the word is weakly increasing.  Identical
even letters contribute \(1\), identical odd letters contribute \(-q\), and
a strict increase contributes \(1-q\), by
\eqref{eq:mirabolic-tensor-operator}.

We now group the contributing words by their distinct even and odd
letters and their run lengths.  This gives
a bicomposition \((\alpha;\beta)\vDash d\).
There are \(|\beta|-\ell(\beta)\) adjacent equal odd pairs,
since an odd run of length \(r\) contributes \(r-1\) such pairs.
The \(\ell(\alpha)+\ell(\beta)\) distinct letters give
\(\ell(\alpha)+\ell(\beta)-1\) strict increases.
Summing the weights over the choices of these letters gives
\(M_\alpha(z)M_\beta(y)\), while the local \(R\)-coefficients contribute
\((-q)^{|\beta|-\ell(\beta)}
(1-q)^{\ell(\alpha)+\ell(\beta)-1}\).  The extra trivial vector enlarges the
even alphabet to \(z=(1,x_1,\ldots,x_m)\).
Including the sign \((-1)^{d-1}\), we obtain
\(\mathsf q_d(x|y;q)\) by
\eqref{eq:super-mirabolic-bicomposition-expansion}.
\end{proof}

\begin{theorem}[Super mirabolic Frobenius formula]
\label{thm:super-mirabolic-frobenius}
Let \(\mu\) be a composition of \(k\le n\).  The hook Schur expansion is
\[
        \mathsf q_\mu(x|y;q)
        =
        \sum_{a=0}^{k}
        \sum_{\substack{\lambda\vdash a\\ \lambda_{m+1}\le s}}
        \chi^{(n)}_{\lambda}
        \left(\widehat T^{(n)}_\mu\right)
        s_\lambda(x|y).
\]
\end{theorem}

\begin{proof}
For \(h\in H^{\mathrm{mir}}_n(q)\) and the diagonal operator
\(W_{x|y}\) in Lemma~\ref{lem:one-cycle-weighted-trace}, set
\[
        \operatorname{Ch}_{x|y}(h)
        =
        \operatorname{Tr}_{\widetilde V^{\otimes n}}
        \left(
        W_{x|y}^{\otimes n}
        \Phi_{m|s,n,\mathbb K}(h)
        \right).
\]
We first compute the weighted trace using the bimodule decomposition in
Theorem~\ref{thm:super-mirabolic-schur-weyl}.
Let \(\alpha_j\) be the multiplicity of \(v_j\) in a tensor word,
for \(1\le j\le m+s\), so \(\alpha_j\ge0\) and
\(\sum_j\alpha_j\le n\).
By \eqref{eq:quantum-coproduct-convention} and the natural action,
\(K_j\) acts by \(v^{-\alpha_j}\) on the space with these
multiplicities.  Conjugation by the diagonal map \(\widetilde D_n\)
does not change this action.  Since the powers \(1,v^{-1},\ldots,v^{-n}\)
are distinct, these spaces are exactly the corresponding common
eigenspaces of the \(K_j\).
The \(U\)-module isomorphism in
Theorem~\ref{thm:super-mirabolic-schur-weyl} preserves these weight
spaces.  In each summand
\(L_{m|s}(\lambda)\otimes_{\mathbb K}N^{(n)}_{(\lambda,a),\mathbb K}\),
the weight space is therefore the corresponding weight space of
\(L_{m|s}(\lambda)\) tensored with the whole \(N\)-factor.
On it, \(W_{x|y}^{\otimes n}\) multiplies by
\(x_1^{\alpha_1}\cdots x_m^{\alpha_m}
y_1^{\alpha_{m+1}}\cdots y_s^{\alpha_{m+s}}\).
Thus it acts on the \(L\)-factor by its weight monomials and as the
identity on the \(N\)-factor.  Its trace on the \(L\)-factor is
\(s_\lambda(x|y)\), by \cite[Theorem~5.1]{BKK00} and
\cite[Theorem~6.10(b)]{BR87}.  The mirabolic element acts
on the \(N\)-factor, so taking the trace gives
\begin{equation}
\label{eq:frobenius-weighted-trace-expansion}
        \operatorname{Ch}_{x|y}
        \left(\widehat T^{(n)}_\mu\right)
        =
        \sum_{a=0}^{n}
        \sum_{\substack{\lambda\vdash a\\ \lambda_{m+1}\le s}}
        \chi^{(n)}_{\lambda}
        \left(\widehat T^{(n)}_\mu\right)
        s_\lambda(x|y).
\end{equation}
We next compute the same trace directly on tensor words.
By \eqref{eq:mirabolic-standard-element},
\[
 \Phi_{m|s,n,\mathbb K}(\widehat T^{(n)}_\mu)
 =\mathcal P_{n-k}\mathcal T_{\gamma_{\mu\uparrow n}}.
\]
The projection formula \eqref{eq:tensor-projection-action} shows
that \(\mathcal P_{n-k}\) kills every word whose first \(n-k\)
positions are not all \(v_\star\).  Its image is
\[
 \mathbb K v_\star^{\otimes(n-k)}
 \otimes\widetilde V^{\otimes\mu_1}\otimes\cdots
 \otimes\widetilde V^{\otimes\mu_\ell}.
\]
The cycle operators act only on the final \(k\) positions, and
\(W_{x|y}^{\otimes n}\) is diagonal.  They preserve the image and
kernel of \(\mathcal P_{n-k}\), so the weighted operator vanishes
on the kernel.  On the image, it is
\[
 1\otimes\bigotimes_{j=1}^{\ell}
 \left(W_{x|y}^{\otimes\mu_j}
       \mathcal T_{\gamma_{(\mu_j)}}\right).
\]
The first factor has trace \(1\), and the trace of a tensor
product is the product of the traces.  Thus
Lemma~\ref{lem:one-cycle-weighted-trace} gives
\begin{equation}
\label{eq:frobenius-direct-trace}
        \operatorname{Ch}_{x|y}
        \left(\widehat T^{(n)}_\mu\right)
        =
        \prod_{i=1}^{\ell}
        \mathsf q_{\mu_i}(x|y;q)
        =
        \mathsf q_\mu(x|y;q).
\end{equation}
It remains to determine which ranks contribute.
Equating \eqref{eq:frobenius-weighted-trace-expansion} and
\eqref{eq:frobenius-direct-trace} gives the expansion with
\(0\le a\le n\).  If \(a>k\), the rank-\(a\) summand contains
fewer than \(n-k\) stars, so \(P_{n-k}\) annihilates it.  Each of its
simple mirabolic factors therefore has zero character on
\(\widehat T^{(n)}_\mu\).  This gives the truncation to \(a\le k\),
including \(k=0\), when only the empty partition contributes.

Finally, \(\mathsf q_d(x|y;q)\) and \(s_\lambda(x|y)\) have
coefficients in \(\mathbb F\), and the values
\(\chi^{(n)}_\lambda(\widehat T^{(n)}_\mu)\) lie in \(\mathbb F\).
Injectivity of
\(\mathbb F[x,y]\to\mathbb K[x,y]\) descends the identity to
\(\mathbb F\).
\end{proof}

\begin{remark}
\label{rem:stable-hook-comparisons}
By \cite[Lemma~6.4]{BR87} and homogeneity, the polynomials
\(s_\lambda(x|y)\), \(\lambda\in H(m,s)\), are linearly independent,
so coefficients may be compared.
For a fixed coefficient comparison, we may take \(m,s\) large enough
that all relevant partitions lie in \(H(m,s)\); the character values
\(\chi^{(n)}_\lambda(\widehat T^{(n)}_\mu)\) in
Theorem~\ref{thm:super-mirabolic-frobenius} are independent of this choice.
\end{remark}

\subsection{A Murnaghan--Nakayama rule for the mirabolic Hecke algebra}
\label{sec:mirabolic-mn-super-schur-specialization}

We derive a Murnaghan--Nakayama rule for the mirabolic Hecke algebra by
combining Theorem~\ref{thm:super-mirabolic-frobenius} with a multiplication formula for
hook Schur functions.  Its coefficients are specializations of skew Schur
functions at a virtual alphabet.

For a formal variable \(t\), the dual Cauchy identity gives the following
formal power series identity for the one-cycle functions.

\begin{lemma}
\label{lem:one-variable-super-mirabolic-kernel}
The one-cycle generating series has the expansion
\begin{equation}
\label{eq:one-variable-mirabolic-super-kernel-mn}
        1+(q-1)\sum_{r\ge1}
        \mathsf q_r(x|y;q)t^r
        =
        \sigma_{-t}[1-q]
        \sum_{\lambda\in H(m,s)}
        s_\lambda(x|y)\,
        s_{\lambda'}[(q-1)t].
\end{equation}
\end{lemma}

\begin{proof}
In the defining product \eqref{eq:super-mirabolic-one-cycle-series},
the factor for the extra even letter \(1\) is
\((1+qt)/(1+t)=\sigma_{-t}[1-q]\).
Since \(p_d[(q-1)t]=(q^d-1)t^d\), the dual Cauchy identity gives
\[
 \sum_{\lambda\in H(m,s)}
 s_\lambda(x|y)s_{\lambda'}[(q-1)t]
 =\exp\left(
 \sum_{d\ge1}\frac{(-1)^{d-1}}{d}
 p_d(x|y)(q^d-1)t^d\right).
\]
Substituting the definition of \(p_d(x|y)\) separates the even
and odd variables and gives
\[
        \prod_{i=1}^m\frac{1+qx_it}{1+x_it}
        \prod_{j=1}^s\frac{1-y_jt}{1-qy_jt}
        =
        \sum_{\lambda\in H(m,s)}
        s_\lambda(x|y)s_{\lambda'}[(q-1)t].
\]
Multiplying by the factor for the trivial summand proves
\eqref{eq:one-variable-mirabolic-super-kernel-mn}.
\end{proof}

For partitions \(\mu\subseteq\lambda\), define
\[
        d(\lambda/\mu)=|\lambda|-|\mu|.
\]
The next lemma expresses \(\mathsf q_r(x|y;q)s_\mu(x|y)\) as a
sum of hook Schur functions.  For \(r\ge1\), define its coefficients by
\begin{equation}
\label{eq:def-mirabolic-mn-coefficient}
        \mathcal R_{\lambda/\mu}(r;q)
        =
        \frac{1}{q-1}
        [t^r]\,
        \sigma_{-t}[1-q]\,
        s_{\lambda'/\mu'}[(q-1)t].
\end{equation}
We set \(\mathcal R_{\lambda/\mu}(r;q)=0\) if
\(\mu\not\subseteq\lambda\).

\begin{lemma}
\label{lem:mirabolic-skew-schur-specialization}
For every partition \(\mu\), the product has the expansion
\begin{equation}
\label{eq:mirabolic-skew-schur-specialization}
        \mathsf q_r(x|y;q)
        s_\mu(x|y)
        =
        \sum_{\lambda\in H(m,s)}
        \mathcal R_{\lambda/\mu}(r;q)
        s_\lambda(x|y),
        \qquad r\ge1.
\end{equation}
\end{lemma}

\begin{proof}
We determine the coefficient of each \(s_\lambda(x|y)\) after
multiplying \eqref{eq:one-variable-mirabolic-super-kernel-mn}
by \(s_\mu(x|y)\).
Let \(A\) be a virtual alphabet, and let \(c_{\mu\nu}^{\lambda}\)
denote the Littlewood--Richardson coefficients.
The product and skew Schur expansions give
\[
 s_\mu(x|y)s_\nu(x|y)
 =\sum_{\lambda\in H(m,s)}
 c_{\mu\nu}^{\lambda}s_\lambda(x|y),
 \qquad
 s_{\lambda'/\mu'}[A]
 =\sum_\nu c_{\mu\nu}^{\lambda}s_{\nu'}[A].
\]
The second identity uses
\(c_{\mu\nu}^{\lambda}=c_{\mu'\nu'}^{\lambda'}\).
Hence multiplying
\eqref{eq:one-variable-mirabolic-super-kernel-mn}
by \(s_\mu(x|y)\) gives
\[
 s_\mu(x|y)\left(1+(q-1)\sum_{r\ge1}
 \mathsf q_r(x|y;q)t^r\right)
 =
 \sigma_{-t}[1-q]\sum_{\lambda\in H(m,s)}
 s_\lambda(x|y)s_{\lambda'/\mu'}[(q-1)t].
\]
Taking the coefficient of \(t^r\), with \(r\ge1\), removes the
constant term on the left.  Dividing by \(q-1\) and using
\eqref{eq:def-mirabolic-mn-coefficient} proves the result.
\end{proof}

For \(\mu\subseteq\lambda\), the coefficient
\(\mathcal R_{\lambda/\mu}(r;q)\) defined in
\eqref{eq:def-mirabolic-mn-coefficient} can be evaluated explicitly.
Since \(s_{\lambda'/\mu'}[(q-1)t]\) is homogeneous of degree
\(d(\lambda/\mu)\) in \(t\), we have
\[
        s_{\lambda'/\mu'}[(q-1)t]
        =
        t^{d(\lambda/\mu)}
        s_{\lambda'/\mu'}[q-1].
\]
Moreover,
\[
        \sigma_{-t}[1-q]
        =
        \frac{1+qt}{1+t}
        =
        1+(q-1)\sum_{j\ge1}(-1)^{j-1}t^j.
\]
Hence \(\mathcal R_{\lambda/\mu}(r;q)=0\) unless
\(0\le d(\lambda/\mu)\le r\), and
\begin{equation}
\label{eq:explicit-mirabolic-mn-coefficient}
        \mathcal R_{\lambda/\mu}(r;q)
        =
        \begin{cases}
        \displaystyle
        \frac{1}{q-1}\,
        s_{\lambda'/\mu'}[q-1],
        & d(\lambda/\mu)=r,\\[10pt]
        \displaystyle
        (-1)^{r-d(\lambda/\mu)-1}
        s_{\lambda'/\mu'}[q-1],
        & 0\le d(\lambda/\mu)<r,\\[6pt]
        0,
        & \text{otherwise.}
        \end{cases}
\end{equation}
In particular,
\[
        \mathcal R_{\mu/\mu}(r;q)=(-1)^{r-1}.
\]

\begin{theorem}
\label{thm:mirabolic-murnaghan-nakayama-rule}
Let \(r\ge1\), let \(\nu\) be a composition with \(r+|\nu|\le n\), and let
\(\lambda\vdash a\) with \(0\le a\le n\).
The character values satisfy the recurrence
\begin{equation}
\label{eq:mirabolic-mn-rule}
        \chi^{(n)}_{\lambda}
        \left(
        \widehat T^{(n)}_{(r,\nu)}
        \right)
        =
        \sum_{b=0}^{a}
        \sum_{\mu\vdash b}
        \mathcal R_{\lambda/\mu}(r;q)\,
        \chi^{(n)}_{\mu}
        \left(
        \widehat T^{(n)}_{\nu}
        \right).
\end{equation}
\end{theorem}

\begin{proof}
Choose \(m,s\) such that \(\lambda\in H(m,s)\).  Then every
\(\mu\subseteq\lambda\) also lies in \(H(m,s)\), and coefficient comparison in
the hook Schur expansion is legitimate by
Remark~\ref{rem:stable-hook-comparisons}.
Apply Theorem~\ref{thm:super-mirabolic-frobenius} first to \(\mathsf q_\nu\), multiply by
\(\mathsf q_r\), and use
Lemma~\ref{lem:mirabolic-skew-schur-specialization}:
\[
\begin{aligned}
        \mathsf q_{(r,\nu)}(x|y;q)
        &=
        \sum_{\substack{\eta\in H(m,s)\\|\eta|\le r+|\nu|}}
        \left(
        \sum_{b=0}^{|\nu|}
        \sum_{\substack{\mu\vdash b\\\mu\in H(m,s)}}
        \mathcal R_{\eta/\mu}(r;q)
        \chi^{(n)}_{\mu}
        \left(
        \widehat T^{(n)}_{\nu}
        \right)
        \right)
        s_\eta(x|y).
\end{aligned}
\]
Comparing the coefficient of \(s_\lambda(x|y)\) with the expansion of
\(\mathsf q_{(r,\nu)}(x|y;q)\) in
Theorem~\ref{thm:super-mirabolic-frobenius} gives
\eqref{eq:mirabolic-mn-rule} when \(a\le r+|\nu|\).
Terms with \(\mu\not\subseteq\lambda\) or \(b>a\) vanish by
the definition of \(\mathcal R\); terms with \(b>|\nu|\) vanish by
Theorem~\ref{thm:super-mirabolic-frobenius}.

If \(a>r+|\nu|\), the left side is zero by
Theorem~\ref{thm:super-mirabolic-frobenius}.
On the right, only \(|\mu|\le|\nu|\) can contribute, but then
\(|\lambda|-|\mu|>r\), so every coefficient \(\mathcal R\) vanishes by
\eqref{eq:explicit-mirabolic-mn-coefficient}.
\end{proof}

\begin{remark}
\label{rem:wan-mn-comparison}
The coefficients \(\mathcal R_{\lambda/\mu}(r;q)\) have a combinatorial
interpretation.  For a generalized ribbon \(\lambda/\mu\), meaning a
skew diagram with no \(2\times2\) block, let \(r_\xi\) and \(c_\xi\)
be the numbers of rows and columns in each connected component \(\xi\).
By \cite[Lemma~6.7(a), Equation~(6.6)]{HR95} and
\cite[Lemma~4.1, Equation~(4.4)]{JL24}, applied to the transposed diagram,
\[
 s_{\lambda'/\mu'}[q-1]
 =\prod_\xi\bigl((q-1)(-1)^{c_\xi-1}q^{r_\xi-1}\bigr),
\]
with the empty product equal to \(1\); the specialization is zero
if \(\lambda/\mu\) is not a generalized ribbon.
See also \cite[Lemma~4.5, Equation~(4.3.7)]{JL26}.
Substitution into \eqref{eq:explicit-mirabolic-mn-coefficient} gives
the corresponding formula for \(\mathcal R_{\lambda/\mu}(r;q)\).

Wan's derivation of the mirabolic Murnaghan--Nakayama rule is based on
the \(q\)-rook rule \cite[Proposition~5.3 and Theorem~5.4]{Wan26}.
Our proof of Theorem~\ref{thm:mirabolic-murnaghan-nakayama-rule} gives
a self-contained derivation directly from the super Frobenius formula
(Theorem~\ref{thm:super-mirabolic-frobenius}).
\end{remark}

\subsection{A Regev-type formula}
\label{sec:regev-type}

Specializing Theorem~\ref{thm:super-mirabolic-frobenius} at \(x_i=y_j=1\) gives the
following Regev-type identity.

For a partition \(\lambda\in H(m,s)\), set
\[
        c^{m|s}_\lambda=s_\lambda(1^m|1^s),
\]
which equals \(\dim_{\mathbb K}L_{m|s}(\lambda)\).
The one-part factors will be the specializations
\(\mathsf q_r(1^m|1^s;q)\).  For \(r\ge1\), put
\[
        F_r(m,s;q)
        =
        (-1)^{r-1}
        \sum_{(\alpha;\beta)\vDash r}
        \binom{m+1}{\ell(\alpha)}
        \binom{s}{\ell(\beta)}
        (-q)^{|\beta|-\ell(\beta)}
        (1-q)^{\ell(\alpha)+\ell(\beta)-1}.
\]

\begin{theorem}
\label{thm:super-mirabolic-theorem-B}
Let \(\mu=(\mu_1,\ldots,\mu_\ell)\vdash k\le n\).
The weighted sum of character values factors as
\[
        \sum_{a=0}^{k}
        \sum_{\substack{\lambda\vdash a\\ \lambda\in H(m,s)}}
        c^{m|s}_\lambda\,
        \chi^{(n)}_{\lambda}
        \left(\widehat T^{(n)}_\mu\right)
        =
        \prod_{j=1}^{\ell}
        F_{\mu_j}(m,s;q).
\]
\end{theorem}

\begin{proof}
Set \(x_i=y_j=1\) in Theorem~\ref{thm:super-mirabolic-frobenius}.  Then
\(s_\lambda(x|y)=c^{m|s}_\lambda\), while
\[
 M_\alpha(1^{m+1})=\binom{m+1}{\ell(\alpha)},\qquad
 M_\beta(1^s)=\binom{s}{\ell(\beta)}.
\]
Substitution in \eqref{eq:super-mirabolic-bicomposition-expansion}
gives \(\mathsf q_r(1^m|1^s;q)=F_r(m,s;q)\).
Multiplicativity over the parts of \(\mu\) proves the character identity.
\end{proof}

Specializing \eqref{eq:super-mirabolic-one-cycle-series} at the same
alphabet gives the generating series for the one-part factors:
\[
        1+(q-1)\sum_{r\ge1}F_r(m,s;q)t^r
        =
        \left(\frac{1+qt}{1+t}\right)^{m+1}
        \left(\frac{1-t}{1-qt}\right)^s.
\]

For \(h\in H^{\mathrm{mir}}_n(q)\), define the tensor character by
\(\chi_{m|s,n}(h)=\operatorname{Tr}_{W_{\mathbb F,n}}
(\Phi_{m|s,n}(h))\).

\begin{corollary}
\label{thm:super-mirabolic-theorem-A}
For \(\mu=(\mu_1,\ldots,\mu_\ell)\vdash k\le n\), the tensor character value is
\[
        \chi_{m|s,n}(\widehat T^{(n)}_\mu)
        =\prod_{j=1}^{\ell}F_{\mu_j}(m,s;q).
\]
\end{corollary}

\begin{proof}
Set every weight variable equal to \(1\) in
\eqref{eq:frobenius-direct-trace}.  Then \(W_{x|y}^{\otimes n}\)
is the identity, so the left side is the trace of
\(\Phi_{m|s,n,\mathbb K}(\widehat T^{(n)}_\mu)\).
This equals \(\chi_{m|s,n}(\widehat T^{(n)}_\mu)\), since traces
are preserved by scalar extension.  Each factor on the right is
\(\mathsf q_{\mu_j}(1^m|1^s;q)=F_{\mu_j}(m,s;q)\), as shown in
Theorem~\ref{thm:super-mirabolic-theorem-B}.
\end{proof}

For the next corollary, specialize to \(m=s=1\), and let
\[
        \mu=(\mu_1,\ldots,\mu_\ell)\vdash k\le n.
\]
For \(d\ge 1\), set
\[
        A_d(q)
        =
        2d+1
        +
        4\sum_{j=1}^{d-1}(d-j)(-q)^j.
\]
An ordinary nonempty hook partition has shape \((a-r,1^r)\), with
\(a\ge1\) and \(0\le r<a\).  We also include the empty partition
among the hook labels in the next sum.

\begin{corollary}
\label{cor:super-mirabolic-theorem-B-hook}
The sum of character values indexed by hook partitions, including the empty partition, is
\[
        \sum_{a=0}^{k}
        \sum_{\substack{\lambda\vdash a\\ \lambda\text{ is a hook}}}
        \chi^{(n)}_{\lambda}
        \left(\widehat T^{(n)}_\mu\right)
        =
        \frac{(-1)^{k-\ell(\mu)}}{2}
        \left(
        \prod_{i=1}^{\ell(\mu)} A_{\mu_i}(q)
        +1
        \right).
\]
\end{corollary}

\begin{proof}
We first determine the multiplicities \(c^{1|1}_\lambda\) in
Theorem~\ref{thm:super-mirabolic-theorem-B}.
Write \(x=x_1\) and \(y=y_1\) in this proof.  The \((1,1)\)-hook
partitions are \(\varnothing\) and
\((a-r,1^r)\), with \(a\ge1\) and \(0\le r<a\).
The factorization in \cite[Theorem~6.20]{BR87} gives
\[
        s_{(a-r,1^r)}(x|y)
        =x^{a-r-1}y^r(x+y),
\]
so \(c^{1|1}_\varnothing=1\) and
\(c^{1|1}_{(a-r,1^r)}=2\).

We next evaluate the factors \(F_d(1,1;q)\).  At \(x=y=1\),
\eqref{eq:super-mirabolic-one-cycle-series} gives
\[
1+(q-1)\sum_{d\ge1}F_d(1,1;q)t^d
=\frac{(1+qt)^2(1-t)}{(1+t)^2(1-qt)}.
\]
Subtracting \(1\), dividing by \(q-1\), and expanding gives
\[
\sum_{d\ge1}F_d(1,1;q)t^d
=\frac{t(3+(1+q)t-qt^2)}{(1+t)^2(1-qt)}
=\sum_{d\ge1}(-1)^{d-1}A_d(q)t^d.
\]
For the last equality, directly summing the definition of \(A_d(q)\) yields
\[
\sum_{d\ge1}(-1)^{d-1}A_d(q)t^d
=\frac{t(3+t)}{(1+t)^2}
 +\frac{4qt^2}{(1+t)^2(1-qt)}.
\]
Thus \(F_d(1,1;q)=(-1)^{d-1}A_d(q)\).

Finally, we separate the empty-partition term, whose multiplicity is \(1\).
The rank-zero module is spanned by \(v_\star^{\otimes n}\).
By \eqref{eq:mirabolic-tensor-operator} and
\eqref{eq:tensor-projection-action}, every \(T_i\) acts there by \(-1\)
and every \(P_j\) by \(1\).  Since the cycle product in
\eqref{eq:mirabolic-standard-element} has \(k-\ell(\mu)\) generators,
\[
\chi^{(n)}_{\varnothing}(\widehat T^{(n)}_\mu)
=(-1)^{k-\ell(\mu)}.
\]
Let \(\Sigma\) denote the character sum in the corollary.
Applying Theorem~\ref{thm:super-mirabolic-theorem-B} and using the above
multiplicities gives
\[
 2\Sigma-(-1)^{k-\ell(\mu)}
 =(-1)^{k-\ell(\mu)}\prod_iA_{\mu_i}(q).
\]
Solving for \(\Sigma\)
proves the formula.
\end{proof}

\section{A super mirabolic RSK correspondence and Roichman formulas}
\label{sec:mirabolic-rsk}

We first express the functions \(\mathsf q_\mu(x|y;q)\) in
Theorem~\ref{thm:super-mirabolic-frobenius} as weighted sums over words.
We then construct the super mirabolic RSK correspondence and rewrite
these sums in terms of tableaux.  Finally, we obtain a Roichman formula
for irreducible mirabolic characters.

Throughout this section we use the signed alphabet
\[
        \mathcal A_{m|s}
        =
        \{\star,x_1,\ldots,x_m,y_1,\ldots,y_s\}.
\]
We declare
\[
        |\star|=|x_i|=\bar 0,
        \qquad
        |y_j|=\bar 1.
\]
For each letter \(a\), write \(x_a\) for its weight.  We use the same
symbols \(x_i,y_j\) for the nonstar letters and their commuting weights:
\[
        x_\star=1,\qquad
        x_{x_i}=x_i,\qquad
        x_{y_j}=y_j.
\]
The alphabet is ordered by
\[
        \star\prec x_1\prec\cdots\prec x_m
        \prec y_1\prec\cdots\prec y_s.
\]
For \(a,b\in \mathcal A_{m|s}\), define the super comparison
\(a\triangleleft b\) by
\[
        a\triangleleft b
        \quad\Longleftrightarrow\quad
        a\prec b
        \ \text{or}\
        a=b\ \text{and}\ |a|=\bar 1.
\]
Thus equality is treated as an ascent for odd letters and as a descent for
even letters.  We write \(a\ntriangleleft b\) for the negation of
\(a\triangleleft b\).

For a word
\[
        w=a_1a_2\cdots a_N
        \in
        \left(\mathcal A_{m|s}\right)^N,
\]
set
\[
        x^w=x_{a_1}x_{a_2}\cdots x_{a_N}.
\]
If \(\mu=(\mu_1,\ldots,\mu_\ell)\) is a composition of \(N\), put
\[
        B(\mu)
        =
        \left\{
        \mu_1,\,
        \mu_1+\mu_2,\,
        \ldots,\,
        \mu_1+\cdots+\mu_\ell=N
        \right\}.
\]
Thus \(B(\mu)\) is the set of right endpoints of the \(\mu\)-blocks.

\subsection{A super word expansion}

We assign a scalar weight to each word so that the sum of its weighted
monomials gives \(\mathsf q_\mu(x|y;q)\).  This scalar uses adjacent
comparisons within each \(\mu\)-block.  For \(1\le j<N\) with
\(j\notin B(\mu)\), define
\[
\varphi_\mu(j,w)
=
\begin{cases}
q,
    & \text{if } a_j\triangleleft a_{j+1},\\[3pt]
0,
    & \text{if } a_j\ntriangleleft a_{j+1},\
      a_{j+1}\triangleleft a_{j+2},\
      \text{and } j+1\notin B(\mu),\\[3pt]
-1,
    & \text{otherwise.}
\end{cases}
\]
If \(j+2>N\), the middle case is ignored.  Define
\[
        \wt_\mu(w)
        =
        \prod_{\substack{1\le j<N\\ j\notin B(\mu)}}
        \varphi_\mu(j,w).
\]
The empty product is understood to be \(1\).

Let
\[
        e_1\prec\cdots\prec e_r
        \qquad\text{and}\qquad
        o_1\prec\cdots\prec o_t
\]
be possibly empty lists of distinct even and odd letters in
\(\mathcal A_{m|s}\), respectively.  Let
\(\alpha=(\alpha_1,\ldots,\alpha_r)\) and
\(\beta=(\beta_1,\ldots,\beta_t)\) be compositions, and put
\[
        N=|\alpha|+|\beta|.
\]
Assume \(N>0\).
Let \(\mathcal W\) be the set of words in which \(e_i\)
occurs \(\alpha_i\) times and \(o_j\) occurs \(\beta_j\) times.

\begin{lemma}
\label{lem:one-block-super-hook-word-sum}
For the prescribed letters and multiplicities, the sum of the word weights is
\[
\sum_{w\in\mathcal W}
\wt_{(N)}(w)
=
        (-1)^{N-1}
        (-q)^{|\beta|-\ell(\beta)}
        (1-q)^{r+t-1}.
\]
\end{lemma}

\begin{proof}
We first determine which words have nonzero weight.  For
\(w=a_1\cdots a_N\in\mathcal W\), write \(A\) for a comparison
\(a_j\triangleleft a_{j+1}\) and \(D\) for its negation.
The definition of \(\wt_{(N)}\) assigns zero precisely when a \(D\)
is immediately followed by an \(A\).  Thus a word has nonzero weight
if and only if its comparison string has the form
\[
        A^hD^{N-1-h}\qquad(0\le h\le N-1).
\]
Split such a word after its \((h+1)\)-st letter.  The left part is
weakly increasing, with repetitions allowed only for odd letters;
the right part is weakly decreasing, with repetitions allowed only
for even letters.  The comparison between the parts, when the right
part is nonempty, is a descent.

List all selected letters as
\[
        z_1\prec z_2\prec\cdots\prec z_L,
        \qquad L=r+t,
\]
and let \(m_i\) be the multiplicity of \(z_i\).  The largest letter
\(z_L\) occurs at the end of the left part.  An odd letter can occur
at most once in the right part, so at least \(m_i-1\) copies of
an odd \(z_i\) must be on the left.  Similarly, an even letter can
occur at most once in the left part, so at least \(m_i-1\) copies
of an even \(z_i\) must be on the right.  If \(z_L\) is odd, it
cannot also occur in the right part, since the comparison between
the parts would then be an ascent.  If \(z_L\) is even, exactly
one of its copies is in the left part.

Only the placement of one copy of each \(z_i\), for \(i<L\),
remains to be chosen.  For \(I\subseteq\{1,\ldots,L-1\}\),
construct \(w_I\) by placing one copy of \(z_L\) on the left,
and one copy of \(z_i\), for \(i<L\), on the left if \(i\in I\)
and on the right otherwise.  Place all remaining copies of odd
letters on the left, and all remaining copies of even letters on
the right.  Write the left part in weakly increasing order and the
right part in weakly decreasing order, and concatenate them.
The total number of copies forced to the left by odd repetitions is
\[
        f=\sum_{\substack{1\le i\le L\\ |z_i|=\bar 1}}(m_i-1)
        =|\beta|-\ell(\beta).
\]

Every comparison within the left part of \(w_I\) is an ascent;
the comparison between the parts and every comparison within the
right part are descents.  Hence \(w_I\) has nonzero weight.
Conversely, the split determined by the comparison string of any
word with nonzero weight recovers \(I\) uniquely.  For \(i<L\)
with \(z_i\) even, we have \(i\in I\) exactly when \(z_i\) occurs
on the left.  For \(i<L\) with \(z_i\) odd, we have \(i\in I\)
exactly when \(z_i\) does not occur on the right.  The forced copies
and the ordering within the two parts then recover the entire word.  Thus the words in \(\mathcal W\)
with nonzero weight are precisely the \(w_I\).

The left part has length \(1+f+|I|\), so the number of ascents
in \(w_I\) is \(f+|I|\).  Each ascent contributes \(q\) and each
descent contributes \(-1\).  Therefore
\[
        \wt_{(N)}(w_I)
        =
        q^{f+|I|}(-1)^{N-1-f-|I|}
        =
        (-1)^{N-1}(-q)^{f+|I|}.
\]
Summing over all subsets \(I\subseteq\{1,\ldots,L-1\}\) gives
\[
\begin{aligned}
\sum_{w\in\mathcal W}\wt_{(N)}(w)
&=(-1)^{N-1}(-q)^f
  \sum_{I\subseteq\{1,\ldots,L-1\}}(-q)^{|I|}\\
&=(-1)^{N-1}
  (-q)^{|\beta|-\ell(\beta)}(1-q)^{r+t-1}.
\end{aligned}
\]
This proves the lemma.
\end{proof}

\begin{proposition}
\label{prop:super-mirabolic-word-expansion}
Let \(\mu\) be a composition of \(N\).
The function \(\mathsf q_\mu\) is the weighted sum over words
\begin{equation}
\label{eq:super-mirabolic-word-expansion}
        \mathsf q_\mu(x|y;q)
        =
        \sum_{w\in(\mathcal A_{m|s})^N}
        \wt_\mu(w)x^w.
\end{equation}
\end{proposition}

\begin{proof}
For \(N=0\), both sides equal \(1\).  Assume \(N\ge1\), and first
take \(\mu=(N)\).  Partition the words according to the distinct even
letters used, the distinct odd letters used, and their multiplicities.  If
these multiplicities are recorded by a bicomposition
\((\alpha;\beta)\vDash N\), then summing over the possible choices of the even
and odd letters gives the monomial factor
\[
        M_\alpha(1,x_1,\ldots,x_m)M_\beta(y).
\]
For each fixed choice of letters, Lemma~\ref{lem:one-block-super-hook-word-sum}
gives the scalar
\[
        (-1)^{N-1}
        (-q)^{|\beta|-\ell(\beta)}
        (1-q)^{\ell(\alpha)+\ell(\beta)-1}.
\]
Therefore
\[
\begin{aligned}
        \sum_{w\in(\mathcal A_{m|s})^N}
        \wt_{(N)}(w)x^w
        &=
        (-1)^{N-1}
        \sum_{(\alpha;\beta)\vDash N}
        M_\alpha(1,x_1,\ldots,x_m)M_\beta(y)       \\
        &\qquad\qquad\qquad \times
        (-q)^{|\beta|-\ell(\beta)}
        (1-q)^{\ell(\alpha)+\ell(\beta)-1}.
\end{aligned}
\]
By \eqref{eq:super-mirabolic-bicomposition-expansion}, the last
expression is \(\mathsf q_N(x|y;q)\).

For general \(\mu=(\mu_1,\ldots,\mu_\ell)\), write
\(w=w^{(1)}\cdots w^{(\ell)}\), with
\(|w^{(i)}|=\mu_i\).  No comparison crosses a block boundary, so
\[
 \wt_\mu(w)x^w
 =\prod_{i=1}^{\ell}\wt_{(\mu_i)}(w^{(i)})x^{w^{(i)}}.
\]
Summing independently over the blocks gives
\[
 \sum_w\wt_\mu(w)x^w
 =\prod_{i=1}^{\ell}
   \left(\sum_{w^{(i)}\in\mathcal A_{m|s}^{\mu_i}}
          \wt_{(\mu_i)}(w^{(i)})x^{w^{(i)}}\right)
 =\prod_{i=1}^{\ell}\mathsf q_{\mu_i}(x|y;q)
 =\mathsf q_\mu(x|y;q).
\]
\end{proof}

\subsection{The super mirabolic RSK correspondence}
\label{subsec:super-mirabolic-rsk}

We now define the column form of the super mirabolic RSK correspondence needed
for the word expansion in Proposition~\ref{prop:super-mirabolic-word-expansion}.
The nonstar letters are processed by the signed column insertion of
La Scala--Nardozza--Senato
\cite[Definitions~2.2 and~2.6 and Proposition~2.8]{LNS06}, whose parity
\(0\) and parity \(1\) letters are our even and odd letters, respectively.
The positions of \(\star\) are recorded separately.

For a partition \(\lambda\) with \(|\lambda|\le N\), let
\(\mathrm{HT}_{m|s}(\lambda)\) be the set of hook semistandard tableaux
of shape \(\lambda\) on the nonstar letters of \(\mathcal A_{m|s}\).
Their entries weakly increase along rows and columns, with even letters
strict down columns and odd letters strict along rows.
For \(P\in\mathrm{HT}_{m|s}(\lambda)\), let \(x^P\) be the product of
the weights of its entries, with the empty product equal to \(1\).
Let \(\mathrm{RT}_N(\lambda)\) be the set of pairs
\(Q=(Q^{(0)},Q^{(1)})\) such that \(Q^{(1)}\) is a standard tableau
of shape \(\lambda\) with entries an arbitrary subset of
\(\{1,\ldots,N\}\), and \(Q^{(0)}\) lists the complementary entries
in one increasing row.  We call such a pair a super mirabolic
recording tableau.

For a word \(w=a_1a_2\cdots a_N\in\mathcal A_{m|s}^{N}\), let
\[
        S(w)=\{j\in\{1,\ldots,N\}\mid a_j\ne\star\}
\]
be its set of nonstar positions.

\begin{definition}[RSK correspondence]
\label{def:super-mirabolic-rsk}
Start with three empty tableaux \(P,Q^{(0)},Q^{(1)}\) and read
\(w=a_1a_2\cdots a_N\) from left to right.  At position \(j\),
proceed as follows.
\begin{enumerate}[label=(\roman*)]
\item If \(a_j=\star\), append \(j\) to the right end of
\(Q^{(0)}\), leaving \(P\) and \(Q^{(1)}\) unchanged.
\item If \(a_j\ne\star\), insert \(a_j\) into the first column of
\(P\).  To insert a letter \(b\) into a column, scan its entries
from top to bottom and select an entry according to the following rule:
\begin{enumerate}[label=(\alph*)]
\item if \(b\) is even, select the first entry \(c\) with \(c\succeq b\);
\item if \(b\) is odd, select the first entry \(c\) with \(c\succ b\).
\end{enumerate}
If such an entry \(c\) exists, replace it by \(b\) and insert \(c\)
into the next column, now using the rule determined by the parity
of \(c\).  If no such entry exists, append \(b\) at the bottom of
the column and stop; an empty column is treated in the same way.
Place \(j\) in a new box of \(Q^{(1)}\) at the position of the
new box of \(P\).
\end{enumerate}
At the end, \(Q^{(1)}\) has entry set \(S(w)\), and \(Q^{(0)}\)
lists the complementary positions in increasing order.
Set \(Q=(Q^{(0)},Q^{(1)})\).  The RSK correspondence is the map
\[
\begin{aligned}
        \mathrm{RSK}:\mathcal A_{m|s}^{N}
        &\longrightarrow
        \coprod_{\substack{\lambda\text{ a partition}\\|\lambda|\le N}}
        \mathrm{HT}_{m|s}(\lambda)\times\mathrm{RT}_N(\lambda),\\
        w&\longmapsto(P,Q).
\end{aligned}
\]
\end{definition}

\begin{example}
\label{ex:super-mirabolic-rsk-insertion}
Take \(m=s=2\) and the word
\[
 w=x_2\,y_1\,y_1\,y_2\,x_1\,\star\,x_1\,y_1\,x_2.
\]
Here \(x_1,x_2\) are even and \(y_1,y_2\) are odd, with
\(\star\prec x_1\prec x_2\prec y_1\prec y_2\).
Each triple below lists \((P,Q^{(0)},Q^{(1)})\) after the indicated
step.  Each arrow is labelled by the letter being read.
Follow the arrows from left to right on the first row, from right
to left on the second row, and then down to the final triple.
\begingroup
\small
\setlength{\arraycolsep}{1.5pt}
\renewcommand{\arraystretch}{1.05}
\begin{center}
\begin{tikzpicture}[x=1pt,y=1pt,>=stealth,
  every node/.style={inner sep=0pt,outer sep=0pt}]
\node[anchor=west] (rsk0) at (4.00000,0) {$\left(\varnothing,\;\varnothing,\;\varnothing\right)$};
\node[anchor=west] (rsk1) at (86.96467,0) {$\left(\begin{array}{|c|}\hline
x_2\\\hline
\end{array},\;\varnothing,\;\begin{array}{|c|}\hline
1\\\hline
\end{array}\right)$};
\node[anchor=west] (rsk2) at (178.50011,0) {$\left(\begin{array}{|c|}\hline
x_2\\\hline
y_1\\\hline
\end{array},\;\varnothing,\;\begin{array}{|c|}\hline
1\\\hline
2\\\hline
\end{array}\right)$};
\node[anchor=west] (rsk3) at (276.11888,0) {$\left(\begin{array}{|c|}\hline
x_2\\\hline
y_1\\\hline
y_1\\\hline
\end{array},\;\varnothing,\;\begin{array}{|c|}\hline
1\\\hline
2\\\hline
3\\\hline
\end{array}\right)$};
\node[anchor=center] (rsk4) at (410.56883,0) {$\left(\begin{array}{|c|}\hline
x_2\\\hline
y_1\\\hline
y_1\\\hline
y_2\\\hline
\end{array},\;\varnothing,\;\begin{array}{|c|}\hline
1\\\hline
2\\\hline
3\\\hline
4\\\hline
\end{array}\right)$};
\node[anchor=north] (rsk5) at ([yshift=-24pt]rsk4.south) {$\left(\begin{array}{|c|c|}\hline
x_1&x_2\\\hline
y_1&\multicolumn{1}{c}{}\\\cline{1-1}
y_1&\multicolumn{1}{c}{}\\\cline{1-1}
y_2&\multicolumn{1}{c}{}\\\cline{1-1}
\end{array},\;\varnothing,\;\begin{array}{|c|c|}\hline
1&5\\\hline
2&\multicolumn{1}{c}{}\\\cline{1-1}
3&\multicolumn{1}{c}{}\\\cline{1-1}
4&\multicolumn{1}{c}{}\\\cline{1-1}
\end{array}\right)$};
\node[anchor=east] (rsk6) at ([xshift=-13.25pt]rsk5.west) {$\left(\begin{array}{|c|c|}\hline
x_1&x_2\\\hline
y_1&\multicolumn{1}{c}{}\\\cline{1-1}
y_1&\multicolumn{1}{c}{}\\\cline{1-1}
y_2&\multicolumn{1}{c}{}\\\cline{1-1}
\end{array},\;\begin{array}{|c|}\hline
6\\\hline
\end{array},\;\begin{array}{|c|c|}\hline
1&5\\\hline
2&\multicolumn{1}{c}{}\\\cline{1-1}
3&\multicolumn{1}{c}{}\\\cline{1-1}
4&\multicolumn{1}{c}{}\\\cline{1-1}
\end{array}\right)$};
\node[anchor=east] (rsk7) at ([xshift=-13.25pt]rsk6.west) {$\left(\begin{array}{|c|c|c|}\hline
x_1&x_1&x_2\\\hline
y_1&\multicolumn{2}{c}{}\\\cline{1-1}
y_1&\multicolumn{2}{c}{}\\\cline{1-1}
y_2&\multicolumn{2}{c}{}\\\cline{1-1}
\end{array},\;\begin{array}{|c|}\hline
6\\\hline
\end{array},\;\begin{array}{|c|c|c|}\hline
1&5&7\\\hline
2&\multicolumn{2}{c}{}\\\cline{1-1}
3&\multicolumn{2}{c}{}\\\cline{1-1}
4&\multicolumn{2}{c}{}\\\cline{1-1}
\end{array}\right)$};
\node[anchor=east] (rsk8) at ([xshift=-13.25pt]rsk7.west) {$\left(\begin{array}{|c|c|c|}\hline
x_1&x_1&x_2\\\hline
y_1&y_2&\multicolumn{1}{c}{}\\\cline{1-2}
y_1&\multicolumn{2}{c}{}\\\cline{1-1}
y_1&\multicolumn{2}{c}{}\\\cline{1-1}
\end{array},\;\begin{array}{|c|}\hline
6\\\hline
\end{array},\;\begin{array}{|c|c|c|}\hline
1&5&7\\\hline
2&8&\multicolumn{1}{c}{}\\\cline{1-2}
3&\multicolumn{2}{c}{}\\\cline{1-1}
4&\multicolumn{2}{c}{}\\\cline{1-1}
\end{array}\right)$};
\node[anchor=north] (rsk9) at ([yshift=-24pt]rsk8.south) {$\left(\begin{array}{|c|c|c|}\hline
x_1&x_1&x_2\\\hline
x_2&y_1&y_2\\\hline
y_1&\multicolumn{2}{c}{}\\\cline{1-1}
y_1&\multicolumn{2}{c}{}\\\cline{1-1}
\end{array},\;\begin{array}{|c|}\hline
6\\\hline
\end{array},\;\begin{array}{|c|c|c|}\hline
1&5&7\\\hline
2&8&9\\\hline
3&\multicolumn{2}{c}{}\\\cline{1-1}
4&\multicolumn{2}{c}{}\\\cline{1-1}
\end{array}\right)$};
\draw[->] (rsk0.east) -- node[above,yshift=1.5pt,font=\scriptsize] {\(x_2\)} (rsk1.west);
\draw[->] (rsk1.east) -- node[above,yshift=1.5pt,font=\scriptsize] {\(y_1\)} (rsk2.west);
\draw[->] (rsk2.east) -- node[above,yshift=1.5pt,font=\scriptsize] {\(y_1\)} (rsk3.west);
\draw[->] (rsk3.east) -- node[above,yshift=1.5pt,font=\scriptsize] {\(y_2\)} (rsk4.west);
\draw[->] (rsk4.south) -- node[midway,right,xshift=3pt,font=\scriptsize] {\(x_1\)} (rsk5.north);
\draw[->] (rsk5.west) -- node[above,yshift=1.5pt,font=\scriptsize] {\(\star\)} (rsk6.east);
\draw[->] (rsk6.west) -- node[above,yshift=1.5pt,font=\scriptsize] {\(x_1\)} (rsk7.east);
\draw[->] (rsk7.west) -- node[above,yshift=1.5pt,font=\scriptsize] {\(y_1\)} (rsk8.east);
\draw[->] (rsk8.south) -- node[midway,right,xshift=3pt,font=\scriptsize] {\(x_2\)} (rsk9.north);
\end{tikzpicture}
\end{center}
\endgroup

At step \(3\), the odd letter \(y_1\) passes the equal \(y_1\)
and is appended at the bottom of the first column.
At step \(6\), \(\star\) contributes only the entry \(6\) to \(Q^{(0)}\).
At step \(7\), the even letter \(x_1\) bumps the equal \(x_1\)
from the first column; that \(x_1\) bumps \(x_2\) from the second
column, and \(x_2\) creates the third column.
At step \(8\), the odd \(y_1\) passes both equal entries in the
first column, bumps \(y_2\), and \(y_2\) is appended to the second
column.  At step \(9\), \(x_2\) bumps the odd \(y_1\), which in
turn bumps \(y_2\) using the odd-letter rule in the second column;
\(y_2\) is then appended to the third column.
In each nonstar step, \(Q^{(1)}\) receives the current position
label in the new box, while its earlier entries stay in place.

The last triple gives
\(\mathrm{RSK}(w)=(P,(Q^{(0)},Q^{(1)}))\), with insertion shape
\((3,3,1,1)\) and
\(x^P=x^w=x_1^2x_2^2y_1^3y_2\).
\end{example}

We now return to general \(m,s,N\).

\begin{theorem}
\label{thm:super-mirabolic-rsk-bijection}
The map \(\mathrm{RSK}\) in Definition~\ref{def:super-mirabolic-rsk}
is a weight-preserving bijection:
\[
        \mathrm{RSK}(w)=(P,Q)\quad\Longrightarrow\quad x^w=x^P.
\]
\end{theorem}

\begin{proof}
First we check that the output belongs to the stated target.
By \cite[Definition~2.6 and Proposition~2.8]{LNS06}, each nonstar
insertion in Definition~\ref{def:super-mirabolic-rsk} preserves the
hook semistandard conditions on \(P\) and adds one box to its Young
diagram.  The same box is added to \(Q^{(1)}\), with a label larger
than all labels already present.  In particular, the boxes to its
left or above have smaller labels, so the rows and columns of
\(Q^{(1)}\) are strictly increasing.  Thus \(P\) and \(Q^{(1)}\)
have the same shape \(\lambda\), with \(|\lambda|=|S(w)|\le N\),
and \(Q^{(1)}\) has entry set \(S(w)\).
The remaining positions appear in increasing order in \(Q^{(0)}\).
Consequently,
\[
 P\in\mathrm{HT}_{m|s}(\lambda),\qquad
 Q=(Q^{(0)},Q^{(1)})\in\mathrm{RT}_N(\lambda),
\]
which proves that the map is well defined.

Next we construct an inverse starting from an arbitrary pair
\[
 (P,Q)\in\mathrm{HT}_{m|s}(\lambda)\times\mathrm{RT}_N(\lambda).
\]
To apply the inverse algorithm in \cite[Definition~4.4]{LNS06},
regard the entries of \(Q^{(1)}\) as an ordered alphabet and assign
odd parity to each of these position labels.  This makes
\cite[Definition~4.3]{LNS06} select column insertion at every step,
while the comparison used in each column is still determined by
the parity of the inserted letter, as in \cite[Definition~2.6]{LNS06}.
Since the entries of \(Q^{(1)}\) are distinct, it is also a super
semistandard tableau on this alphabet.  Thus the pair
\((P,Q^{(1)})\) satisfies the hypotheses of
\cite[Theorem~4.2]{LNS06}, whose forward and inverse procedures
apply to every such pair.

Recover the letters \(a_j\) by processing \(j=N,N-1,\ldots,1\):
\begin{enumerate}[label=(\roman*)]
\item If \(j\in Q^{(0)}\), set \(a_j=\star\) and delete \(j\)
from \(Q^{(0)}\).
\item If \(j\in Q^{(1)}\), it is the largest remaining label of
that standard tableau, so its box is an outer corner.  Delete this
box from \(Q^{(1)}\) and apply reverse column insertion to the
corresponding box of \(P\), recovering \(a_j\).
\end{enumerate}
By \cite[Definition~4.4 and Theorem~4.2]{LNS06}, each reverse
insertion is defined, leaves a hook semistandard tableau of the
smaller shape, and uniquely undoes the corresponding forward
insertion.  Hence the procedure produces a unique word
\(w=a_1\cdots a_N\).  Applying Definition~\ref{def:super-mirabolic-rsk}
to this word returns the prescribed pair \((P,Q)\), proving
surjectivity.  Conversely, applying the reverse procedure to the
output of any word recovers that word, proving injectivity.
If \(Q^{(1)}\) is empty, the procedure gives the all-star word;
when \(N=0\), it gives the empty word.

Finally, insertion only rearranges the nonstar letters and does
not change their multiplicities.  Since each star has weight \(1\),
we obtain \(x^w=x^P\).
\end{proof}

Let \(w\in\mathcal A_{m|s}^{N}\), with \(\mathrm{RSK}(w)=(P,Q)\).
To express \(\wt_\mu(w)\) using only the recording tableau \(Q\),
we will recover each adjacent comparison from the positions of its two
recording labels.  We introduce the following relation on
\(\mathrm{RT}_N(\lambda)\).  Let
\[
        Q=(Q^{(0)},Q^{(1)}).
\]
For entries \(i,j\in\{1,\ldots,N\}\), we write
\[
        i\mathrel{\swarrow} j
\]
if one of the following holds.

\begin{enumerate}[label=(\roman*)]
\item \(i,j\in Q^{(1)}\), and \(i\) lies weakly southwest of \(j\) in the
ordinary diagram of \(Q^{(1)}\): the row index of \(i\) is at least that of
\(j\), and its column index is at most that of \(j\).

\item \(i\in Q^{(1)}\) and \(j\in Q^{(0)}\).

\item \(i,j\in Q^{(0)}\), and \(i\) appears to the left of \(j\) in the row
\(Q^{(0)}\).
\end{enumerate}
We write
\[
        i\mathrel{\nearrow} j
\]
for the complementary relation.  Thus exactly one of
\[
        i\mathrel{\swarrow} j,
        \qquad
        i\mathrel{\nearrow} j
\]
holds.

Let \(T\) be a hook semistandard tableau, and let
\[
        b,c\in
        \mathcal A_{m|s}\setminus\{\star\}.
\]
Insert \(b\) into \(T\), and then insert \(c\) into the resulting tableau, by
the column insertion in Definition~\ref{def:super-mirabolic-rsk}.
Let \(\square_b\) and
\(\square_c\) be the two new boxes created by these insertions.

\begin{lemma}
\label{lem:column-super-local-position}
The new boxes satisfy
\[
        b\triangleleft c
        \quad\Longleftrightarrow\quad
        \square_c
        \text{ lies weakly southwest of }
        \square_b.
\]
\end{lemma}

\begin{proof}
By the column-bumping lemma \cite[Proposition~2.10]{LNS06},
\(b\triangleleft c\) holds if and only if \(\square_c\) lies in a
strictly lower row than \(\square_b\); in this case it also lies weakly
to its left.  Two successive new boxes in the same row occur from left
to right, so \(\square_c\) cannot lie weakly southwest of \(\square_b\)
in that case.  This gives the stated equivalence.
\end{proof}

Let
\[
        \mathrm{RSK}(w)=(P,Q),
        \qquad
        w=a_1a_2\cdots a_N.
\]

\begin{proposition}
\label{prop:super-relative-position}
For \(1\le j<N\), the adjacent letters and their recording positions satisfy
\begin{equation}
\label{eq:super-relative-position-sw}
        a_j\triangleleft a_{j+1}
        \quad\Longleftrightarrow\quad
        j+1\mathrel{\swarrow}j.
\end{equation}
\end{proposition}

\begin{proof}
If \(a_j=a_{j+1}=\star\), then
\(a_j\ntriangleleft a_{j+1}\), because \(\star\) is even.  Both labels lie in
\(Q^{(0)}\), and \(j+1\) appears to the right of \(j\).  Hence
\(j+1\mathrel{\nearrow}j\).

If \(a_j=\star\) and \(a_{j+1}\ne\star\), then
\(a_j\triangleleft a_{j+1}\).  In this case \(j\in Q^{(0)}\) and
\(j+1\in Q^{(1)}\), so by definition
\(j+1\mathrel{\swarrow}j\).

If \(a_j\ne\star\) and \(a_{j+1}=\star\), then
\(a_j\ntriangleleft a_{j+1}\).  Here \(j\in Q^{(1)}\) and
\(j+1\in Q^{(0)}\), so
\(j+1\mathrel{\nearrow}j\).

Finally, suppose that neither \(a_j\) nor \(a_{j+1}\) is \(\star\).
Consider the insertion tableau just before the \(j\)-th letter is
processed.  The next two steps insert \(a_j\) and then
\(a_{j+1}\) into this tableau, with no intervening insertion.
Let \(\square_j\) and \(\square_{j+1}\) be the two new boxes.
By Lemma~\ref{lem:column-super-local-position},
\[
        a_j\triangleleft a_{j+1}
        \quad\Longleftrightarrow\quad
        \square_{j+1}
        \text{ lies weakly southwest of }
        \square_j.
\]
The algorithm puts \(j\) and \(j+1\) in these positions in
\(Q^{(1)}\).  Later insertions add recording boxes without moving
any existing recording entry, so their relative positions in the
final \(Q^{(1)}\) are unchanged.  This gives
\eqref{eq:super-relative-position-sw}.
\end{proof}

Since the two relative-position relations are complementary, equivalently
\begin{equation}
\label{eq:super-relative-position-ne}
        a_j\ntriangleleft a_{j+1}
        \quad\Longleftrightarrow\quad
        j+1\mathrel{\nearrow}j.
\end{equation}

\begin{remark}
Zhao's super RSK \cite[Section~5]{Z24} uses parity-dependent row/column
insertion and reads \(Q\) with the input parities.  Our signed column insertion
encodes the parity comparison in the relative positions of \(Q\), while
\(Q^{(0)}\) records the star positions.  Its length
\(|Q^{(0)}|\) is the number of stars, so \(N-|Q^{(0)}|\) gives
the mirabolic rank.
\end{remark}

\subsection{Recording-tableau weights and character values}

Fix a composition \(\mu\) of \(N\), and let
\[
        Q\in \mathrm{RT}_N(\lambda).
\]
For \(1\le j<N\) with \(j\notin B(\mu)\), define
\[
\varphi_\mu(j,Q)
=
\begin{cases}
q,
    & \text{if } j+1\mathrel{\swarrow}j,\\[3pt]
0,
    & \text{if } j+1\mathrel{\nearrow}j,\
      j+2\mathrel{\swarrow}j+1,\
      \text{and } j+1\notin B(\mu),\\[3pt]
-1,
    & \text{otherwise.}
\end{cases}
\]
If \(j+2>N\), the middle case is ignored.  Define
\begin{equation}
\label{eq:def-super-roichman-weight}
        \wt_\mu(Q)
        =
        \prod_{\substack{1\le j<N\\ j\notin B(\mu)}}
        \varphi_\mu(j,Q).
\end{equation}

If
\[
        \mathrm{RSK}(w)=(P,Q),
\]
Proposition~\ref{prop:super-relative-position} identifies each adjacent
comparison in \(w\) with the corresponding relative position in \(Q\).
The two comparisons in the zero case agree as well.  Thus
\[
 \varphi_\mu(j,w)=\varphi_\mu(j,Q)
 \qquad(1\le j<N,\ j\notin B(\mu)).
\]
Multiplying these equalities gives
\begin{equation}
\label{eq:word-weight-equals-recording-weight}
        \wt_\mu(w)
        =
        \wt_\mu(Q).
\end{equation}

\begin{theorem}
\label{thm:super-mirabolic-schur-expansion}
Let \(\mu\) be a composition of \(N\).  The hook Schur expansion is
\begin{equation}
\label{eq:super-mirabolic-schur-expansion}
        \mathsf q_\mu(x|y;q)
        =
        \sum_{\substack{\lambda\in H(m,s)\\|\lambda|\le N}}
        \left(
        \sum_{Q\in\mathrm{RT}_N(\lambda)}
        \wt_\mu(Q)
        \right)
        s_\lambda(x|y).
\end{equation}
\end{theorem}

\begin{proof}
Use the bijection \(\mathrm{RSK}\) from
Theorem~\ref{thm:super-mirabolic-rsk-bijection} to reindex the sum in
\eqref{eq:super-mirabolic-word-expansion}.
For each fixed shape \(\lambda\), the bijection allows \(P\) and
\(Q\) to range independently over
\(\mathrm{HT}_{m|s}(\lambda)\) and \(\mathrm{RT}_N(\lambda)\).
By \eqref{eq:word-weight-equals-recording-weight}, the scalar
\(\wt_\mu(w)\) depends only on \(Q\), whereas content preservation
gives \(x^w=x^P\).  We can therefore separate the two finite sums:
\[
\mathsf q_\mu(x|y;q)=
\sum_{\substack{\lambda\text{ a partition}\\|\lambda|\le N}}
\left(\sum_{Q\in\mathrm{RT}_N(\lambda)}\wt_\mu(Q)\right)
\left(\sum_{P\in\mathrm{HT}_{m|s}(\lambda)}x^P\right).
\]
The bound \(|\lambda|\le N\) holds because each nonstar letter
creates one insertion box.  For each \(\lambda\), the sum over
\(P\) is \(s_\lambda(x|y)\) by
\cite[Definition~6.3 and Theorem~6.10]{BR87}.  It vanishes outside
\(H(m,s)\), proving the formula.
\end{proof}

For the character formula, let \(\mu\) be a composition of \(N\le n\), and
let \(\lambda\vdash a\) with \(0\le a\le N\).

\begin{theorem}[Super mirabolic Roichman formula]
\label{thm:super-mirabolic-roichman}
The character value is
\begin{equation}
\label{eq:super-mirabolic-roichman}
        \chi^{(n)}_{\lambda}
        \left(
        \widehat T^{(n)}_\mu
        \right)
        =
        \sum_{Q\in\mathrm{RT}_N(\lambda)}
        \wt_\mu(Q).
\end{equation}
\end{theorem}

\begin{proof}
Take \(s=0\) and \(m\ge\max\{1,N\}\), so every partition of size at
most \(N\) lies in \(H(m,0)\).  Theorems
\ref{thm:super-mirabolic-frobenius} and
\ref{thm:super-mirabolic-schur-expansion} then give
\[
 \sum_{\substack{\nu\text{ a partition}\\|\nu|\le N}}
 \left(
   \chi^{(n)}_\nu(\widehat T^{(n)}_\mu)
   -\sum_{Q\in\mathrm{RT}_N(\nu)}\wt_\mu(Q)
 \right)s_\nu(x_1,\ldots,x_m)=0.
\]
These Schur polynomials are linearly independent, so each coefficient
vanishes, including the one indexed by \(\lambda\).
The set \(\mathrm{RT}_N(\lambda)\) and the weights \(\wt_\mu(Q)\) depend only on
the recording labels and their relative positions, hence the resulting
formula is independent of the auxiliary alphabet.
\end{proof}

\section{Parameter inversion}
\label{sec:q-rook-normalization}

We first use \(\Theta_q\) from Proposition~\ref{prop:parameter-inversion-isomorphism} to
transfer the duality, annihilator, and character formulas from the mirabolic Hecke algebra to the
\(q\)-rook algebra.  We then give a representation-theoretic interpretation
of the hook and two-row character sums for the \(q\)-rook algebra.
Finally, we transport the formulas for \(q\)-rook character tables and
generating functions to the mirabolic Hecke algebra.

\subsection{Transport of irreducible characters}

Recall the irreducible \(R_n(q)\)-modules \(M^{R,n}_\lambda\),
for \(\lambda\in\mathscr P_{\le n}\), defined in
Subsection~\ref{subsec:irreducible-characters}.
Write \(\Upsilon^{(n)}_\lambda\) for the character of \(M^{R,n}_\lambda\).
Fix an integer \(r\ge1\), and set \(x=(x_1,\ldots,x_r)\) and
\(z=(1,x_1,\ldots,x_r)\).  For \(d\ge1\), define \(q_d(z;q)\) by
\[
 q_d(z;q)=\sum_{\alpha\vDash d}
 q^{d-\ell(\alpha)}(q-1)^{\ell(\alpha)-1}M_\alpha(z).
\]
For a composition \(\mu=(\mu_1,\ldots,\mu_\ell)\), set
\(q_\mu(z;q)=\prod_{i=1}^{\ell}q_{\mu_i}(z;q)\), with
\(q_\varnothing(z;q)=1\).

We first compare \(q_\mu(z;q)\) with
\(\mathsf q_\mu(x|\varnothing;q^{-1})\).  Recall that
\(d(\mu)=|\mu|-\ell(\mu)\).  These functions satisfy
\begin{equation}
\label{eq:ordinary-hl-parameter-inversion}
        q_\mu(1,x_1,\ldots,x_r;q)
        =
        (-q)^{d(\mu)}
        \mathsf q_\mu(x|\varnothing;q^{-1}).
\end{equation}
Here \(\mathsf q_\mu(x|\varnothing;q^{-1})\) is taken with
\(m=r\) and \(s=0\).
Indeed, \eqref{eq:super-mirabolic-bicomposition-expansion} gives
\[
\begin{aligned}
 (-q)^{d-1}\mathsf q_d(x|\varnothing;q^{-1})
 &=q^{d-1}\sum_{\alpha\vDash d}
       (1-q^{-1})^{\ell(\alpha)-1}M_\alpha(z)\\
 &=\sum_{\alpha\vDash d}
       q^{d-\ell(\alpha)}(q-1)^{\ell(\alpha)-1}M_\alpha(z)
 =q_d(z;q).
\end{aligned}
\]
Multiplication over the parts of \(\mu\) proves
\eqref{eq:ordinary-hl-parameter-inversion}; for the empty composition
both sides are \(1\).
We next compare \(\Upsilon^{(n)}_\lambda(T^R_\mu;q)\) with
\(\chi^{(n)}_\lambda(\widehat T^{(n)}_\mu;q^{-1})\) using \(\Theta_q\).

For \(\lambda\vdash a\), with \(0\le a\le n\), define an
\(H^{\mathrm{mir}}_n(q^{-1})\)-action on the vector space
\(M^{R,n}_\lambda\) by
\[
 h\cdot w:=\Theta_q(h)w,
 \qquad h\in H^{\mathrm{mir}}_n(q^{-1}),\quad w\in M^{R,n}_\lambda,
\]
where the right-hand side uses the original \(R_n(q)\)-action.
The notation \(N^{(n)}_{(\lambda,a)}\big|_{q\mapsto q^{-1}}\) denotes
the \(H^{\mathrm{mir}}_n(q^{-1})\)-module obtained by replacing \(q\)
with \(q^{-1}\) in every entry of the matrices of \(T_i\) and \(P_j\)
acting on \(N^{(n)}_{(\lambda,a)}\), in the basis from
Subsection~\ref{subsec:irreducible-characters}.

By the definition in Subsection~\ref{subsec:irreducible-characters},
this action makes \(M^{R,n}_\lambda\) isomorphic to
\(N^{(n)}_{(\lambda,a)}\big|_{q\mapsto q^{-1}}\): replacing \(q\)
by \(q^{-1}\) in that definition restores the original action matrices
of \(M^{R,n}_\lambda\) and changes \(\Theta_{q^{-1}}\) to \(\Theta_q\).

\begin{lemma}
\label{lem:ordinary-frobenius-comparison}
For every composition \(\mu\) of size at most \(n\) and every
\(\lambda\vdash a\le n\), the character values are related by
\begin{equation}
\label{eq:q-rook-mirabolic-character-comparison}
        \Upsilon^{(n)}_\lambda(T^R_\mu;q)
        =
        (-q)^{d(\mu)}
        \chi^{(n)}_{\lambda}
        \left(
        \widehat T^{(n)}_\mu
        \right)
        \bigg|_{q\mapsto q^{-1}}.
\end{equation}
\end{lemma}

\begin{proof}
With the action \(h\cdot w=\Theta_q(h)w\) just defined,
Lemma~\ref{lem:standard-elements-under-theta} gives
\[
 \Theta_q\!\left(
       \widehat T^{(n)}_\mu\big|_{q\mapsto q^{-1}}\right)
 =(-q^{-1})^{d(\mu)}T^R_\mu.
\]
Taking traces in \(M^{R,n}_\lambda\) and multiplying by
\((-q)^{d(\mu)}\) proves
\eqref{eq:q-rook-mirabolic-character-comparison}.
\end{proof}

\begin{corollary}
\label{cor:type-a-parameter-inversion}
For \(\lambda,\mu\in\mathscr P_k\), the type \(A\) character values satisfy
\begin{equation}
\label{eq:type-a-parameter-inversion}
        \chi^{A,\lambda'}(T_{\gamma_\mu};q)
        =
        (-q)^{|\mu|-\ell(\mu)}
        \chi^{A,\lambda}(T_{\gamma_\mu};q^{-1}).
\end{equation}
\end{corollary}

\begin{proof}
For \(k=0\), both sides are \(1\).
For \(k\ge1\), Lemma~\ref{lem:hecke-specht-labels} identifies
\(S^{\lambda'}_{k,\mathrm{std}}\) with the pullback of
\(\widetilde S^\lambda_{k,\mathrm{std}}\) along
\(\phi\circ\psi:H_k(q)\to H_k(q^{-1})\), which sends
\(T_i\) to \(-qu_i\).  Here \(u_i\) are the standard generators
at parameter \(q^{-1}\).  The standard product
\(T_{\gamma_\mu}\) has \(|\mu|-\ell(\mu)\) generators.
Its action on the pullback is therefore \((-q)^{|\mu|-\ell(\mu)}\)
times the action of \(T_{\gamma_\mu}\) at parameter \(q^{-1}\).
Taking traces proves \eqref{eq:type-a-parameter-inversion}.
\end{proof}

\subsection{Transport of tensor actions and annihilators}

We transport \(\Phi_{m|s,n}\) from
Proposition~\ref{prop:super-mirabolic-representation}, at parameter
\(q^{-1}\), through
\(\Theta_q:H^{\mathrm{mir}}_n(q^{-1})\to R_n(q)\).
In the fixed tensor basis, write \(R_i(q^{-1})\) for the operator obtained
from \(R_i\) by replacing \(q\) with \(q^{-1}\), and put
\[
        \mathcal T_i^R=qR_i(q^{-1}),\qquad
        \mathcal P_j^R=\mathcal P_j.
\]
These operators have coefficients in \(\mathbb F\) and act on
\(W_{\mathbb F,n}=\widetilde V_{\mathbb F}^{\otimes_{\mathbb F}n}\).

\begin{proposition}
\label{prop:super-q-rook-representation}
The assignments \(\mathsf T_i\mapsto\mathcal T_i^R\) and
\(\mathsf P_j\mapsto\mathcal P_j^R\) extend to a representation of
\(R_n(q)\) on \(W_{\mathbb F,n}\).
\end{proposition}

\begin{proof}
The inverse of \(\Theta_q\) sends \(\mathsf T_i\) to \(-qT_i\) and
\(\mathsf P_j\) to \(P_j\).  The representation \(\Phi_{m|s,n}\) from
Proposition~\ref{prop:super-mirabolic-representation}, at parameter
\(q^{-1}\), sends these elements to \(qR_i(q^{-1})\) and
\(\mathcal P_j\), respectively.  Composing the two homomorphisms proves
the assertion.
\end{proof}

Write \(R_{n,\mathbb K}(q)=\mathbb K\otimes_{\mathbb F}R_n(q)\).
Denote this representation and its scalar extension by
\[
 \Phi^R_{m|s,n}:R_n(q)\longrightarrow
       \operatorname{End}_{\mathbb F}(W_{\mathbb F,n}),\qquad
 \Phi^R_{m|s,n,\mathbb K}:R_{n,\mathbb K}(q)\longrightarrow
       \operatorname{End}_{\mathbb K}(\widetilde V^{\otimes n}).
\]
The representation identity is
\begin{equation}
\label{eq:q-rook-representation-transport}
 \Phi^R_{m|s,n}
 =\Phi_{m|s,n}\big|_{q\mapsto q^{-1}}\circ\Theta_q^{-1}.
\end{equation}

Let \(\sigma_{\mathbb K}\) be the \(\mathbb C\)-automorphism of \(\mathbb K\) with
\(\sigma_{\mathbb K}(v)=v^{-1}\), and let
\(\sigma_{\mathbb F}\) be its restriction to \(\mathbb F\), so that
\(\sigma_{\mathbb F}(q)=q^{-1}\).  Applying \(\sigma_{\mathbb K}\)
to the matrices of the standard action \(\widetilde\rho_n\) gives the
standard action \(\widetilde\rho_n^R\) of
\(U_v(\mathfrak{gl}(m|s))\).  Since
\(\sigma_{\mathbb K}(\widetilde D_n)=\widetilde D_n^{-1}\), we
define \(\widetilde\Psi_n^R\) on \(\widetilde V^{\otimes n}\) by
\begin{equation}
\label{eq:q-rook-quantum-action}
 \widetilde\Psi_n^R(u)
 =\widetilde D_n^{-1}\widetilde\rho_n^R(u)\widetilde D_n,
 \qquad u\in U_v(\mathfrak{gl}(m|s)).
\end{equation}
If \(S_i(v)\) denotes \(S_i\) in \eqref{eq:R-conjugation-standard}
with its quantum parameter replaced by \(v\), then \(\mathcal T_i^R\)
satisfies
\begin{equation}
\label{eq:q-rook-standard-operator-conjugation}
 \mathcal T_i^R
 =\widetilde D_n^{-1}\bigl(vS_i(v)\bigr)\widetilde D_n.
\end{equation}
It follows by applying \(\sigma_{\mathbb K}\) to
\eqref{eq:R-conjugation-standard} and multiplying by \(q=v^2\).

For \(\lambda\in H(m,s)\), write \(L^R_{m|s}(\lambda)\) for the module
obtained from \(L_{m|s}(\lambda)\) by applying \(\sigma_{\mathbb K}\)
entrywise to the matrices of the quantum generators.  It is a covariant
\(U_v(\mathfrak{gl}(m|s))\)-module with character
\(s_\lambda(x|y)\).  Also set
\(M^{R,n}_{\lambda,\mathbb K}
=\mathbb K\otimes_{\mathbb F}M^{R,n}_\lambda\).

\begin{theorem}[Super \(q\)-rook Schur--Weyl duality]
\label{thm:super-q-rook-schur-weyl}
The algebras \(\Phi^R_{m|s,n,\mathbb K}(R_{n,\mathbb K}(q))\) and
\(\widetilde\Psi_n^R(U_v(\mathfrak{gl}(m|s)))\) are mutual full
centralizers on \(\widetilde V^{\otimes n}\).
The tensor space \(\widetilde V^{\otimes n}\) has the bimodule decomposition
\begin{equation}
\label{eq:q-rook-bimodule-decomposition}
 \widetilde V^{\otimes n}
 \cong
 \bigoplus_{a=0}^{n}
 \bigoplus_{\substack{\lambda\vdash a\\\lambda\in H(m,s)}}
 L^R_{m|s}(\lambda)\otimes_{\mathbb K}M^{R,n}_{\lambda,\mathbb K}.
\end{equation}
\end{theorem}

\begin{proof}
We first prove the centralizer equalities by applying \(\sigma_{\mathbb K}\)
to Theorem~\ref{thm:super-mirabolic-schur-weyl}.  For matrices \(X,Y\)
over \(\mathbb K\), applying the automorphism entrywise gives
\[
 [\sigma_{\mathbb K}(X),\sigma_{\mathbb K}(Y)]
 =\sigma_{\mathbb K}([X,Y]).
\]
Since the automorphism is invertible, \(X\) commutes with \(Y\)
if and only if \(\sigma_{\mathbb K}(X)\) and
\(\sigma_{\mathbb K}(Y)\) commute.  Hence applying
\(\sigma_{\mathbb K}\) entrywise preserves the centralizer equalities.  The action \(\widetilde\Psi_n\) becomes
\(\widetilde\Psi_n^R\) in \eqref{eq:q-rook-quantum-action}.
Composing the parameter-inverted action \(\Phi_{m|s,n,\mathbb K}\)
with \(\Theta_q^{-1}\) gives \(\Phi^R_{m|s,n,\mathbb K}\) by
\eqref{eq:q-rook-representation-transport}.  This proves both
centralizer equalities.

We next obtain \eqref{eq:q-rook-bimodule-decomposition} from the bimodule
decomposition in Theorem~\ref{thm:super-mirabolic-schur-weyl}.
Applying \(\sigma_{\mathbb K}\) gives the quantum factors
\(L^R_{m|s}(\lambda)\).  By the definition in
Subsection~\ref{subsec:irreducible-characters}, composing the actions on the
parameter-inverted factors \(N^{(n)}_{(\lambda,a),\mathbb K}\)
with \(\Theta_q^{-1}\) gives \(M^{R,n}_{\lambda,\mathbb K}\).
The condition \(\lambda_{m+1}\le s\) and the characters
\(s_\lambda(x|y)\) are preserved, since the
field automorphism changes coefficients and fixes the weight variables.
\end{proof}

\begin{remark}
Theorem~\ref{thm:super-q-rook-schur-weyl} extends the ordinary
\(q\)-rook Schur--Weyl duality of \cite[Corollary~4.3]{Hal04} to the super setting.
\end{remark}

Let
\[
 I^R_{m|s,n}(q)=\ker\Phi^R_{m|s,n}\subseteq R_n(q).
\]
The coefficient substitution convention from
Section~\ref{sec:tensor-space-kernels} agrees with
\(\sigma_{\mathbb F}\).  Taking kernels in
\eqref{eq:q-rook-representation-transport} gives
\begin{equation}
\label{eq:q-rook-annihilator-transport}
 I^R_{m|s,n}(q)
 =\Theta_q\bigl(I^{\mathrm{mir}}_{m|s,n}(q^{-1})\bigr).
\end{equation}
Indeed, for \(h\in R_n(q)\), \eqref{eq:q-rook-representation-transport}
gives \(\Phi^R_{m|s,n}(h)=0\) exactly when
\(\Theta_q^{-1}(h)\in I^{\mathrm{mir}}_{m|s,n}(q^{-1})\).
Equivalently, at the inverse parameter,
\begin{equation}
\label{eq:mirabolic-annihilator-transport}
 I^{\mathrm{mir}}_{m|s,n}(q)
 =\Theta_{q^{-1}}^{-1}\bigl(I^R_{m|s,n}(q^{-1})\bigr).
\end{equation}
Recall \(\rho=((s+1)^{m+1})\), \(r_0=(m+1)(s+1)\), and
\(f_T,f_\rho\) from \eqref{eq:rectangle-jm-primitive-idempotent}
and \eqref{eq:rectangle-jm-central-idempotent}.
For \(n\ge r_0\), recall from
Subsection~\ref{subsec:rectangular-jm-idempotents} that
\(\iota_{r_0,n}:R_{r_0}(q)\hookrightarrow R_n(q)\) is the embedding
\(\mathsf T_i\mapsto\mathsf T_i\) for \(1\le i<r_0\) and
\(\mathsf P_j\mapsto\mathsf P_j\) for \(1\le j\le r_0\).

\begin{corollary}
\label{cor:super-q-rook-annihilator}
If \(n<r_0\), then \(I^R_{m|s,n}(q)=0\).
If \(n\ge r_0\), then, for every \(T\in\operatorname{SYT}(\rho)\),
\begin{equation}
\label{eq:q-rook-annihilator-principal}
 I^R_{m|s,n}(q)=R_n(q)\,\iota_{r_0,n}(f_T)\,R_n(q).
\end{equation}
\end{corollary}

\begin{proof}
Apply \eqref{eq:q-rook-annihilator-transport} to
Theorem~\ref{thm:super-mirabolic-annihilator} at parameter \(q^{-1}\).
Inverting the parameter in \eqref{eq:mirabolic-rectangle-idempotent}
gives \(\Theta_q(g_T(q^{-1}))=f_T\).  Since \(\Theta_q\)
commutes with the generator-preserving inclusions,
\[
 \Theta_q\bigl(\jmath_{r_0,n}(g_T(q^{-1}))\bigr)
 =\iota_{r_0,n}(f_T).
\]
An algebra isomorphism takes the ideal generated by an element to the
ideal generated by its image.  This proves
\eqref{eq:q-rook-annihilator-principal}; for \(n<r_0\),
\eqref{eq:q-rook-annihilator-transport} gives \(I^R_{m|s,n}(q)=0\).
\end{proof}

Using \(f_\rho\) in place of \(f_T\) gives in the same way
\[
 I^R_{m|s,n}(q)=R_n(q)\,\iota_{r_0,n}(f_\rho)\,R_n(q)
 \qquad(n\ge r_0).
\]
Parameter inversion preserves dimensions.  Hence \(\Phi^R_{m|s,n}\)
is faithful exactly when \(n<(m+1)(s+1)\), and
\eqref{eq:mirabolic-kernel-dimension} gives
\[
 \dim_{\mathbb F}I^R_{m|s,n}(q)
 =\sum_{\substack{\lambda\in\mathscr P_{\le n}\\\lambda_{m+1}>s}}
       \left(\binom n{|\lambda|}f^\lambda\right)^2.
\]
Equivalently, the decomposition \eqref{eq:q-rook-bimodule-decomposition}
gives the following expression for \(\ker\Phi^R_{m|s,n,\mathbb K}\):
\begin{equation}
\label{eq:q-rook-kernel-wedderburn}
 \ker\Phi^R_{m|s,n,\mathbb K}
 =\bigoplus_{\substack{\lambda\in\mathscr P_{\le n}\\\lambda_{m+1}>s}}
       \operatorname{End}_{\mathbb K}(M^{R,n}_{\lambda,\mathbb K}).
\end{equation}

\begin{remark}
Corollary~\ref{cor:super-q-rook-annihilator} extends Xiao's pure-even
annihilator theorem \cite[Theorem~1.2]{X17} to tensor superspace:
for \(n\ge r_0\), the annihilator is generated by
\(\iota_{r_0,n}(f_T)\), where \(T\in\operatorname{SYT}(\rho)\)
and \(\rho=((s+1)^{m+1})\).
When \(s=0\), this rectangle becomes the column \((1^{m+1})\),
and the condition \(\lambda_{m+1}>s\) in
\eqref{eq:q-rook-kernel-wedderburn} becomes \(\ell(\lambda)>m\).
Writing \(v\) for Xiao's parameter, with \(v^2=q\), we identify
his Hecke generators multiplied by \(v\) with \(\mathsf T_i\)
and leave the projection generators unchanged.
His quasi-idempotent \(\Phi_{m+1}\) in \cite[Section~3]{X17}
is then a nonzero scalar multiple of \(f_\rho\) over \(\mathbb K\); for
\(n\ge m+1\), their images under \(\iota_{m+1,n}\) therefore
generate the same annihilator ideal.
\end{remark}

Define \(\mathscr S_q^R(m|s;\le n)\) and
\(\overline R_{m|s,n}(q)\) by
\[
 \mathscr S_q^R(m|s;\le n)
 =\widetilde\Psi_n^R\bigl(U_v(\mathfrak{gl}(m|s))\bigr),\qquad
 \overline R_{m|s,n}(q)=R_n(q)/I^R_{m|s,n}(q).
\]
The subscript \(q=v^2\) records the parameter of the commuting
\(q\)-rook algebra.  Write
\(\overline R_{m|s,n,\mathbb K}(q)=\mathbb K\otimes_{\mathbb F}
\overline R_{m|s,n}(q)\), identified with
\(\Phi^R_{m|s,n,\mathbb K}(R_{n,\mathbb K}(q))\).

\begin{corollary}
\label{cor:q-rook-quotient-double-centralizer}
The algebras \(\overline R_{m|s,n,\mathbb K}(q)\) and
\(\mathscr S_q^R(m|s;\le n)\) are mutual full centralizers on
\(\widetilde V^{\otimes n}\).
\end{corollary}

\begin{proof}
Exactness of scalar extension and the first isomorphism theorem give
\[
 \overline R_{m|s,n,\mathbb K}(q)
 \cong\Phi^R_{m|s,n,\mathbb K}(R_{n,\mathbb K}(q)).
\]
The assertion now follows from Theorem~\ref{thm:super-q-rook-schur-weyl}.
\end{proof}

Equations~\eqref{eq:q-rook-annihilator-transport} and
\eqref{eq:mirabolic-annihilator-transport} induce isomorphisms
between \(\overline H^{\mathrm{mir}}_{m|s,n}(q^{-1})\) and
\(\overline R_{m|s,n}(q)\), and between
\(\overline H^{\mathrm{mir}}_{m|s,n}(q)\) and
\(\overline R_{m|s,n}(q^{-1})\).
The algebra \(\mathscr S_q^R(m|s;\le n)\) is likewise obtained by
applying \(\sigma_{\mathbb K}\) to the matrices in
\(\mathscr S_q(m|s;\le n)\).  The dimensions of
\(\overline R_{m|s,n}(q)\) and \(\mathscr S_q^R(m|s;\le n)\) are therefore
\begin{align*}
 \dim_{\mathbb F}\overline R_{m|s,n}(q)
 &=\dim_{\mathbb F}\overline H^{\mathrm{mir}}_{m|s,n}(q)
 =\sum_{a=0}^{n}\sum_{\substack{\lambda\vdash a\\\lambda\in H(m,s)}}
       \left(\binom na f^\lambda\right)^2,\\
 \dim_{\mathbb K}\mathscr S_q^R(m|s;\le n)
 &=\dim_{\mathbb K}\mathscr S_q(m|s;\le n)
 =\sum_{a=0}^{n}\sum_{\substack{\lambda\vdash a\\\lambda\in H(m,s)}}
       s_\lambda(1^m|1^s)^2,
\end{align*}
by \eqref{eq:mirabolic-quotient-dimension} and
\eqref{eq:mirabolic-schur-dimension}.

\subsection{Super \texorpdfstring{\(q\)}{q}-rook character formulas}

We now translate Theorem~\ref{thm:super-mirabolic-frobenius} into the normalization
used in \cite{DHP03}.  For \(d\ge1\), define
\(\mathsf q^R_d(x|y;q)\) by
\[
        \mathsf q^R_d(x|y;q)
        =
        (-q)^{d-1}\,
        \mathsf q_d(x|y;q^{-1})
        \qquad (d\ge1),
\]
and, for a composition \(\mu=(\mu_1,\ldots,\mu_\ell)\), set
\[
        \mathsf q^R_\mu(x|y;q)
        =
        \prod_{i=1}^{\ell}
        \mathsf q^R_{\mu_i}(x|y;q).
\]
For the empty composition, we set
\[
        \mathsf q^R_\varnothing(x|y;q)=1.
\]
Equivalently, since \(d(\mu)=|\mu|-\ell(\mu)\),
\[
        \mathsf q^R_\mu(x|y;q)
        =
        (-q)^{d(\mu)}
        \mathsf q_\mu(x|y;q^{-1}).
\]

Recall \(z=(1,x_1,\ldots,x_m)\), and let \(w\) be a formal variable.

\begin{lemma}
\label{lem:q-rook-super-hl-explicit}
For \(d\ge1\), \(\mathsf q^R_d(x|y;q)\) has the expansion
\begin{equation}
\label{eq:q-rook-super-hl-explicit}
\begin{aligned}
        \mathsf q^R_d(x|y;q)
        =
        \sum_{(\alpha;\beta)\vDash d}
        &M_\alpha(z)M_\beta(y)
        (-1)^{|\beta|-\ell(\beta)}
        q^{|\alpha|-\ell(\alpha)}
        (q-1)^{\ell(\alpha)+\ell(\beta)-1}.
\end{aligned}
\end{equation}
\end{lemma}

\begin{proof}
Replace \(q\) by \(q^{-1}\) in
\eqref{eq:super-mirabolic-bicomposition-expansion} and multiply by
\((-q)^{d-1}\).  For a fixed bicomposition \((\alpha;\beta)\),
the coefficient of \(M_\alpha(z)M_\beta(y)\) is
\[
\begin{aligned}
 &(-q)^{d-1}(-1)^{d-1}(-q^{-1})^{|\beta|-\ell(\beta)}
       (1-q^{-1})^{\ell(\alpha)+\ell(\beta)-1}\\
 &\qquad=(-1)^{|\beta|-\ell(\beta)}q^{|\alpha|-\ell(\alpha)}
       (q-1)^{\ell(\alpha)+\ell(\beta)-1},
\end{aligned}
\]
where the exponent of \(q\) simplifies using
\(d=|\alpha|+|\beta|\).  This proves \eqref{eq:q-rook-super-hl-explicit}.
For one even letter \(a\) and one odd letter \(b\), respectively,
summing the possible positive run lengths gives
\[
\begin{aligned}
 1+(q-1)\sum_{r\ge1}q^{r-1}(aw)^r
 &=\frac{1-aw}{1-qaw},\\
 1+(q-1)\sum_{r\ge1}(-1)^{r-1}(bw)^r
 &=\frac{1+qbw}{1+bw}.
\end{aligned}
\]
Multiplying \((1-aw)/(1-qaw)\) over
\(a\in\{1,x_1,\ldots,x_m\}\) and \((1+qbw)/(1+bw)\) over
\(b\in\{y_1,\ldots,y_s\}\) gives
\(1+(q-1)\sum_{d\ge1}\mathsf q^R_d(x|y;q)w^d\).
\end{proof}

The functions \(\mathsf q^R_d(x|y;q)\) also have the generating series
\begin{equation}
\label{eq:q-rook-super-hl-generating-series}
        1+(q-1)\sum_{d\ge1}\mathsf q^R_d(x|y;q)w^d
        =
        \prod_{a\in\{1,x_1,\ldots,x_m\}}
        \frac{1-aw}{1-qaw}
        \prod_{b\in\{y_1,\ldots,y_s\}}
        \frac{1+qbw}{1+bw}.
\end{equation}

\begin{theorem}[Super Frobenius formula for the \(q\)-rook algebra]
\label{thm:super-q-rook-frobenius}
Let \(\mu\) be a composition of \(k\le n\).
The hook Schur expansion of \(\mathsf q^R_\mu(x|y;q)\) is
\begin{equation}
\label{eq:super-q-rook-frobenius}
        \mathsf q^R_\mu(x|y;q)
        =
        \sum_{a=0}^{k}
        \sum_{\substack{\lambda\vdash a\\ \lambda_{m+1}\le s}}
        \Upsilon^{(n)}_\lambda(T^R_\mu)
        s_\lambda(x|y).
\end{equation}
\end{theorem}

\begin{proof}
Replace \(q\) by \(q^{-1}\) in
Theorem~\ref{thm:super-mirabolic-frobenius} and multiply by
\((-q)^{d(\mu)}\).  The left-hand side becomes
\(\mathsf q^R_\mu(x|y;q)\) by definition.  The Schur functions
are independent of \(q\), and
\eqref{eq:q-rook-mirabolic-character-comparison} identifies each
coefficient on the right with
\(\Upsilon^{(n)}_\lambda(T^R_\mu)\).
The summation range is unchanged, giving \eqref{eq:super-q-rook-frobenius}.
\end{proof}

\begin{remark}
Theorem~\ref{thm:super-q-rook-frobenius} extends the ordinary
\(q\)-rook Frobenius formula of \cite[Theorem~2.3]{DHP03} to tensor superspace.
For \(s=0\), \(\mathsf q^R_d(x|\varnothing;q)=q_d(1,x_1,\ldots,x_m;q)\),
and \eqref{eq:super-q-rook-frobenius} recovers that formula.

For \(\mu\vdash n\), the full-rank part \(|\lambda|=n\) of
\eqref{eq:super-q-rook-frobenius} recovers Mitsuhashi's type \(A\)
super Frobenius formula \cite[Theorem~5.5]{Mit10}.
\end{remark}

For a composition \(\mu\) of \(N\le n\), a partition
\(\lambda\) with \(|\lambda|\le N\), and
\(Q\in\mathrm{RT}_N(\lambda)\), define \(\wt^R_\mu(Q)\) by
\begin{equation}
\label{eq:q-rook-roichman-weight}
 \wt^R_\mu(Q)=(-q)^{d(\mu)}
                  \left.\wt_\mu(Q)\right|_{q\mapsto q^{-1}}.
\end{equation}
The product in \eqref{eq:def-super-roichman-weight} contains
\(d(\mu)=N-\ell(\mu)\) factors \(\varphi_\mu(j,Q)\).
Each factor in \(\wt^R_\mu(Q)\) is therefore
\((-q)\left.\varphi_\mu(j,Q)\right|_{q\mapsto q^{-1}}\),
which takes the values \(-1,0,q\) when \(\varphi_\mu(j,Q)\)
is \(q,0,-1\), respectively.

\begin{corollary}[Roichman formula for the \(q\)-rook algebra]
\label{cor:q-rook-roichman}
For every \(\lambda\vdash a\) with \(0\le a\le N\),
\(\Upsilon^{(n)}_\lambda(T^R_\mu)\) is given by
\[
 \Upsilon^{(n)}_\lambda(T^R_\mu)
 =\sum_{Q\in\mathrm{RT}_N(\lambda)}\wt^R_\mu(Q).
\]
\end{corollary}

\begin{proof}
Apply \(q\mapsto q^{-1}\) to
Theorem~\ref{thm:super-mirabolic-roichman} and multiply by
\((-q)^{d(\mu)}\).  Equation~\eqref{eq:q-rook-mirabolic-character-comparison}
identifies the left-hand side with \(\Upsilon^{(n)}_\lambda(T^R_\mu)\),
and \eqref{eq:q-rook-roichman-weight} gives the right-hand side.
\end{proof}

\begin{remark}
For partitions \(\mu\), Corollary~\ref{cor:q-rook-roichman}
is the \(q\)-rook Roichman formula of \cite[Theorem~4.3]{DHP03}.
\end{remark}

We next specialize \eqref{eq:super-q-rook-frobenius} at \(x_i=y_j=1\).
Recall that \(c^{m|s}_\lambda=s_\lambda(1^m|1^s)\). For every integer \(r\ge1\), define
\begin{equation}
\label{eq:def-q-rook-regev-factor}
        F^R_r(m,s;q)
        =
        \sum_{(\alpha;\beta)\vDash r}
        \binom{m+1}{\ell(\alpha)}
        \binom{s}{\ell(\beta)}
        (-1)^{|\beta|-\ell(\beta)}
        q^{|\alpha|-\ell(\alpha)}
        (q-1)^{\ell(\alpha)+\ell(\beta)-1}.
\end{equation}
The generating series of \(F^R_r(m,s;q)\) is
\begin{equation}
\label{eq:q-rook-regev-generating-series}
        \sum_{r\ge1}F^R_r(m,s;q)w^r
        =
        \frac{1}{q-1}
        \left[
        \left(\frac{1-w}{1-qw}\right)^{m+1}
        \left(\frac{1+qw}{1+w}\right)^s
        -1
        \right].
\end{equation}

\begin{theorem}
\label{thm:super-q-rook-regev}
Let \(\mu=(\mu_1,\ldots,\mu_\ell)\vdash k\le n\).
The weighted sum of character values factors as
\begin{equation}
\label{eq:super-q-rook-regev}
        \sum_{a=0}^{k}
        \sum_{\substack{\lambda\vdash a\\ \lambda\in H(m,s)}}
        c^{m|s}_\lambda\,
        \Upsilon^{(n)}_\lambda(T^R_\mu)
        =
        \prod_{j=1}^{\ell}
        F^R_{\mu_j}(m,s;q).
\end{equation}
\end{theorem}

\begin{proof}
Specialize \eqref{eq:super-q-rook-frobenius} at \(x_i=y_j=1\).  Since
\[
        M_\alpha(1^{m+1})=\binom{m+1}{\ell(\alpha)},
        \qquad
        M_\beta(1^s)=\binom{s}{\ell(\beta)},
\]
\(\mathsf q^R_r(1^m|1^s;q)=F^R_r(m,s;q)\), proving
\eqref{eq:super-q-rook-regev}.  The same specialization of
\eqref{eq:q-rook-super-hl-generating-series} gives
\eqref{eq:q-rook-regev-generating-series}.
\end{proof}

For the specialization \(m=s=1\), recall the polynomial \(A_d(q)\) from
Corollary~\ref{cor:super-mirabolic-theorem-B-hook}, and define, for
\(d\ge1\),
\[
        A^R_d(q)
        =
        q^{d-1}A_d(q^{-1})
        =
        (2d+1)q^{d-1}
        +
        4\sum_{j=1}^{d-1}(d-j)(-1)^j q^{d-1-j}.
\]
In sums over hook
partitions we include the empty partition at rank zero.

\begin{corollary}
\label{cor:super-q-rook-regev-hook}
For \(\mu=(\mu_1,\ldots,\mu_\ell)\vdash k\le n\),
the sum of character values indexed by hook partitions is
\[
        \sum_{a=0}^{k}
        \sum_{\substack{\lambda\vdash a\\ \lambda\text{ is a hook}}}
        \Upsilon^{(n)}_\lambda(T^R_\mu)
        =
        \frac{1}{2}
        \left(
        \prod_{i=1}^{\ell(\mu)} A^R_{\mu_i}(q)
        +
        q^{k-\ell(\mu)}
        \right).
\]
\end{corollary}

\begin{proof}
As in the proof of Corollary~\ref{cor:super-mirabolic-theorem-B-hook},
\(c^{1|1}_\varnothing=1\) and \(c^{1|1}_\lambda=2\) for every nonempty
hook partition \(\lambda\).  Parameter inversion gives
\[
        F^R_d(1,1;q)
        =
        (-q)^{d-1}F_d(1,1;q^{-1})
        =
        q^{d-1}A_d(q^{-1})
        =
        A^R_d(q).
\]
On \(M^{R,n}_\varnothing\), each \(\mathsf P_j\) acts as \(1\) and each
\(\mathsf T_i\) as \(q\), hence
\[
        \Upsilon^{(n)}_{\varnothing}(T^R_\mu)
        =
        q^{k-\ell(\mu)}.
\]
Let \(\Sigma^R\) denote the sum on the left-hand side of
Corollary~\ref{cor:super-q-rook-regev-hook}.
Theorem~\ref{thm:super-q-rook-regev} now reads
\(2\Sigma^R-q^{k-\ell(\mu)}=\prod_iA^R_{\mu_i}(q)\).  Solving proves the formula.
\end{proof}

\begin{remark}
Corollary~\ref{cor:super-q-rook-regev-hook} recovers the formula in
\cite[Theorem~3.12, Equation~(3.35)]{JWL26}, which is derived there
using the vertex operator realization of Schur functions.
Restricting the tensor action to full rank also recovers Zhao's formula
for the sum of hook characters of the type \(A\) Hecke algebra
\cite[Theorem~4.6]{Z19}.
\end{remark}

\subsection{Representations for the hook and two-row character sums}
\label{subsec:jwl-representation-realizations}

By constructing \(R_n(q)\)-modules over \(\mathbb F\), we give a
representation-theoretic interpretation of the hook and two-row character sums in
\cite[Equations~(3.35) and~(3.36)]{JWL26}.
First take \(m=s=1\), and write
\(W_{\mathbb F,n}=W_{\bar0}\oplus W_{\bar1}\) for its even
and odd subspaces.  The \(R_n(q)\)-action is even, so both subspaces
are submodules.

\begin{proposition}
\label{cor:jwl-hook-representation}
The spaces \(W_{\bar1}\) and \(W_{\bar0}\) have the following decompositions
as ordinary \(R_n(q)\)-modules over \(\mathbb F\):
\begin{align*}
 W_{\bar1}
 &\cong\bigoplus_{a=1}^{n}\bigoplus_{r=0}^{a-1}
             M^{R,n}_{(a-r,1^r)},\\
 W_{\bar0}
 &\cong M^{R,n}_{\varnothing}
       \oplus\bigoplus_{a=1}^{n}\bigoplus_{r=0}^{a-1}
             M^{R,n}_{(a-r,1^r)}.
\end{align*}
\end{proposition}

\begin{proof}
The nonempty \((1,1)\)-hook partitions are \((a-r,1^r)\), with
\(a\ge1\) and \(0\le r<a\).  Let \(x\) and \(y\) be the
weights of the even and odd letters.  In a semistandard hook tableau,
every box below the first row contains \(y\).  The first row contains
only \(x\), or a possibly empty string of \(x\)'s followed by a
single \(y\).  These two tableaux give
\begin{equation}
\label{eq:one-one-hook-schur-parity}
 s_{(a-r,1^r)}(x|y)=x^{a-r}y^r+x^{a-r-1}y^{r+1}.
\end{equation}
The parity of a weight space is the exponent of \(y\) modulo \(2\).
Thus \(L^R_{1|1}(a-r,1^r)\) has one-dimensional even and odd
subspaces.  The space \(L^R_{1|1}(\varnothing)\) is even.  Taking parity
parts of \eqref{eq:q-rook-bimodule-decomposition} proves the displayed
decompositions after scalar extension to \(\mathbb K\).

By Proposition~\ref{prop:q-rook-split-decomposition}, the algebra
\(R_n(q)\) is split semisimple over \(\mathbb F\), and its simple modules \(M^{R,n}_\lambda\) remain irreducible over
\(\mathbb K\).  Decompose \(W_{\bar0}\) and \(W_{\bar1}\)
over \(\mathbb F\) and extend scalars.  The multiplicities of their
simple factors are unchanged by this extension, so the decompositions
over \(\mathbb K\) determine the same decompositions over
\(\mathbb F\).
\end{proof}

Taking traces on \(W_{\bar1}\) and using
Corollary~\ref{cor:super-q-rook-regev-hook} gives, for \(\mu\vdash n\),
\begin{align}
\label{eq:jwl-hook-representation-character}
 \operatorname{Tr}_{W_{\bar1}}(T^R_\mu)
 =\sum_{a=1}^{n}\sum_{r=0}^{a-1}
       \Upsilon^{(n)}_{(a-r,1^r)}(T^R_\mu)
 =\frac12\left(\prod_{i=1}^{\ell(\mu)}A^R_{\mu_i}(q)
                  -q^{n-\ell(\mu)}\right).
\end{align}
The polynomial denoted by \(A(d;q)\) in
\cite[Equation~(3.35)]{JWL26} is \((-1)^{d-1}A^R_d(q)\).
Thus \(\prod_{i=1}^{\ell(\mu)}A(\mu_i;q)
=(-1)^{n-\ell(\mu)}\prod_{i=1}^{\ell(\mu)}A^R_{\mu_i}(q)\),
which converts \eqref{eq:jwl-hook-representation-character} into the
normalization of \cite[Equation~(3.35)]{JWL26}.  Hence \(W_{\bar1}\)
gives a representation-theoretic interpretation of the sum in
\eqref{eq:jwl-hook-representation-character}.

For the character sum in \cite[Equation~(3.36)]{JWL26}, take \(m=2\) and \(s=0\).
We omit trailing
zero parts, so \((a,0)=(a)\), and \((0,0)\) denotes the empty
partition.

\begin{proposition}
\label{prop:jwl-two-row-representation}
The space \(W_{\mathbb F,n}\) has the following decomposition as an
\(R_n(q)\)-module over \(\mathbb F\):
\begin{equation}
\label{eq:jwl-two-row-module-decomposition}
 W_{\mathbb F,n}
 \cong\bigoplus_{a=0}^{n}\bigoplus_{r=0}^{\lfloor a/2\rfloor}
       \left(M^{R,n}_{(a-r,r)}\right)^{\oplus(a-2r+1)}.
\end{equation}
\end{proposition}

\begin{proof}
The \((2,0)\)-hook partitions have at most two rows and are
\(\lambda=(a-r,r)\), with \(0\le r\le\lfloor a/2\rfloor\).
In a semistandard tableau on the two even letters, each of the first
\(r\) columns is forced to contain \(1\) above \(2\).  The
remaining \(a-2r\) boxes lie in the first row and consist of a string
of \(1\)'s followed by a string of \(2\)'s.  There are
\(a-2r+1\) possibilities.  Therefore
\[
 \dim_{\mathbb K}L^R_{2|0}(a-r,r)
 =s_{(a-r,r)}(1,1)=a-2r+1.
\]
Equation~\eqref{eq:q-rook-bimodule-decomposition} gives
\eqref{eq:jwl-two-row-module-decomposition} after extension to \(\mathbb K\).  Split semisimplicity
over \(\mathbb F\) and preservation of simple multiplicities under
scalar extension descend it to \(\mathbb F\), as in the proof of
Proposition~\ref{cor:jwl-hook-representation}.
\end{proof}

For \(\mu\vdash n\), the trace of \(T^R_\mu\) on \(W_{\mathbb F,n}\) is
\begin{align}
\label{eq:jwl-two-row-representation-character}
 \operatorname{Tr}_{W_{\mathbb F,n}}(T^R_\mu)
 &=\sum_{a=0}^{n}\sum_{r=0}^{\lfloor a/2\rfloor}
       (a-2r+1)\,\Upsilon^{(n)}_{(a-r,r)}(T^R_\mu)\\*
 &=q^{n-\ell(\mu)}\prod_{i=1}^{\ell(\mu)}
       \left(3\mu_i-3q^{-1}(\mu_i-1)
       +\frac{(\mu_i-1)(\mu_i-2)(1-q^{-1})^2}{2}\right).\notag
\end{align}
Indeed, specialize \eqref{eq:def-q-rook-regev-factor} at
\((m,s)=(2,0)\).  For \(d\ge1\), only even compositions with
one, two or three parts contribute, and the number with \(j\) parts
is \(\binom{d-1}{j-1}\).  Hence
\[
 F^R_d(2,0;q)
 =3q^{d-1}+3(d-1)q^{d-2}(q-1)
       +\binom{d-1}{2}q^{d-3}(q-1)^2.
\]
Terms with more than \(d\) parts are zero, which includes the
cases \(d=1,2\).  After factoring out \(q^{d-1}\), the first
two terms become \(3d-3q^{-1}(d-1)\).
Theorem~\ref{thm:super-q-rook-regev} now gives
\eqref{eq:jwl-two-row-representation-character}, the character identity
in \cite[Equation~(3.36)]{JWL26}.
Thus \(W_{\mathbb F,n}\) gives a representation-theoretic interpretation
of the two-row character sum in \cite[Equation~(3.36)]{JWL26}.

\subsection{Character tables and generating functions}
\label{subsec:ordinary-mirabolic-consequences}

We use \eqref{eq:q-rook-mirabolic-character-comparison}, with \(q\)
replaced by \(q^{-1}\), to transport ordinary \(q\)-rook results of
\cite{DHP03} at parameter \(q^{-1}\) to ordinary mirabolic results
at parameter \(q\).  For \(\lambda,\mu\in
\mathscr P_{\le n}\), set
\[
        \Xi_n(q)
        =
        \left(
        \chi^{(n)}_{\lambda}
        \bigl(\widehat T^{(n)}_\mu;q\bigr)
        \right)_{\lambda,\mu},
        \qquad
        \Xi_n^R(q)
        =
        \left(
        \Upsilon^{(n)}_\lambda
        (T^R_\mu;q)
        \right)_{\lambda,\mu}.
\]
Rows and columns are ordered by rank,
\[
        \mathscr P_0,\mathscr P_1,\ldots,\mathscr P_n,
\]
where rank means the size of the indexing partition.
Replacing \(q\) by \(q^{-1}\) in
\eqref{eq:q-rook-mirabolic-character-comparison} gives
\begin{equation}
\label{eq:mirabolic-q-rook-character-transport}
        \chi^{(n)}_{\lambda}
        \bigl(\widehat T^{(n)}_\mu;q\bigr)
        =
        (-q)^{d(\mu)}
        \Upsilon^{(n)}_\lambda
        (T^R_\mu;q^{-1}),
        \qquad
        d(\mu)=|\mu|-\ell(\mu).
\end{equation}
Thus, with our convention that rows are labelled by irreducibles and columns by
standard elements,
\begin{equation}
\label{eq:mirabolic-q-rook-character-table-scaling}
        \Xi_n(q)
        =
        \Xi_n^R(q^{-1})
        \cdot
        \operatorname{diag}_{\mu\in\mathscr P_{\le n}}
        \left(
        (-q)^{|\mu|-\ell(\mu)}
        \right).
\end{equation}
Thus the \(\mu\)-column of \(\Xi_n^R(q^{-1})\) is multiplied by \((-q)^{|\mu|-\ell(\mu)}\).

We recall the convention for Hecke algebras of type \(A\) needed for the diagonal blocks of \(\Xi_n(q)\).  Let
\[
        \Xi^A_k(q)
        =
        \left(
        \chi^{A,\lambda}(T_{\gamma_\mu};q)
        \right)_{\mu,\lambda\in\mathscr P_k}
\]
be the usual character table of the Hecke algebra of type \(A\), with rows indexed by cycle type
\(\mu\) and columns by irreducible labels \(\lambda\), and set
\[
        \Xi^{A,\vee}_k(q)
        =
        \left(
        \chi^{A,\lambda'}(T_{\gamma_\mu};q)
        \right)_{\mu,\lambda\in\mathscr P_k}.
\]

Write \(\Xi_n[k,k]\) for the block with row and column labels in
\(\mathscr P_k\), and let a superscript \(\mathsf t\) denote matrix
transpose.

\begin{corollary}
\label{thm:mirabolic-block-upper-triangular}
\label{thm:diagonal-block-type-A}
The mirabolic character table \(\Xi_n(q)\) is block upper triangular with
respect to the rank ordering.  Its rank \(k\) diagonal block is independent of
\(n\) and is
\[
        \Xi_n[k,k]
        =
        \left(\Xi^{A,\vee}_k(q)\right)^{\mathsf t}.
\]
\end{corollary}

\begin{proof}
Dieng--Halverson--Poladian prove the corresponding block upper triangularity
for \(\Xi_n^R(q)\), and its rank \(k\) diagonal block is
\[
        \left(\Xi^A_k(q)\right)^{\mathsf t}
\]
in the normalization of
\cite[Corollary~2.5 and~(0.3)]{DHP03}.  The column scaling
\eqref{eq:mirabolic-q-rook-character-table-scaling} preserves block upper
triangularity.  For \(\lambda,\mu\in\mathscr P_k\), the \(\lambda,\mu\)-entry of \(\Xi_n[k,k]\) is
\[
        (-q)^{|\mu|-\ell(\mu)}
        \chi^{A,\lambda}(T_{\gamma_\mu};q^{-1}).
\]
Corollary~\ref{cor:type-a-parameter-inversion} identifies this entry with
\(\chi^{A,\lambda'}(T_{\gamma_\mu};q)\), proving the formula for
\(\Xi_n[k,k]\).
\end{proof}

Thus, for \(\lambda,\mu\in\mathscr P_k\), the diagonal-block identity is
\[
        \chi^{(n)}_{\lambda}
        \bigl(\widehat T^{(n)}_\mu;q\bigr)
        =
        \chi^{A,\lambda'}(T_{\gamma_\mu};q).
\]

\begin{remark}
Set \(\Xi_0(q)=(1)\).
Let \(C_n\) be the submatrix of \(\Xi_n(q)\) with rows of rank less than \(n\)
and columns of rank \(n\).  The recursion in \cite[Corollary~2.5]{DHP03} and
\eqref{eq:mirabolic-q-rook-character-table-scaling} give
\[
        \Xi_n(q)
        =
        \begin{pmatrix}
        \Xi_{n-1}(q) & C_n\\
        0 & \left(\Xi^{A,\vee}_n(q)\right)^{\mathsf t}
        \end{pmatrix}.
\]
For \(|\lambda|<n\) and \(|\mu|=n\), its entries are
\((C_n)_{\lambda,\mu}=(-q)^{|\mu|-\ell(\mu)}
\Upsilon^{(n)}_\lambda(T^R_\mu;q^{-1})\), by
\eqref{eq:mirabolic-q-rook-character-table-scaling}; they may also be
computed from Theorem~\ref{thm:super-mirabolic-roichman}.
\end{remark}

Fix a partition \(\lambda\) and an ordinary alphabet
\(Z=(z_1,z_2,\ldots)\).
Let \(A_1\) be the set of finite-support sequences of nonnegative integers and
write
\[
        |\alpha|=\sum_i\alpha_i,\qquad
        \ell(\alpha)=\#\{i\mid \alpha_i\ne0\},\qquad
        Z^\alpha=\prod_i z_i^{\alpha_i}.
\]
For a sequence \(\alpha\in A_1\), omit its zero entries when forming
\(\widehat T^{(n)}_\alpha\).  The values
\(\chi^{(n)}_\lambda(\widehat T^{(n)}_\alpha;q)\) are independent of the
ambient rank \(n\ge\max\{1,|\lambda|,|\alpha|\}\), by
\cite[Corollary~2.5]{DHP03} and
\eqref{eq:mirabolic-q-rook-character-transport}.  Put
\(\chi_{\lambda,\alpha}(q)=
\chi^{(n)}_{\lambda}(\widehat T^{(n)}_\alpha;q)\), with value zero when
\(|\lambda|>|\alpha|\), and define
\begin{equation}
\label{eq:def-mir-character-generating-function}
        G_\lambda(Z;q)
        =
        \sum_{\alpha\in A_1}
        (q-1)^{\ell(\alpha)}
        \chi_{\lambda,\alpha}(q)
        Z^\alpha.
\end{equation}
Similarly, in each coefficient choose
\(n\ge\max\{1,|\lambda|,|\alpha|\}\), and set
\[
G^R_\lambda(Z;q)
        =
        \sum_{\alpha\in A_1}
        (q-1)^{\ell(\alpha)}
        \Upsilon^{(n)}_\lambda(T^R_\alpha;q)
        Z^\alpha.
\]
The formula for \(G^R_\lambda(Z;q)\) proved in
\cite[Theorem~6.12]{JL26} is
\[
        G^R_\lambda(Z;q)
        =
        \sigma_Z[q-1]\,
        s_\lambda[(q-1)Z].
\]

\begin{corollary}
\label{thm:mirabolic-character-generating-function}
For every partition \(\lambda\), the series \(G_\lambda(Z;q)\) satisfies
\begin{equation}
\label{eq:mirabolic-character-generating-function}
        G_\lambda(Z;q)
        =
        \sigma_{-Z}[1-q]\,
        s_{\lambda'}[(q-1)Z].
\end{equation}
\end{corollary}

\begin{proof}
We first compare \(G_\lambda(Z;q)\) with \(G^R_\lambda(-qZ;q^{-1})\).
In \eqref{eq:mirabolic-q-rook-character-transport}, the factor
\((-q)^{|\alpha|-\ell(\alpha)}\) is exactly produced by replacing \(q\) by
\(q^{-1}\) in \(G^R_\lambda(Z;q)\) and then making the ordinary
variable substitution
\[
        z_i\longmapsto -qz_i.
\]
Indeed, the coefficient of \(Z^\alpha\) in \(G^R_\lambda(-qZ;q^{-1})\) carries
the scalar
\[
        (q^{-1}-1)^{\ell(\alpha)}(-q)^{|\alpha|}
        =
        (q-1)^{\ell(\alpha)}(-q)^{|\alpha|-\ell(\alpha)}.
\]
Together with
\[
        \chi_{\lambda,\alpha}(q)
        =
        (-q)^{|\alpha|-\ell(\alpha)}
        \Upsilon^{(n)}_\lambda(T^R_\alpha;q^{-1}),
\]
this gives
\[
        G_\lambda(Z;q)
        =
        G^R_\lambda(-qZ;q^{-1}).
\]
The variables are replaced by the ordinary scalar multiples \(-qz_i\).
On power sums,
\[
        p_r\bigl[(q^{-1}-1)(-qZ)\bigr]
        =
        (-1)^{r-1}(q^r-1)p_r[Z].
\]
The standard involution \(p_r\mapsto(-1)^{r-1}p_r\) sends
\(s_\lambda\) to \(s_{\lambda'}\), so
\[
        s_\lambda\bigl[(q^{-1}-1)(-qZ)\bigr]
        =
        s_{\lambda'}[(q-1)Z],
\]
and
\[
        \sigma_{-qZ}[q^{-1}-1]
        =
        \sigma_{-Z}[1-q].
\]
Substitution in the formula for \(G^R_\lambda(Z;q)\) from
\cite[Theorem~6.12]{JL26} gives
\eqref{eq:mirabolic-character-generating-function}; coefficient extraction
recovers every \(\chi_{\lambda,\alpha}(q)\).
\end{proof}

\section*{Acknowledgments}
We thank Jinkui Wan for sharing results from her earlier manuscript with us. N.J. is partially supported by the Simons Foundation (No. MP-TSM-00002518). 
N.L. is supported by Beijing Natural Science Foundation (No. 1264053).

\end{document}